\documentclass[english]{scrartcl}

\usepackage[a4paper,left=2.1cm, right=2.1cm, top=2.1cm, bottom=2.1cm]{geometry}
\usepackage{amsmath}
\usepackage{amssymb}
\usepackage{enumitem}
\usepackage[colorlinks = true, linkcolor = blue, citecolor = blue]{hyperref}
\usepackage{amsthm}
\usepackage{cite}
\usepackage[nameinlink]{cleveref}
\usepackage{tikz}
\usepackage{graphicx}
\usepackage{color}
\usepackage[caption=false]{subfig}

\newtheorem{theorem}{Theorem}[section]
\newtheorem{definition}[theorem]{Definition}
\newtheorem{remark}[theorem]{Remark}
\newtheorem{proposition}[theorem]{Proposition}
\newtheorem{lemma}[theorem]{Lemma}
\newtheorem{corollary}[theorem]{Corollary}
\newtheorem{example}[theorem]{Example}

\Crefname{theorem}{Theorem}{Theorems}
\Crefname{definition}{Definition}{Definitions}
\Crefname{remark}{Remark}{Remarks}
\Crefname{proposition}{Proposition}{Propositions}
\Crefname{lemma}{Lemma}{Lemmas}
\Crefname{corollary}{Corollary}{Corollaries}
\Crefname{example}{Example}{Examples}
\crefname{equation}{}{}
\crefname{figure}{Figure}{Figures}

\title{Stabilization of Control-Affine Systems by Local Approximations of Trajectories\footnote{A preliminary version \cite{Suttner2017} of this work has appeared in the Proceedings of the 20th IFAC World Congress, 2017.}
}
\author{
  Raik Suttner\thanks{Institute of Mathematics, University of Wuerzburg, Germany (\texttt{raik.suttner@mathematik.uni-wuerzburg.de}).}
}
\date{}

\begin{document}
\maketitle

\begin{abstract}
We study convergence and stability properties of control-affine systems. Our considerations are motivated by the problem of stabilizing a control-affine system by means of output feedback for states in which the output function attains an extreme value. Following a recently introduced approach to extremum seeking control, we obtain access to the ascent and descent directions of the output function by approximating Lie brackets of suitably defined vector fields. This approximation is based on convergence properties of trajectories of control-affine systems under certain sequences of open-loop controls. We show that a suitable notion of convergence for the sequences of open-loop controls ensures local uniform convergence of the corresponding sequences of flows. We also show how boundedness properties of the control vector fields influence the quality of the approximation. Our convergence results are sufficiently strong to derive conditions under which the proposed output feedback control law induces exponential stability. We also apply this control law to the problem of purely distance-based formation control for non\-holonomic multi-agent systems.
\end{abstract}
%
%
%
%
\section{Introduction}
Consider a driftless control-affine system
\begin{align}
\dot{x} & \ = \ \sum_{k=1}^{p}u_k\,g_k(x), \label{eq:intro:01} \allowdisplaybreaks \\
y & \ = \ \psi(x) \nonumber
\end{align}
on a smooth manifold~$M$ with a real-valued output~$y$. We assume that the control vector fields $g_1,\ldots,g_p$ are smooth, and that the output~$y$ is given by a smooth function~$\psi$ on~$M$. Suppose that we are interested in an output feedback control law for the real-valued input channels~$u_1,\ldots,u_p$ that stabilizes the system around states where~$\psi$ attains an extreme value. Without loss of generality, we restrict our attention to the minima of~$\psi$. In this case,~$\psi$ can be interpreted as a cost function that assigns every system state to a real number. Because we deal with a control-affine system, a natural attempt is to choose each input~$u_k$ in such a way that, at every point~$x$ of~$M$, the tangent vector~$u_kg_k(x)$ points into a descent direction of~$\psi$. The system is then constantly driven into a descent direction of~$\psi$. This can be accomplished by defining each input~$u_k$ at every point~$x$ of~$M$ as the negative directional derivative of~$\psi$ along~$g_k(x)$. In other words, the approach is to choose~$u_k$ as the negative Lie derivative of the function~$\psi$ along the vector field~$g_k$. We denote the Lie derivative of~$\psi$ along~$g_k$ at any point~$x$ of~$M$ by $(g_k\psi)(x)$. However, this control law is not output feedback because its implementation requires information about the Lie derivatives~$g_k\psi$ in any state~$x$ of the system. In our setup, both~$g_k\psi$ and~$x$ are treated as unknown quantities.

The problem of stabilizing~\cref{eq:intro:01} at minima of~$\psi$ by means of output feedback has been studied extensively in the literature on extremum seeking control. In its most general form, the purpose of extremum seeking control is not restricted to control-affine systems with real-valued outputs, but is aimed at optimizing the output of a general non\-linear control system without any information about the model or the current system state. For an overview of different attempts and strategies on obtaining extremum seeking control, the reader is referred to ~\cite{ZhangBook,LiuBook,ScheinkerBook}. A universal solution to this problem is not known.

The content of the present paper is motivated by an approach to extremum seeking control that was first introduced in~\cite{Duerr2010}, and then extended and improved in various subsequent works such as \cite{Duerr2013,Duerr2014,Duerr2015,Scheinker2013,Scheinker20132,Scheinker2014}. As explained above, the negative Lie derivatives~$-g_k\psi$ are promising candidates for the inputs~$u_k$ to steer~\cref{eq:intro:01} into a descent direction of~$\psi$, but they do not fall into the class of output feedback. Nonetheless, there are ways to obtain access to descent directions $-(g_k\psi)(x)g_k(x)$ of~$\psi$ by means of output feedback. For this purpose, we consider vector fields as derivations on the algebra of smooth functions. By doing so, one can define the Lie bracket of two smooth vector fields~$f_1$ and~$f_2$ as the vector field~$[f_1,f_2]$ that acts on any smooth function~$\alpha$ on~$M$ as the derivation $[f_1,f_2]\alpha=f_1(f_2\alpha)-f_2(f_1\alpha)$. It is now easy to check that the vector field~$-(g_k\psi)g_k$ is the same as the Lie bracket~$[\psi{g_k},g_k]$, where~$\psi{g_k}$ denotes the vector field $x\mapsto\psi(x)g_k(x)$. Note that this choice of Lie bracket, which is due to~\cite{Duerr2010}, is not the only possible way to obtain access to the vector field $-(g_k\psi)g_k$. Another option, which is introduced in~\cite{Scheinker2014}, is the Lie bracket $[(\sin\circ\psi)g_k,(\cos\circ\psi)g_k]$. In~\cite{Scheinker20132}, the non\-smooth approach $\frac{1}{1-2a}[\psi^{1-a}g_k,\psi^{a}g_k]$ with fixed $a\in(0,1/2)$ is studied. A systematic investigation of suitable Lie brackets can be found in~\cite{Grushkovskaya2017}. As a common feature, one can recognize that these Lie brackets are of the form $[(h_{\text{s}}\circ\psi)g_k,(h_{\text{c}}\circ\psi)g_k]$ with suitably chosen functions~$h_{\text{s}},h_{\text{c}}$.

At this point, we have merely rewritten the vector fields $-(g_k\psi)g_k$ in terms of Lie brackets. However, it is not yet clear how to relate these Lie brackets to~\cref{eq:intro:01} through output feedback, which is a crucial step in the procedure. The strategy is to find time-varying output feedback for~\cref{eq:intro:01} such that the trajectories behave at least approximately like the trajectories of~\cref{eq:intro:01} under the state feedback $u_k=-g_k\psi$. For this purpose, one can exploit known convergence properties of control-affine systems under sequences of open-loop controls, which are extensively studied in~\cite{Kurzweil1987,Kurzweil19882,Kurzweil1988,Sussmann1991,Liu19972,Liu1997}. In the following, we only indicate this convergence theory through the example of system~\cref{eq:intro:01}. The general formalism is explained in the main part of the paper. Suppose that we have chosen suitable functions $h_{\text{s}},h_{\text{c}}\colon\mathbb{R}\to\mathbb{R}$ such that the Lie brackets $[(h_{\text{s}}\circ\psi)g_k,(h_{\text{c}}\circ\psi)g_k]$ coincide with the vector fields $-(g_k\psi)g_k$ of interest. For $k=1,\ldots,m$ we denote the vector fields $(h_{\text{s}}\circ\psi)g_k$ and $(h_{\text{c}}\circ\psi)g_k$ by~$f_{2k-1}$ and~$f_{2k}$, respectively. This defines~$m=2p$ vector fields on~$M$. Our goal is to approximate the solutions of the differential equation
\[
\Sigma^{\infty}\colon\qquad\dot{x} \ = \ -\sum_{k=1}^{p}(g_k\psi)(x)\,g_k(x) \ = \ \sum_{k=1}^{p}[f_{2k-1},f_{2k}](x).
\]
For each $i=1,\ldots,m$, we choose a sequence $(u^j_i)_j$ of measurable and bounded functions $u^j_i\colon\mathbb{R}\to\mathbb{R}$. Then, for every sequence index~$j$, we consider the differential equation
\[
\Sigma^j\colon\qquad\dot{x} \ = \ \sum_{i=1}^{m}u^j_i(t)\,f_i(x).
\]
Note that each~$\Sigma^j$ can be interpreted as a control-affine system with control vector fields~$f_i$ under the open-loop controls $u_i=u^j_i(t)$. In this situation, we can apply the general convergence theory, which leads to the following result: If the sequences $(u^j_i)_j$ satisfy certain convergence conditions in terms of iterated integrals, then the solutions of the~$\Sigma^j$ converge to the solutions of~$\Sigma^{\infty}$ as the sequence index~$j$ tends to~$\infty$. Clearly, the notion of convergence for the sequences $(u^j_i)_j$ has to depend somehow on the limit equation~$\Sigma^{\infty}$, but we do not go into details at this point. On the other hand, we have the notion of convergence for the trajectories. Roughly speaking, this means that for a fixed initial state and on a fixed compact time-interval, the sequence of solutions of the~$\Sigma^j$ converges uniformly to the solution of~$\Sigma^{\infty}$.

The general convergence theory in~\cite{Kurzweil1988,Liu1997} is not restricted to the example above, but can be applied for arbitrary smooth vector fields~$f_i$ in~$\Sigma^{j}$, and for the case in which iterated Lie brackets of the~$f_i$ with possibly time-varying coefficient functions appear on the right-hand side of~$\Sigma^{\infty}$. A system of this form is then usually referred to as an \emph{extended system}. It is shown in~\cite{Liu19972} that, for every extended system, one can find sequences $(u_i^j)_j$ of inputs such that the trajectories of the corresponding sequence of open-loop systems converge to the trajectories of the desired extended system. This result can be applied in the context of motion planning for non\-holonomic systems, as in~\cite{Tilbury1992}. Another application, which is also studied in the present paper, is the stabilization of control systems. The idea is that stability properties of the extended system carry over to members of the approximating sequence of systems with sufficiently large sequence index. It turns out that this strategy works quite well when the~$f_i$ are homogeneous vector fields in the Euclidean space. For this situation, the following implication is known from~\cite{Morin1999,Moreau2003}: If the extended system is asymptotically stable, then the convergence of trajectories ensures that the same is true for members of the approximating sequence of systems with sufficiently large sequence index. However, this statement is not valid when the assumption of homogeneity is dropped. One obstacle is that convergence for a fixed initial state and a fixed compact time interval is in general not strong enough to transfer stability. Instead, the convergence of the trajectories must be uniform with respect to the initial states within compact sets, and the initial time. Under the assumption that this type of convergence is present, the authors of~\cite{Moreau19994} prove the following implication: If the extended system is \emph{locally uniformly asymptotically stable} (LUAS), then the approximating sequence of systems is \emph{practically locally uniformly asymptotically stable} (PLUAS). Here, \emph{uniform} means uniform with respect to the initial time because the systems are allowed to have possibly non\-periodic time dependencies. Moreover, \emph{practical} means that we can only expect practical stability for members of the approximating sequence with sufficiently large sequence index. 

The strategy of using convergence of trajectories to transfer stability from an extended system to an approximating sequence of systems has also been applied in papers on extremum seeking control, which we cited above. To see this, we insert the sequence of output feedback control laws $u_k = u^j_{2k-1}(t)\,h_{\text{s}}(\psi(x)) + u^j_{2k}(t)\,h_{\text{c}}(\psi(x))$ for $k=1,\ldots,p$ into~\cref{eq:intro:01}, where the functions~$h_{\text{s}},h_{\text{c}}$ and the sequences $(u^j_i)$ are chosen as before. We then obtain the sequence of differential equations~$\Sigma^j$. The results in~\cite{Duerr2013} show that, if the sequences $(u^j_i)_j$ satisfy certain convergence conditions in terms of iterated integrals, which are slightly stronger than the conditions in~\cite{Kurzweil1987,Liu1997}, then the convergence of trajectories is sufficiently strong to transfer stability. In particular, this leads to the following implication: If a point~$x$ of~$M$ is locally asymptotically stable for the autonomous~$\Sigma^{\infty}$, then this point~$x$ is PLUAS for the sequence $(\Sigma^j)_j$. Note that a sufficient condition to ensure that~$x$ is locally asymptotically stable for~$\Sigma^{\infty}$ is that $\psi$ attains a strict local minimum at $x$, and that there exists a punctured neighborhood of $x$ in which $\psi$ has no critical point and the vector fields~$g_k$ span the tangent bundle. In this case, one can solve the problem of stabilizing~\cref{eq:intro:01} at a minimum of~$\psi$ by means of output feedback, at least approximately.

The present paper contributes to the above research fields in various ways. For instance, we generalize the convergence results in~\cite{Duerr2013}, which are valid for extended systems with Lie brackets of two vector fields, toward extended systems with arbitrary iterated Lie brackets on the right-hand side, in the sense of~\cite{Kurzweil1987,Liu1997}. We treat the general case of a control-affine system with drift, and allow a possible time dependence of the vector fields. Our convergence results are sufficiently strong to ensure that the approximating sequence of systems is PLUAS whenever the extended system is LUAS. Moreover, we study how the quality of the approximation is influenced by boundedness assumptions on the vector fields of the control-affine system. Under suitable boundedness assumptions, we prove the following new implication: If an extended system is \emph{locally uniformly exponentially stable} (LUES), then the members of the approximating sequence of systems with sufficiently large sequence index are LUES as well. We also study the problem of stabilizing a general control-affine system at minima of its output function by means of output feedback. Under the additional assumption that at least the minimum value of the output function is known, we present an output feedback control law that not only can induce practical stability but in fact exponential stability. Moreover, we show by the example of a non\-holonomic system that our approach has the ability to obtain access to descent directions of the output function along Lie brackets of the control vector fields. 

The remainder of this paper is organized as follows. In \Cref{sec:02}, we recall and extend the algebraic formalism from~\cite{Liu1997}, which is needed to define a suitable notion of convergence for sequences of open-loop controls. The output feedback control law, which we present in \Cref{sec:06}, is not smooth but at least locally Lipschitz continuous. For this reason, we have to define Lie derivatives and Lie brackets in a suitable way. These and other some other definitions from differential geometry are summarized in \Cref{sec:03}. In \Cref{sec:04}, we recall some known notions of stability and present convergence conditions for trajectories of sequences of time-varying systems that are sufficient for the transfer of stability. Then, in \Cref{sec:05}, we present the main convergence results for control-affine systems. In \Cref{sec:06}, we apply these results to the problem of output optimization and distanced-based formation control.


\section{Algebraic preliminaries}\label{sec:02}
We begin this section by reviewing some aspects of the algebraic formalism in~\cite{Liu1997}, which is used throughout this paper. Let $\mathbf{X}=\{X_1,\ldots,X_{m}\}$ be a finite set of $m\geq1$ non\-commuting variables (also called indeterminates). A \emph{monomial} in~$\mathbf{X}$ is a finite sequence of elements of~$\mathbf{X}$, which is usually written as a product of the form $X_I:=X_{i_1}X_{i_2}\cdots{X_{i_{\ell}}}$, where $I=(i_1,\ldots,i_{\ell})$ is a finite multi-index with $i_1,i_2,\ldots,i_{\ell}\in\{1,\ldots,m\}$. The \emph{length} $\ell$ of such a multi-index~$I$ is denoted by~$|I|$. If $I=(i_1,\ldots,i_{{\ell}})$ and $J=(j_1,\ldots,j_{p})$ are two multi-indices, then we define their \emph{concatenation product}~$IJ$ as the multi-index $(i_1,\ldots,i_{\ell},j_1,\ldots,j_{p})$ by first writing the elements of~$I$, and then those of~$J$. The concatenation product for multi-indices naturally leads to the product $X_IX_J:=X_{IJ}$ of the associated monomials. In particular, for the empty multi-index of length~$0$, we obtain the empty product in~$\mathbf{X}$, which is the unit element in the monoid of monomials. A \emph{non\-commutative polynomial} in~$\mathbf{X}$ over~$\mathbb{R}$ (or simply a \emph{polynomial}) is a linear combination over~$\mathbb{R}$ of monomials in~$\mathbf{X}$. If~$p$ is a polynomial, it can then be written as $p=\sum_{I}p_IX_I$ with unique real-valued \emph{coefficients}~$p_I$, where the sum over all multi-indices~$I$ is finite because all but finitely many of the~$p_I$ are zero. The set of all non\-commutative polynomials in~$\mathbf{X}$ over~$\mathbb{R}$ is denoted by $A(\mathbf{X})$, which has a natural structure of an associative~$\mathbb{R}$-algebra. 
Because of its universal property among all associative~$\mathbb{R}$-algebra with~$m$ generators (cf.~\cite{ReutenauerBook}), the algebra $A(\mathbf{X})$ is called the \emph{free} associative algebra generated by~$\mathbf{X}$ over~$\mathbb{R}$. For any two non\-commutative polynomials $p,q\in{A(\mathbf{X})}$, we define their Lie bracket as usual by $[p,q]:=pq-qp$. This turns $A(\mathbf{X})$ into a Lie algebra. Let $L(\mathbf{X})$ be the Lie subalgebra of $A(\mathbf{X})$ generated by the variables $X_1,\ldots,X_m$ with respect to the bracket product. The elements of $L(\mathbf{X})$ are called \emph{Lie polynomials}. One can show that $L(\mathbf{X})$ is spanned by all Lie brackets of the form
\[
[X_I] \ := \ [X_{i_1},[X_{i_2},\cdots[X_{i_{\ell-1}},X_{i_{\ell}}]\cdots]]
\]
with $I=(i_1,\ldots,i_{\ell})$ being a non\-empty multi-index with $i_1,\ldots,i_{\ell}\in\{1,\ldots,m\}$. Because of its universal property among all Lie algebras with~$m$ generators (cf.~\cite{ReutenauerBook}), the Lie algebra $L(\mathbf{X})$ is called the \emph{free} Lie algebra generated by~$\mathbf{X}$ over~$\mathbb{R}$.

Let~$\mathbf{v}$ be a map defined on~$\mathbb{R}$ with values in $A(\mathbf{X})$. We can then write $\mathbf{v}=\sum_{I}v_IX_I$ with unique \emph{coefficient functions} $v_I\colon\mathbb{R}\to\mathbb{R}$. The map~$\mathbf{v}$ is said to be \emph{bounded} (\emph{Lebesgue measurable}, \emph{continuous}, \emph{locally absolutely continuous} etc.) if all of its coefficient functions $v_I\colon\mathbb{R}\to\mathbb{R}$ are bounded (Lebesgue measurable, continuous, locally absolutely continuous, etc.). For every positive integer $r$, let $A^r_0(\mathbf{X})$ denote the subspace of polynomials $\sum_Ip_IX_I\in{A(\mathbf{X})}$ whose coefficients $p_I$ with multi-indices $I$ of length $|I|=0$ and $|I|>r$ are zero. As in~\cite{Liu1997}, we introduce the following types of $A^r_0(\mathbf{X})$-valued maps.
\begin{definition}\label{def:inputs}
Let~$r$ be a positive integer. A \emph{polynomial input of order~$\leq{r}$} is a Lebesgue measurable and bounded $A^r_0(\mathbf{X})$-valued map $\mathbf{v}$ defined on~$\mathbb{R}$; we write $\mathbf{v}=\sum_{0<|I|\leq{r}}v_IX_I$. An \emph{extended input of order~$\leq{r}$} is a polynomial input of order~$\leq{r}$, which takes values in $L(\mathbf{X})$. An \emph{ordinary input}~$\mathbf{u}$ is a polynomial input of order~$\leq{1}$; here, we write $\mathbf{u}=\sum_{i=1}^{m}u_iX_i$.
\end{definition}
When it is not important to name the indeterminates $X_1,\ldots,X_m$ explicitly, we simply write $\mathbf{v}=(v_I)_{0<|I|\leq{r}}$ instead of $\mathbf{v}=\sum_{0<|I|\leq{r}}v_IX_I$ for a polynomial input of order~$\leq{r}$, and write $\mathbf{u}=(u_i)_{i=1,\ldots,m}$ instead of $\mathbf{u}=\sum_{i=1}^{m}u_iX_i$ for an ordinary input.

The next definition is important ingredient to obtain a suitable notion of convergence for sequences of ordinary inputs. Again, we refer the reader to~\cite{Liu1997} for further explanations.
\begin{definition}\label{def:generalizedDifference}
An \emph{$r$th-order generalized difference} of an ordinary input $\mathbf{u}=\sum_{i=1}^{m}u_iX_i$ and a polynomial input $\mathbf{v}=\sum_{0<|I|\leq{r}}v_IX_I$ of order~$\leq{r}$ is a locally absolutely continuous $A^r_0(\mathbf{X})$-valued map $\mathbf{W}$ defined on~$\mathbb{R}$ that satisfies the differential equation
\begin{equation}\label{eq:GDDefinition}
\dot{\mathbf{W}} \ = \ -\mathbf{u}\mathbf{W} + \mathbf{v} - \mathbf{u}
\end{equation}
up to order~$r$, i.e., if we write $\mathbf{W}=\sum_{0<|I|\leq{r}}\widetilde{UV}\vphantom{UV}_IX_I$, the differential equations
\begin{align*}
\dot{\widetilde{UV}}\vphantom{UV}_i & \ = \ v_i - u_i \allowdisplaybreaks \\
\dot{\widetilde{UV}}\vphantom{UV}_{iI} & \ = \ v_{iI} - u_i\,\widetilde{UV}\vphantom{UV}_I
\end{align*}
are satisfied almost everywhere on $\mathbb{R}$ for $i=1,\ldots,m$ and multi-indices~$I$ of length $0<|I|<r$.
\end{definition}
\begin{remark}
As in~\cite{Liu1997,Liu19972}, we use the notation $\sum_{0<|I|\leq{r}}\widetilde{UV}\vphantom{UV}_IX_I$ for an $r$th-order generalized difference of an ordinary input $\mathbf{u}=\sum_{i=1}^{m}u_iX_i$ and a polynomial input $\mathbf{v}=\sum_{0<|I|\leq{r}}v_IX_I$ of order~$\leq{r}$. Note that for any given initial value $\mathbf{p}=\sum_{0<|I|\leq{r}}p_IX_I\in{A^r_0(\boldsymbol{X})}$ and any initial time $t_0\in\mathbb{R}$, the unique global solution of~\cref{eq:GDDefinition} with initial condition $(t_0,\mathbf{p})$ can be computed recursively by
\begin{align*}
\widetilde{UV}\vphantom{UV}_i(t) \ = \ & p_i + \int_{t_0}^{t}\big(v_i(s) - u_i(s)\big)\,\mathrm{d}s, \allowdisplaybreaks \\
\widetilde{UV}\vphantom{UV}_{iI}(t) \ = \ & p_{iI} + \int_{t_0}^{t}\big(v_{iI}(s) - u_i(s)\,W_I(s)\big)\,\mathrm{d}s
\end{align*}
with the usual convention $\int_{t_0}^{t}\alpha(s)\,\mathrm{d}s:=-\int_{t}^{t_0}\alpha(s)\,\mathrm{d}s$ for the Lebesgue integral of a locally Lebesgue integrable function $\alpha$ if $t<t_0$.
\end{remark}
Now, we define the notion of convergence for sequences of ordinary inputs, which is used in \Cref{sec:05} to prove convergence of trajectories of control systems. The conditions are slightly stronger than in~\cite{Liu1997}. This is the price that we have to pay to obtain stronger convergence results for the trajectories.
\begin{definition}\label{def:GDConvergence}
Let $(\mathbf{u}^j)_j$ be a sequence of ordinary inputs $\mathbf{u}^j=\sum_{i=1}^{m}u^j_iX_i$. Let $\mathbf{v}=\sum_{0<|I|\leq{r}}v_IX_I$ be a polynomial input of order~$\leq{r}$. We say that \emph{$(\mathbf{u}^j)_j$ GD($r$)-converges uniformly to~$\mathbf{v}$} if there exists a sequence of polynomial inputs $\mathbf{v}^j=\sum_{0<|I|\leq{r}}v^j_IX_I$ of order~$\leq{r}$, and a sequence of $r$th-order generalized differences $\sum_{0<|I|\leq{r}}\widetilde{UV}\vphantom{UV}^j_IX_I$ of~$\mathbf{u}^j$ and~$\mathbf{v}^j$ such that for all multi-indices $0<|I|\leq{r}$, the following conditions are satisfied:
\begin{enumerate}[label=c\arabic*($r$)]
	\item\label{def:GDConvergence:1} the sequence of~$v^j_I$ converges to~$v_I$ uniformly on~$\mathbb{R}$ as $j\to\infty$;
	\item\label{def:GDConvergence:2} the sequence of $\widetilde{UV}\vphantom{UV}^j_I$ converges to~$0$ uniformly on~$\mathbb{R}$ as $j\to\infty$; and
	\item\label{def:GDConvergence:3} if $|I|=r$, then, for $i=1,\ldots,m$, the sequence of $u^j_i\widetilde{UV}\vphantom{UV}^j_I$ converges to~$0$ uniformly on~$\mathbb{R}$ as $j\to\infty$,
\end{enumerate}
\end{definition}
For later reference, we cite the following result from~\cite{Liu1997}.
\begin{proposition}\label{thm:algebraicIdentity}
Suppose that the conditions~\ref{def:GDConvergence:1} and~\ref{def:GDConvergence:2} in \Cref{def:GDConvergence} are satisfied. Then, the polynomial input~$\mathbf{v}$ is in fact $L(\mathbf{X})$-valued almost everywhere on~$\mathbb{R}$, and
\[
\sum_{0<|I|\leq{r}}v_IX_I \ = \ \sum_{0<|I|\leq{r}}\frac{v_I}{|I|}\,[X_I]
\]
holds almost everywhere on~$\mathbb{R}$.
\end{proposition}
\begin{remark}\label{thm:SussmannAlgorithm}
For some applications, it is important to know how to find a sequence $(\mathbf{u}^j)_j$ of ordinary inputs that GD($r$)-converges uniformly to a prescribed extended input~$\mathbf{v}$ of order~$\leq{r}$. This problem is solved in~\cite{Liu19972}. A careful analysis of the proof of the main result in~\cite{Liu19972} shows that the sequence $(\mathbf{u}^j)_j$ constructed therein satisfies the above conditions of uniform GD($r$)-convergence.
\end{remark}


\section{Time-varying functions, vector fields, and systems}\label{sec:03}
We assume that the reader is familiar with basic concepts of differential geometry, as presented in \cite{AbrahamBook,LeeBook}, for example. Let~$M$ be a \emph{smooth manifold}, i.e., a paracompact Hausdorff space endowed with an $n$-dimensional smooth structure. Throughout this paper, \emph{smooth} means~$C^{\infty}$, and the term \emph{function} is reserved for real-valued maps. The algebra of smooth functions on~$M$ is denoted by $C^{\infty}(M)$. A \emph{pseudo-distance function} on~$M$ is a non\-negative function $d\colon{M}\times{M}\to\mathbb{R}$ with the following three properties for all $x,y,z\in{M}$: (i)~$x=y$ implies $d(x,y)=0$, (ii)~$d(x,y)=d(y,x)$, and (iii)~$d(x,z)\leq{d(x,y)+d(y,z)}$. If~$d$ is a pseudo-distance function on~$M$, then for every $x\in{M}$ and every non\-empty subset~$E$ of~$M$, we let $d(x,E):=\inf_{y\in{E}}d(x,y)$ denote the pseudo-distance between~$x$ and~$E$ with respect to~$d$. A \emph{distance function} on~$M$ is a pseudo-distance function on~$M$ such that $d(x,y)=0$ implies $x=y$. We use the term \emph{distance function} rather than \emph{metric} to avoid any confusion with the notion of a Riemannian metric. A distance function~$d$ on~$M$ is said to be \emph{locally Euclidean} if for every point $x\in{M}$ there exist a smooth chart $(U,\varphi)$ for~$M$ at~$x$ and constants $L_1,L_2>0$ such that $L_1\|\varphi(z)-\varphi(y)\|\leq{d(y,z)}\leq{L_2}\|\varphi(z)-\varphi(y)\|$ for all $y,z\in{U}$. Here and in the following, $\|\cdot\|$ denotes the Euclidean norm. Note that every distance function that arises from a Riemannian metric is locally Euclidean.

Next, we recall some definitions from~\cite{Sussmann1998}. A \emph{time-varying function on~$M$} is a function with domain $\mathbb{R}\times{M}$. A \emph{Carath{\'e}odory function} on~$M$ is a time-varying function~$\beta$ on~$M$ such that (i) for every $x\in{M}$, the function $\beta(\,\cdot\,,x)\colon\mathbb{R}\to\mathbb{R}$ is Lebesgue measurable, and (ii) for every $t\in\mathbb{R}$, the function $\beta(t,\,\cdot\,)\colon{M}\to\mathbb{R}$ is continuous. Let~$\beta$ be a time-varying function on~$M$, and let~$b$ be a non\-negative function on~$M$. For a given subset~$V$ of~$M$, we say that~$\beta$ is \emph{uniformly bounded by a multiple of~$b$ on~$V$} if there is a constant $c>0$ such that $|\beta(t,x)|\leq{c\,b(x)}$ holds for every $(t,x)\in\mathbb{R}\times{V}$. We say that~$\beta$ is \emph{locally uniformly bounded by a multiple of~$b$} if for every $x\in{M}$ there is a neighborhood~$V$ of~$x$ in~$M$ such that~$\beta$ is uniformly bounded by a multiple of~$b$ on~$V$. Moreover, we simply say that~$\beta$ is \emph{locally uniformly bounded} if it is locally uniformly bounded by $b\equiv1$. Finally, we say that~$\beta$ is \emph{locally uniformly Lipschitz continuous} if for every point $x\in{M}$ there exist a smooth chart $(U,\varphi)$ for~$M$ at~$x$ and a constant $L>0$ such that $|\beta(t,z)-\beta(t,y)|\leq{L\|\varphi(z)-\varphi(y)\|}$ holds for every $t\in\mathbb{R}$ and all $y,z\in{U}$. Equivalently, using a locally Euclidean distance function~$d$ on~$M$, the function~$\beta$ is locally uniformly Lipschitz continuous if and only if for every point $x\in{M}$, there exist a neighborhood~$V$ of~$x$ in~$M$ and a constant $L>0$ such that $|\beta(t,z)-\beta(t,y)|\leq{L d(y,z)}$ for every $t\in\mathbb{R}$ and all $y,z\in{V}$. Note that all of the above definitions also apply to functions on~$M$ because every function $\alpha$ on~$M$ can be considered as the time-varying function $(t,x)\mapsto\alpha(x)$. In this case, we usually omit the adjective ``uniform'', which indicates uniformity with respect to time.

A \emph{time-varying vector field on~$M$} is a map defined on $\mathbb{R}\times{M}$ that assigns every $(t,x)\in\mathbb{R}\times{M}$ to a tangent vector to~$M$ at~$x$. We consider the tangent space at any point $x\in{M}$ as the vector space of all derivations of $C^{\infty}(M)$ at~$x$. In this sense, every time-varying vector field~$f$ on~$M$ is a time-varying differential operator on $C^{\infty}(M)$. For every $(t,x)\in\mathbb{R}\times{M}$ and every $\alpha\in{C^{\infty}(M)}$, we let $(f\alpha)(t,x)$ denote the \emph{Lie derivative} of~$\alpha$ along $f$ at $(t,x)$, i.e., the real number obtained by applying the tangent vector $f(t,x)$ to $\alpha$. Using Lie derivatives along smooth functions on~$M$, all definitions for time-varying functions can be made for time-varying vector fields. For instance,~$f$ is a \emph{Carath{\'e}odory vector field} if for every $\alpha\in{C^{\infty}(M)}$, the time-varying function $f\alpha$ is a Carath{\'e}odory function on~$M$. Using the same convention as for functions on~$M$, we associate every vector field $g$ on~$M$ with the corresponding time-varying vector field $(t,x)\mapsto{g(x)}$.

A \emph{curve} on~$M$ is a continuous map from an open, non\-empty interval into~$M$. A curve $\gamma\colon{J}\to{M}$ is said to be \emph{locally absolutely continuous} if for every $\alpha\in{C^{\infty}}(M)$, the function $\alpha\circ\gamma\colon{J}\to\mathbb{R}$ is locally absolutely continuous. If $\gamma$ is locally absolutely continuous, then it is differentiable at almost every $t\in{J}$, and we denote by $\dot{\gamma}(t)$ the tangent vector to~$M$ at $\gamma(t)$ that maps every $\alpha\in{C^{\infty}(M)}$ to $(\alpha\circ\gamma)\dot{\vphantom{.}}(t)$. Let~$f$ be a time-varying vector field on~$M$. For given $t_0\in\mathbb{R}$ and $x_0\in{M}$, an \emph{integral curve} of~$f$ with initial condition $(t_0,x_0)$ is a locally absolutely continuous curve $\gamma\colon{J}\to{M}$ with $t_0\in{J}$ and $\gamma(t_0)=x_0$ such that $\dot{\gamma}(t)=f(t,\gamma(t))$ holds for almost every $t\in{J}$. We call the formal expression $\dot{x}=f(t,x)$ a \emph{time-varying system} on~$M$. A \emph{solution} (or also \emph{trajectory}) of $\dot{x}=f(t,x)$ is an integral curve of~$f$. The system $\dot{x}=f(t,x)$ is said to have the \emph{existence property of solutions} if for every $t_0\in\mathbb{R}$ and $x_0\in{M}$ there is an integral curve of~$f$ with the initial condition $(t_0,x_0)$. A \emph{maximal solution} of $\dot{x}=f(t,x)$ is a solution that cannot be extended to a solution of $\dot{x}=f(t,x)$ defined on a strictly larger interval. The system $\dot{x}=f(t,x)$ has the \emph{uniqueness property of solutions} if, whenever $\gamma_1\colon{J_1}\to{M}$ and $\gamma_2\colon{J_2}\to{M}$ are two solutions of $\dot{x}=f(t,x)$ for which there exists $t\in{J_1\cap{J_2}}$ with $\gamma_1(t)=\gamma_2(t)$, then $\gamma_1=\gamma_2$ holds on $J_1\cap{J_2}$. Suppose that $\dot{x}=f(t,x)$ has the existence and uniqueness property of solutions. Then we define $\Phi(t,t_0,x_0):=\gamma_{t_0,x_0}(t)$ for every triple $(t,t_0,x_0)$ with $t_0\in\mathbb{R}$, $x_0\in{M}$, and $t$ in the domain of the unique maximal integral curve $\gamma_{t_0,x_0}$ of~$f$ with initial condition $(t_0,x_0)$. The corresponding map $\Phi$ is called the \emph{flow} of~$f$ (or also the \emph{flow} of $\dot{x}=f(t,x)$).

In this paper, we are also interested in Lie derivatives of not necessarily differentiable functions. We define the Lie derivative for the following situation. Suppose that~$f$ is a time-varying vector field on~$M$ with the existence and uniqueness property of integral curves, and let $\Phi$ be the flow of~$f$. Let~$\beta$ be a time-varying function on~$M$, and let $(t,x)\in\mathbb{R}\times{M}$. If the limit
\[
(f\beta)(t,x) \ := \ \lim_{s\to0}\frac{1}{s}\big(\beta\big(t,\Phi(t+s,t,x)\big)-\beta(t,x)\big)
\]
exists, then we call it the \emph{Lie derivative} of~$\beta$ along~$f$ at $(t,x)$. Note that this definition is consistent with the notion of Lie derivatives of smooth functions in terms of derivations (see, e.g., \cite{AbrahamBook}). Moreover, if the limit
\[
(\partial_t\beta)(t,x) \ := \ \lim_{s\to0}\frac{1}{s}\big(\beta(t+s,x)-\beta(t,x)\big)
\]
exists, then we call it the \emph{time derivative} of~$\beta$ at $(t,x)$. Suppose that $f,g$ are time-varying vector fields on~$M$ with the existence and uniqueness property of integral curves. Let $(t,x)\in\mathbb{R}\times{M}$. If for every $\alpha\in{C^{\infty}(M)}$, the Lie derivatives $(f(g\alpha))(t,x)$ and $(g(f\alpha))(t,x)$ exist, then
\[
([f,g]\alpha)(t,x) \ := \ (f(g\alpha))(t,x) - (g(f\alpha))(t,x)
\]
defines a tangent vector $[f,g](t,x)$ to~$M$ at $(t,x)$, which is called the \emph{Lie bracket} of $f,g$ at $(t,x)$.

Let $\pi\colon{M}\to\tilde{M}$ be a smooth map from the smooth manifold~$M$ to another smooth manifold $\tilde{M}$. A function $\alpha$ on~$M$ is said to be \emph{constant on the fibers of~$\pi$} if for all $x,y\in{M}$ with $\pi(x)=\pi(y)$ we have $\alpha(x)=\alpha(y)$. A time-varying vector field~$f$ on~$M$ is said to be \emph{$\pi$-related} to a time-varying vector field $\tilde{f}$ on $\tilde{M}$ if $\operatorname{T}_x\!\pi(f(t,x))=\tilde{f}(t,\pi(x))$ holds for all $t\in\mathbb{R}$ and $x\in{M}$, where $\operatorname{T}_x\!\pi$ denotes the \emph{tangent map} of~$\pi$ at~$x$. We say that a time-varying vector field~$f$ on~$M$ is \emph{tangent to the fibers of~$\pi$} if it is~$\pi$-related to the zero vector field on $\tilde{M}$. The map $\pi\colon{M}\to\tilde{M}$ is said to be \emph{proper} if the preimage of any compact subset of $\tilde{M}$ under~$\pi$ is a compact subset of~$M$. The map $\pi\colon{M}\to\tilde{M}$ is said to be a \emph{submersion} if at every $x\in{M}$, the tangent map $\operatorname{T}_x\!\pi$ is surjective. Under the assumption that $\pi\colon{M}\to\tilde{M}$ is a smooth surjective submersion, it is known (see~\cite{LeeBook}) that $\alpha\in{C^{\infty}(M)}$ is constant on the fibers of $\pi$ if and only if there exists $\tilde{\alpha}\in{C^{\infty}(\tilde{M})}$ such that $\alpha=\tilde{\alpha}\circ\pi$.


\section{Stability of time-varying systems}\label{sec:04}
Let~$M$ be a smooth manifold, and let~$d$ be a pseudo-distance function on~$M$. Throughout this section, we let $(\Sigma^j)_j$ be a sequence of time-varying systems $\Sigma^j$ on~$M$. We assume that each $\Sigma^j$ has the existence and uniqueness property of solutions. For every sequence index~$j$, let $\Phi^j$ be the flow of $\Sigma^j$. On the other hand, let $\Sigma^{\infty}$ be another time-varying system on~$M$ with the existence and uniqueness property of solutions, which will play the role of a ``limit system'' with respect to the sequence $(\Sigma^{j})_j$. Let $\Phi^{\infty}$ be the flow of $\Sigma^{\infty}$.
\begin{remark}
The choice of a sequence index~$j$ as parameter is motivated by the notation in \cite{Kurzweil1987,Sussmann1991,Liu1997,Liu19972}. We note that all results in the present paper remain valid when the notion of a sequence is replaced by a \emph{net}, as it is done in~\cite{Liu1999}. In particular, the sequence index~$j$ may be replaced by a small parameter $\varepsilon>0$ that tends to $0$ as in \cite{Kurzweil1988,Morin1999,Moreau19994,Moreau2003}, or by a frequency parameter $\omega>0$ that tends to $\infty$ as in~\cite{Duerr2013,Duerr2014,Scheinker2013,Scheinker2014}.
\end{remark}

\subsection{Notions of stability}
For a single system like $\Sigma^{\infty}$, we first recall well-known notions of asymptotic and exponential stability with respect to a nonempty subset $E$ of $M$., see, e.g., \cite{MichelBook}.
\begin{definition}\label{def:LUAS}
The set $E$ is said to be \emph{locally uniformly asymptotically stable} (abbreviated \emph{LUAS}) for $\Sigma^{\infty}$ if
\begin{enumerate}[label=(\roman*)]
	\item\label{item:LUAS:stability}~$E$ is \emph{uniformly stable} for $\Sigma^{\infty}$, i.e., for every $\varepsilon>0$ there is $\delta>0$ such that for every $t_0\in\mathbb{R}$ and every $x_0\in{M}$ with $d(x_0,E)\leq\delta$, the trajectory $\Phi^{\infty}(\cdot,t_0,x_0)$ exists on $[t_0,\infty)$ with $d(\Phi^{\infty}(t,t_0,x_0),E)\leq\varepsilon$ for every $t\geq{t_0}$; and if
	\item\label{item:LUAS:attraction}~$E$ is \emph{locally uniformly attractive} for $\Sigma^{\infty}$, i.e., there is some $\delta>0$ with the following property: for every $\varepsilon>0$ there is $T>0$ such that for every $t_0\in\mathbb{R}$ and every $x_0\in{M}$ with $d(x_0,E)\leq\delta$, the trajectory $\Phi^{\infty}(\cdot,t_0,x_0)$ exists on $[t_0,\infty)$ with $d(\Phi^{\infty}(t,t_0,x_0),E)\leq\varepsilon$ for every $t\geq{t_0+T}$.
\end{enumerate}
\end{definition}
\begin{definition}\label{def:LUES}
The set $E$ is said to be \emph{locally uniformly exponentially stable} (abbreviated \emph{LUES}) for $\Sigma^{\infty}$ if there are $\delta,\lambda,\mu>0$ such that for every $t_0\in\mathbb{R}$ and every $x_0\in{M}$ with $d(x_0,E)\leq\delta$, the trajectory $\Phi^{\infty}(\cdot,t_0,x_0)$ exists on $[t_0,\infty)$ with $d(\Phi^{\infty}(t,t_0,x_0),E)\leq\lambda\,d(x_0,E)\,\mathrm{e}^{-\mu(t-t_0)}$ for every $t\geq{t_0}$.
\end{definition}

For sequences of time-varying systems there is a weaker notions of asymptotic stability, which is due to~\cite{Moreau19994}.
\begin{definition}\label{def:PLUAS}
The set $E$ is said to be \emph{practically locally uniformly asymptotically stable} (abbreviated \emph{PLUAS}) for $(\Sigma^{j})_j$ if
\begin{enumerate}[label=(\roman*)]
	\item\label{item:PLUAS:stability}~$E$ is \emph{practically uniformly stable} for $(\Sigma^{j})_j$, i.e., for every $\varepsilon>0$ there are $\delta>0$ and a sequence index $j_0$ such that for every $j\geq{j_0}$, every $t_0\in\mathbb{R}$, and every $x_0\in{M}$ with $d(x_0,E)\leq\delta$, the trajectory $\Phi^{j}(\cdot,t_0,x_0)$ exists on $[t_0,\infty)$ with $d(\Phi^{j}(t,t_0,x_0),E)\leq\varepsilon$ for every $t\geq{t_0}$; and if
	\item\label{item:PLUAS:attraction}~$E$ is \emph{practically locally uniformly attractive} for $(\Sigma^{j})_j$, i.e., there is some $\delta>0$ with the following property: for every $\varepsilon>0$ there are $T>0$ and a sequence index $j_0$ such that for every $j\geq{j_0}$, every $t_0\in\mathbb{R}$, and every $x_0\in{M}$ with $d(x_0,E)\leq\delta$, the trajectory $\Phi^{j}(\cdot,t_0,x_0)$ exists on $[t_0,\infty)$ with $d(\Phi^{j}(t,t_0,x_0),E)\leq\varepsilon$ for every $t\geq{t_0+T}$.
\end{enumerate}
\end{definition}
The set $E$ is PLUAS for $(\Sigma^j)_j$, then this ensures that the trajectories of $(\Sigma^j)_j$ are locally attracted into a prescribed $\varepsilon$-neighborhood of~$E$ as long as~$j$ is sufficiently large. However, in general, it is not known how large~$j$ has to be chosen for a given $\varepsilon>0$. In the present paper, we prove a stronger notion of stability for a sequence of systems, namely the following.
\begin{definition}
For a given sequence index $j_0$, we say that the set $E$ is \emph{LUES} for $(\Sigma^j)_{j\geq{j_0}}$ if there are $\delta,\lambda,\mu>0$ such that for every $j\geq{j_0}$, every $t_0\in\mathbb{R}$, and every $x_0\in{M}$ with $d(x_0,E)\leq\delta$, the trajectory $\Phi^{j}(\cdot,t_0,x_0)$ exists on $[t_0,\infty)$ with $d(\Phi^{j}(t,t_0,x_0),E)\leq\lambda\,d(x_0,E)\,\mathrm{e}^{-\mu(t-t_0)}$ for every $t\geq{t_0}$.
\end{definition}

\subsection{Convergence-based transfer of stability}
Our aim is to carry over stability properties of $\Sigma^{\infty}$ to members of the sequence $(\Sigma^j)_j$. It is shown in several earlier works, such as~\cite{Morin1999,Moreau2003,Duerr2013}, that this can be done by means of convergence of trajectories. For this reason, we introduce the following notion of convergence.
\begin{definition}\label{def:convergence}
Let $S$ be a subset of~$M$, and let~$b$ be a non\-negative function on~$M$. We say that \emph{the trajectories of $(\Sigma^j)_j$ converge uniformly in $S$ on compact time intervals with respect to $(d,b)$ to the trajectories of $\Sigma^{\infty}$} if for every $\varepsilon>0$ and every $T>0$ there exists a sequence index $j_0$ such that the following \emph{approximation property} holds: whenever $t_0\in\mathbb{R}$ and $x_0\in{S}$ are such that $\Phi^{\infty}(\cdot,t_0,x_0)$ exists on $[t_0,t_0+T]$ with $\Phi^{\infty}(t,t_0,x_0)\in{S}$ for every $t\in[t_0,t_0+T]$, then, for every $j\geq{j_0}$, also $\Phi^{j}(\cdot,t_0,x_0)$ exists on $[t_0,t_0+T]$ and
\[
d\big(\Phi^{j}(t,t_0,x_0), \Phi^{\infty}(t,t_0,x_0)\big) \ \leq \ \varepsilon\,b(x_0)
\]
holds for every $t\in[t_0,t_0+T]$.
\end{definition}
In particular, when the bound~$b$ is identically equal to $1$ on~$M$, then~\Cref{def:convergence} coincides with the notion of convergence in~\cite{Moreau19994,Moreau2000,Duerr2013}. If the bound~$b$ is chosen as the pseudo-distance to a subset of~$M$, then~\Cref{def:convergence} corresponds to the approximation property in~\cite{Moreau20002,Moreau2003}. Using the same strategy as in these papers, it is easy to derive the subsequent two propositions. We omit the proofs here. In the following, let $E$ be a nonempty subset of $M$.
\begin{proposition}\label{thm:LUASToPLUAS}
Suppose that $E$ is LUAS for $\Sigma^{\infty}$. Define $b(x):=1$ for every $x\in{M}$. Suppose that the trajectories of $(\Sigma^j)_j$ converge uniformly in some $\delta$-neighbor\-hood of~$E$ on compact time intervals with respect to $(d,b)$ to the trajectories of $\Sigma^{\infty}$. Then~$E$ is PLUAS for $(\Sigma^j)_j$.
\end{proposition}
%
\begin{proposition}\label{thm:LUESToLUES}
Suppose that $E$ is LUES for $\Sigma^{\infty}$. Define $b(x):=d(x,E)$ for every $x\in{M}$. Suppose that the trajectories of $(\Sigma^j)_j$ converge uniformly in some $\delta$-neighborhood of~$E$ on compact time intervals with respect to $(d,b)$ to the trajectories of $\Sigma^{\infty}$. Then there exists a sequence index $j_0$ such that~$E$ is LUES for $(\Sigma^j)_{j\geq{j_0}}$.
\end{proposition}

\section{Convergence results for control-affine systems}\label{sec:05}
In this section, we present our main convergence results for trajectories of control-affine systems under sequences of open loop controls. They are strong enough to ensure that the convergence properties in \Cref{thm:LUASToPLUAS,thm:LUESToLUES} are satisfied. As in the previous sections,~$M$ is any smooth manifold.

\subsection{A suitable class of control-affine systems}\label{sec:convergence:01}
As an abbreviation for several technical assumptions, we introduce the subsequent class of control-affine systems. A similar class is also considered in~\cite{Liu1997}. Using the terminology and notation from~\Cref{sec:03}, we make the following definition.
\begin{definition}\label{def:Sys}
For given positive integers $m,r$, we denote by $\operatorname{Sys}(M;m,r)$ the set of all $(m+1)$-tuples $(f_0,f_1,\ldots,f_m)$ of time-varying vector fields $f_0,f_1,\ldots,f_m$ on~$M$ such that the following properties hold for every $\alpha\in{C^{\infty}(M)}$.
\begin{enumerate}[label=(\roman*)]
	\item\label{item:def:Sys:i} The time-varying function $f_0\alpha$ is locally uniformly bounded, locally uniformly Lipschitz continuous, and, for every $x\in{M}$, the function $(f_0\alpha)(\cdot,x)\colon\mathbb{R}\to\mathbb{R}$ is measurable.
	\item\label{item:def:Sys:ii} For every multi-index $I=(i_1,\ldots,i_{\ell})$ of length $0<\ell\leq{r}$ with $i_1,\ldots,i_{\ell}\in\{1,\ldots,m\}$, the \emph{iterated Lie derivative}
	\[
	f_{I}\alpha \ := \ f_{i_1}(f_{i_2}(\cdots{f_{i_{\ell-1}}}(f_{i_{\ell}}\alpha)\cdots))
	\]
	exists as a locally uniformly bounded and locally uniformly Lipschitz continuous time-varying function on~$M$, and its time derivative
	\[
	\partial_tf_I\alpha \ := \ \partial_t(f_I\alpha)
	\]
	exists as a locally uniformly bounded time-varying function on~$M$ and is continuous as a function on $\mathbb{R}\times{M}$.
	\item\label{item:def:Sys:iii} For every multi-index $I=(i_1,\ldots,i_{r})$ of length $r$ with $i_1,\ldots,i_{r}\in\{1,\ldots,m\}$ and every $i\in\{1,\ldots,m\}$, the Lie derivative
	\[
	f_{iI}\alpha \ := \ f_if_I\alpha \ := \ f_i(f_I\alpha)
	\]
	exists as a locally uniformly bounded Carath{\'e}odory function on~$M$.
\end{enumerate}
If $(f_0,f_1,\ldots,f_m)\in\operatorname{Sys}(M;m,r)$, then for every multi-index $I=(i_1,\ldots,i_{\ell})$ of length $0<\ell\leq{r}$ with $i_1,\ldots,i_{\ell}\in\{1,\ldots,m\}$, we define the \emph{iterated Lie bracket}
\[
[f_I] \ := \ [f_{i_1},[f_{i_2},\cdots[f_{i_{\ell-1}},f_{i_{\ell}}]\cdots]].
\]
\end{definition}
If $(f_0,f_1,\ldots,f_m)\in\operatorname{Sys}(M;m,r)$, then we call $f_0$ the \emph{drift vector field} and $f_1,\ldots,f_m$ the \emph{control vector fields}.
\begin{remark}
Note that conditions~\ref{item:def:Sys:i} and~\ref{item:def:Sys:ii} in \Cref{def:Sys} ensure that the time-varying vector fields $f_0,f_1,\ldots,f_m$ have the existence and uniqueness property of integral curves. Thus, one can define Lie derivatives of time-varying functions along these vector fields as explained in \Cref{sec:03}.
\end{remark}
\begin{remark}
If $f_0$ is a $C^1$ vector field without time dependence, and if $f_1,\ldots,f_m$ are $C^r$ vector fields without time dependence, then $(f_0,f_1,\ldots,f_m)\in\operatorname{Sys}(M;m,r)$.
\end{remark}
Our goal in this section is to derive a convergence result for control-affine systems under open-loop controls. The proof of this result is based on repeated integration by parts along trajectories of the system. For this purpose, we need the subsequent lemma. The notions of polynomial and ordinary inputs, which appear in the following statement, are introduced in~\Cref{sec:02}.
\begin{lemma}\label{thm:integrationByParts}
For positive integers $m,r$, let $(f_0,f_1,\ldots,f_m)\in\operatorname{Sys}(M;m,r)$. Then, for every polynomial input $(v_I)_{0<|I|\leq{r}}$ of order $\leq{r}$, the time-varying system
\begin{equation}\label{eq:integrationByParts:A}
\dot{x} \ = \ f_0(t,x) + \sum_{0<|I|\leq{r}}\frac{v_I(t)}{|I|}\,[f_I](t,x)
\end{equation}
has the existence and uniqueness property of solutions. Moreover, for any ordinary input $(u_i)_{i=1,\ldots,m}$, every solution $\gamma\colon{J}\to{M}$ of
\begin{equation}\label{eq:integrationByParts:B}
\dot{x} \ = \ f_0(t,x) + \sum_{i=1}^{m}u_i(t)\,f_i(t,x),
\end{equation}
every multi-index~$I$ of length $0<|I|\leq{r}$, and every $\alpha\in{C^{\infty}(M)}$, we have that
\begin{enumerate}[label=(\alph*)]
	\item the function $t\mapsto(f_0f_I\alpha)(t,\gamma(t))$ exists almost everywhere on~$J$ and is locally Lebesgue integrable;
	\item the function $t\mapsto(f_I\alpha)(t,\gamma(t))$ is locally absolutely continuous on~$J$ and its derivative is given by
	\[
	\big(\partial_tf_I\alpha\big)\big(t,\gamma(t)\big) + \big(f_0f_I\alpha\big)\big(t,\gamma(t)\big) + \sum_{i=1}^{m}u_i(t)\big(f_if_I\alpha\big)\big(t,\gamma(t)	\big)
	\]
	at almost every $t\in{J}$.
\end{enumerate}
\end{lemma}
The proof is given in \Cref{sec:integrationByParts}.

\subsection{An open-loop convergence result}\label{sec:convergence:02}
In this subsection, let $m,r$ be positive integers and let $(f_0,f_1,\ldots,f_m)\in\operatorname{Sys}(M;m,r)$. As in \Cref{def:inputs}, let $(\mathbf{u}^j)_j$ be a sequence of ordinary inputs $\mathbf{u}^j=(u^j_i)_{i=1,\ldots,m}$, and let $\mathbf{v}=(v_I)_{0<|I|\leq{r}}$ be a polynomial input of order $\leq{r}$. For every sequence index~$j$, we consider the time-varying system
\[
\Sigma^{j}\colon\qquad\dot{x} \ = \ f^j(t,x) \ := \ f_0(t,x) + \sum_{i=1}^{m}u^j_i(t)\,f_i(t,x),
\]
and we consider the time-varying system
\[
\Sigma^{\infty}\colon\qquad\dot{x} \ = \ f^{\infty}(t,x) \ := \ f_0(t,x) + \sum_{0<|I|\leq{r}}\frac{v_I(t)}{|I|}\,[f_I](t,x).
\]
By \Cref{thm:integrationByParts}, the systems $\Sigma^{j}$ and $\Sigma^{\infty}$ have the existence and uniqueness property of solutions. Note that this is a special case of the situation that is studied in \Cref{sec:04}.

The subsequent \Cref{thm:IntegralExpansion} contains a key ingredient for the proof of our main convergence result, \Cref{thm:convergenceResult}. It is based on repeated integration by parts on the right-hand side of $\Sigma^j$ in integral form. The same strategy is also applied in~\cite{Kurzweil1987,Liu1997} to relate $\Sigma^j$ and $\Sigma^{\infty}$.
\begin{proposition}\label{thm:IntegralExpansion}
Suppose that there exist a sequence $(\mathbf{v}^j)_j$ of polynomial inputs $\mathbf{v}^j=(v^j_I)_{0<|I|\leq{r}}$ of order $\leq{r}$, and, for every sequence index~$j$, an $r$th-order generalized difference $(\widetilde{UV}\vphantom{UV}^j_I)_{0<|I|\leq{r}}$ of $\mathbf{u}^j$ and $\mathbf{v}^j$ such that conditions~\ref{def:GDConvergence:1} and~\ref{def:GDConvergence:2} in \Cref{def:GDConvergence} are satisfied. Then, for every $\alpha\in{C^{\infty}(M)}$, every sequence index~$j$, and every solution $\gamma\colon{J}\to{M}$ of $\Sigma^{j}$, we have
\begin{subequations}\label{eq:IntegralExpansion:01}
\begin{align}
\alpha(\gamma(t)) = \ & \alpha(\gamma(t_0)) + \int_{t_0}^{t}(f^{\infty}\alpha)(s,\gamma(s))\mathrm{d}s \label{eq:IntegralExpansion:01:A} \allowdisplaybreaks \\
& + (D_1^j\alpha)(t,\gamma(t)) - (D_1^j\alpha)(t_0,\gamma(t_0)) + \int_{t_0}^t(D_2^j\alpha)(s,\gamma(s))\mathrm{d}s \label{eq:IntegralExpansion:01:B}
\end{align}
\end{subequations}
for all $t_0,t\in{J}$, where
\[
(D_1^j\alpha)(s,\gamma(s)) \ := \ -\sum_{0<|I|\leq{r}}\widetilde{UV}\vphantom{UV}^j_I(s)\,(f_I\alpha)(s,\gamma(s))
\]
for every $s\in{J}$, and
\begin{align*}
(D_2^j\alpha)(s,\gamma(s)) \, := \ & \sum_{0<|I|\leq{r}}\widetilde{UV}\vphantom{UV}^j_I(s)\,(\partial_tf_I\alpha)(s,\gamma(s)) + \sum_{0<|I|\leq{r}}\widetilde{UV}\vphantom{UV}^j_I(s)\,(f_0f_I\alpha)(s,\gamma(s)) \allowdisplaybreaks \\
& \!\!\!\!\!\!\!\!\!\!\!\!\!\!\!\!\!\!\!\!\!\!\!\!\!\!\!\!\!\!\!\!\!\!\!\! + \sum_{0<|I|\leq{r}}(v^j_I(s)-v_I(s))\,(f_I\alpha)(s,\gamma(s)) + \sum_{i=1}^{m}\sum_{|I|=r}u^j_i(s)\,\widetilde{UV}\vphantom{UV}^j_{I}(s)\,(f_{i}f_{I}\alpha)(s,\gamma(s))
\end{align*}
for almost every $s\in{J}$.
\end{proposition}
\begin{proof}
Since $\gamma$ is a solution of $\Sigma^j$, it satisfies the integral equation
\[
\alpha(\gamma(t)) \ = \ \alpha(\gamma(t_0)) + \int_{t_0}^{t}\Big((f_0\alpha)(s,\gamma(s))+\sum_{i=1}^{m}u^j_i(s)\,(f_i\alpha)(s,\gamma(s))\Big)\,\mathrm{d}s
\]
for all $t_0,t\in{J}$. We add and subtract $v^j_i(s)$ to each of the $u^j_i(s)$ in the above equation and then apply integration by parts. Note that integration by parts is an admissible operation because of \Cref{thm:integrationByParts}. Since $\widetilde{UV}\vphantom{UV}^j_i$ is by definition an antiderivative of $u^j_i-v^j_i$, we obtain
\begin{align*}
\alpha(\gamma(t)) \ = \ & \alpha(\gamma(t_0)) + \int_{t_0}^{t}\Big((f_0\alpha)(s,\gamma(s))+\sum_{i=1}^{m}v^j_i(s)\,(f_i\alpha)(s,\gamma(s))\Big)\,\mathrm{d}s \allowdisplaybreaks \\
& - \sum_{i=1}^{m}\widetilde{UV}\vphantom{UV}^j_i(t)\,(f_i\alpha)(t,\gamma(t)) + \sum_{i=1}^{m}\widetilde{UV}\vphantom{UV}^j_i(t_0)\,(f_i\alpha)(t_0,\gamma(t_0)) \allowdisplaybreaks \\
& + \sum_{i=1}^{m}\int_{t_0}^{t}\widetilde{UV}\vphantom{UV}^j_i(s)\,\big((\partial_tf_i\alpha)(s,\gamma(s))+(f_0f_i\alpha)(s,\gamma(s))\big)\,\mathrm{d}s \allowdisplaybreaks \\
& + \sum_{i=1}^{m}\sum_{i'=1}^{m}\int_{t_0}^{t}u^j_i(s)\,\widetilde{UV}\vphantom{UV}^j_{i'}(s)\,(f_{i}f_{i'}\alpha)(s,\gamma(s))\,\mathrm{d}s.
\end{align*}
The procedure can be repeated by adding and subtracting $v^j_{i,i'}(s)$ to each of the $u^j_i(s)\widetilde{UV}\vphantom{UV}^j_{i'}(s)$ in the last line of the above equation followed by integration by parts. For the integration step, we exploit again the assumption that $\widetilde{UV}\vphantom{UV}^j_{i,i'}$ is an antiderivative of $v^j_{i,i'}-u^j_i\,\widetilde{UV}\vphantom{UV}^j_{i'}$. Continuing this way up to order $r$, we obtain the expansion
\begin{subequations}\label{eq:proof:IntegralExpansion:01}
\begin{align}
\alpha(\gamma(t)) \, = \ & \alpha(\gamma(t_0)) + \int_{t_0}^{t}\Big((f_0\alpha)(s,\gamma(s))+\sum_{0<|I|\leq{r}}v^j_I(s)\,(f_I\alpha)(s,\gamma(s))\Big)\,\mathrm{d}s \label{eq:proof:IntegralExpansion:01:A} \allowdisplaybreaks \\
& - \sum_{0<|I|\leq{r}}\widetilde{UV}\vphantom{UV}^j_I(t)\,(f_I\alpha)(t,\gamma(t)) + \sum_{0<|I|\leq{r}}\widetilde{UV}\vphantom{UV}^j_I(t_0)\,(f_I\alpha)(t_0,\gamma(t_0)) \label{eq:proof:IntegralExpansion:01:B} \allowdisplaybreaks \\
& + \sum_{0<|I|\leq{r}}\int_{t_0}^{t}\widetilde{UV}\vphantom{UV}^j_I(s)\,\big((\partial_tf_I\alpha)(s,\gamma(s))+(f_0f_I\alpha)(s,\gamma(s))\big)\,\mathrm{d}s \label{eq:proof:IntegralExpansion:01:C} \allowdisplaybreaks \\
& + \sum_{i=1}^{m}\sum_{|I|=r}\int_{t_0}^{t}u^j_i(s)\,\widetilde{UV}\vphantom{UV}^j_{I}(s)\,(f_{i}f_{I}\alpha)(s,\gamma(s))\,\mathrm{d}s. \label{eq:proof:IntegralExpansion:01:D}
\end{align}
\end{subequations}
In the last step, we add and subtract $\sum_{0<|I|\leq{r}}v_I(s)\,(f_I\alpha)(s,\gamma(s))$ to the integral in~\cref{eq:proof:IntegralExpansion:01:A}. Then \cref{eq:proof:IntegralExpansion:01} becomes
\begin{align*}
\alpha(\gamma(t)) \, = \ & \alpha(\gamma(t_0)) + \int_{t_0}^{t}\Big((f_0\alpha)(s,\gamma(s))+\sum_{0<|I|\leq{r}}v^j_I(s)\,(f_I\alpha)(s,\gamma(s))\Big)\,\mathrm{d}s \allowdisplaybreaks \\
& + (D_1^j\alpha)(t,\gamma(t)) - (D_1^j\alpha)(t_0,\gamma(t_0)) + \int_{t_0}^t(D_2^j\alpha)(s,\gamma(s))\mathrm{d}s.
\end{align*}
Now the proof is complete if we can show that
\begin{equation}\label{eq:pluggingIn}
\sum_{0<|I|\leq{r}}\int_{t_0}^{t}v_I(s)\,(f_I\alpha)(s,\gamma(s))\,\mathrm{d}s \ = \ \sum_{0<|I|\leq{r}}\int_{t_0}^{t}\frac{v_I(s)}{|I|}\,([f_I]\alpha)(s,\gamma(s))\,\mathrm{d}s.
\end{equation}
Since we assume that conditions~\ref{def:GDConvergence:1} and~\ref{def:GDConvergence:2} in \Cref{def:GDConvergence} are satisfied, the polynomial input~$\mathbf{v}$ satisfies the algebraic identity in \Cref{thm:algebraicIdentity}. It is shown in~\cite{Liu1997} that this algebraic identity implies~\cref{eq:pluggingIn}, simply by ``plugging in the vector fields $f_1,\ldots,f_m$ for the indeterminates $X_1,\ldots,X_m$''. This completes the proof.
\end{proof}
\begin{remark}\label{thm:DOConvergence}
The integral expansion~\cref{eq:IntegralExpansion:01} in \Cref{thm:IntegralExpansion} can be rewritten in terms of differential operators, which allows a comparison to the results in~\cite{Morin1999}. The authors of~\cite{Morin1999} consider the vector fields $f^j$ and $f^{\infty}$ on the right-hand sides of $\Sigma^j$ and $\Sigma^{\infty}$ as time-varying differential operators and introduce the notion of \emph{DO-convergence} (short for \emph{differential operator-convergence}); see Definition~6.1 in~\cite{Morin1999}. 
When we take the derivative of~\cref{eq:IntegralExpansion:01}, and use the fact that $\gamma\colon{J}\to{M}$ is a solution of $\Sigma^j$, we obtain
\[
(f^{j}\alpha)(s,\gamma(s)) \ = \ (f^{\infty}\alpha)(s,\gamma(s)) + ((\partial_t+f^j)(D_1^j\alpha))(s,\gamma(s)) + (D_2^j\alpha)(s,\gamma(s))
\]
for almost every $s\in{J}$. One can show that if $(\mathbf{u}^j)_j$ GD($r$)-converges uniformly to~$\mathbf{v}$, then the above equation implies that the sequence $(f^j)_j$ DO-converges to $f^\infty$ in the sense of~\cite{Morin1999}. However, the convergence result in~\cite{Morin1999} only ensures convergence of trajectories for a fixed initial condition, which is not sufficient for our purposes.
\end{remark}
\begin{remark}\label{thm:boundedAE}
We know from \Cref{thm:integrationByParts} that, for every multi-index~$I$ of length $0<|I|\leq{r}$, the Lie derivative of $f_I\alpha$ along $f_0$ exists almost everywhere along trajectories of~$\Sigma^j$. The subsequent convergence theorem, \Cref{thm:convergenceResult}, assumes certain boundedness properties of these Lie derivatives whenever they exist. For this reason, we modify the notion of local uniform boundedness with respect to a non\-negative function~$b$ on~$M$ from \Cref{sec:03} as follows. For every multi-index~$I$ of length $0<|I|\leq{r}$, we say (by a slight abuse of terminology) that $f_0(f_I\alpha)$ is \emph{locally uniformly bounded by a multiple of~$b$} if for every point $x_0\in{M}$, there exists a neighborhood $V\subseteq{M}$ of $x_0$ and a constant $c>0$ such that $|(f_0(f_I\alpha))(t,x)|\leq{c}\,b(x)$ for every $(t,x)\in\mathbb{R}\times{V}$ at which $f_0(f_I\alpha)$ exists.
\end{remark}
Using the terminology in~\Cref{thm:boundedAE}, we are ready to state our general convergence theorem. For the purpose of applications to non\-holonomic systems as in~\Cref{sec:06:02}, we allow a pseudo-distance function on~$M$ to measure distances between trajectories; see~\Cref{sec:03}. This pseudo-distance has to be the \emph{lift} of a locally Euclidean distance function by a suitable projection map~$\pi$. In particular, when $\pi$ is chosen as the identity map on~$M$, then we get convergence with respect to a distance function.
\begin{theorem}\label{thm:convergenceResult}
Suppose that $(\mathbf{u}^j)_j$ GD($r$)-converges uniformly to~$\mathbf{v}$. Let $\pi\colon{M}\to\tilde{M}$ be a smooth proper surjective submersion onto a smooth manifold $\tilde{M}$. Let $\tilde{d}$ be a locally Euclidean distance function on~$\tilde{M}$ and define a pseudo-distance function~$d$ on~$M$ by $d(x,x'):=\tilde{d}(\pi(x),\pi(x'))$ for all $x,x'\in{M}$. Let~$b$ be a non\-negative locally Lipschitz continuous function on~$M$ that is constant on the fibers of~$\pi$. Suppose that $f^{\infty}$ is~$\pi$-related to a time-varying vector field on~$\tilde{M}$. Suppose that for every $i=1,\ldots,m$, every multi-index $0<|I|\leq{r}$, and every $\alpha\in{C^{\infty}(M)}$ that is constant on the fibers of $\pi$, the time-varying functions $f_0\alpha$, $f_i\alpha$, $f_if_I\alpha$, $\partial_tf_I\alpha$, and $f_0f_I\alpha$ are locally uniformly bounded by a multiple of~$b$. Then, for every compact subset~$K$ of~$M$, the trajectories of $(\Sigma^j)_j$ converge in~$K$ on compact time intervals with respect to $(d,b)$ to the trajectories of $\Sigma^{\infty}$.
\end{theorem}
\begin{proof}
As in \Cref{sec:04}, we denote the flow of~$\Sigma^j$ and~$\Sigma^{\infty}$ by~$\Phi^{j}$ and~$\Phi^{\infty}$, respectively. Fix an arbitrary compact set $K\subseteq{M}$ and define $\tilde{K}:=\pi(K)\subseteq\tilde{M}$. Choose a compact neighborhood $\tilde{K}'\subseteq\tilde{M}$ of~$\tilde{K}$. Since~$\tilde{K}$ is a compact set in the interior of~$\tilde{K}'$, we can find a sufficiently small $\delta>0$ such that the $\delta$-neighborhood $\{\tilde{x}\in\tilde{M} \ | \ \tilde{d}(\tilde{x},\tilde{K})\leq\delta\}$ of~$\tilde{K}$ is contained in the interior of~$\tilde{K}'$. 
Since~$\pi$ is proper, the set $K':=\pi^{-1}(\tilde{K}')$ is also compact, and, by definition of~$d$, the  $\delta$-neighborhood $\{x\in{M} \ | \ d(x,K)\leq\delta\}$ of~$K$ is contained in the interior of~$K'$. In the following, we restrict all considerations to the compact sets~$\tilde{K}'$ and~$K'$. After choosing a suitable embedding, we may assume that~$\tilde{K}'$ is a subset of an embedded submanifold of~$\mathbb{R}^N$ for~$N$ sufficiently large. Note that the restriction of the Euclidean distance on~$\mathbb{R}^N$ to an embedded submanifold of~$\mathbb{R}^N$ is a locally Euclidean distance function. 
Since we are only interested in convergence within the compact set~$\tilde{K}'$, we may replace~$\tilde{d}$ by the restriction of the Euclidean distance on~$\mathbb{R}^N$ to~$\tilde{K}'$. 
Thus, the pseudo-distance function~$d$ on~$M$ is given by
\[
d(x,x') \ = \ \|\pi(x')-\pi(x)\|
\]
for all $x,x'\in{K'}$. Because of the hypothesis that $(f_0,f_1,\ldots,f_m)\in\operatorname{Sys}(M;m,r)$, we know that the time-varying vector fields~$f_0$ and~$[f_I]$ for $0<|I|\leq{r}$ are locally uniformly Lipschitz continuous. Moreover, by definition of a polynomial input, the component functions $v_I\colon\mathbb{R}\to\mathbb{R}$ of~$\mathbf{v}$ are bounded. Thus, the time-varying vector field~$f^{\infty}$ is also locally uniformly Lipschitz continuous. By hypothesis, $f^{\infty}$ is~$\pi$-related to a time-varying vector field~$\tilde{f}^{\infty}$ on~$\tilde{M}$. Since~$f^{\infty}$ is locally uniformly bounded and locally uniformly Lipschitz continuous, and since~$\pi$ is a surjective submersion, it follows that~$\tilde{f}^{\infty}$ is locally uniformly bounded and locally uniformly Lipschitz continuous as well. 
We conclude that there exists $L_{\infty}>0$ such that
\begin{equation}\label{eq:22052017}
\|(f^{\infty}\pi)(t,x')-(f^{\infty}\pi)(t,x)\| \ \leq \ L_{\infty}\,d(x,x')
\end{equation}
for every $t\in\mathbb{R}$ and all $x,x'\in{K'}$, where $f^{\infty}\pi$ is the $\mathbb{R}^N$-valued time-varying map that consists of the Lie derivatives of the components of~$\pi$ along $f^{\infty}$. 
A similar argument, using that~$b$ locally Lipschitz continuous and constant on the fibers of~$\pi$, shows that there exists $L_b>0$ such that
\begin{equation}\label{eq:23052017}
|b(x')-b(x)| \ \leq \ L_b\,d(x,x')
\end{equation}
for all $x,x'\in{K'}$. 

Now fix arbitrary $\varepsilon>0$ and $T>0$. After shrinking $\varepsilon$ sufficiently, we may suppose that we have $\varepsilon\max_{x\in{K'}}b(x)<\delta$. We have to prove that there exists a sequence index $j_0$ such that the following \emph{approximation property} holds: whenever $t_0\in\mathbb{R}$ and $x_0\in{K}$ are such that $\Phi^{\infty}(\cdot,t_0,x_0)$ stays in~$K$ on $[t_0,t_0+T]$, then, for every $j\geq{j_0}$, the trajectory $\Phi^{j}(\cdot,t_0,x_0)$ exists on $[t_0,t_0+T]$ and
\[
\big\|\pi(\Phi^{\infty}(t,t_0,x_0))-\pi(\Phi^{j}(t,t_0,x_0))\big\| \ \leq \ \varepsilon\,b(x_0)
\]
holds for every $t\in[t_0,t_0+T]$. The assumption that $(\mathbf{u}^j)_j$ GD($r$)-converges uniformly to~$\mathbf{v}$ guarantees the existence of a sequence of polynomial inputs $\mathbf{v}^j=(v^j_I)_{0<|I|\leq{r}}$ and a sequence of generalized differences $(\widetilde{UV}\vphantom{UV}^j_I)_{0<|I|\leq{r}}$ as in \Cref{def:generalizedDifference}. In this situation, \Cref{thm:IntegralExpansion} states that for every sequence index~$j$, every $t_0\in\mathbb{R}$, every $x_0\in{M}$, and every $t$ in the domain of $\Phi^{j}(\cdot,t_0,x_0)$, we have
\begin{subequations}\label{eq:24052017}
\begin{align}
\pi\big(\Phi^j(t,t_0,x_0)\big) \, = \ & \pi(x_0) + \int_{t_0}^{t}(f^{\infty}\pi)\big(s,\Phi^j(s,t_0,x_0)\big)\mathrm{d}s - (D_1^j\pi)\big(t_0,x_0) \label{eq:eq:24052017:A} \allowdisplaybreaks \\
& + (D_1^j\pi)\big(t,\Phi^j(t,t_0,x_0)\big) + \int_{t_0}^t(D_2^j\pi)\big(s,\Phi^j(s,t_0,x_0)\big)\mathrm{d}s \label{eq:24052017:B}
\end{align}
\end{subequations}
where the components of $D_1^j\pi$ and $D_2^j\pi$ are defined as in \Cref{thm:IntegralExpansion}. Using the boundedness assumptions of the theorem and the fact that the coefficient functions $v_I$ of the polynomial input~$\mathbf{v}$ are (by definition) bounded, it follows that there exists a sufficiently large $c_{\infty}>0$ such that
\begin{equation}\label{eq:26052017}
\|(f^{\infty}\pi)(s,x)\| \ \leq \ c_{\infty}\,b(x)
\end{equation}
for every $(s,x)\in\mathbb{R}\times{K'}$. Because of the uniform GD($r$)-convergence, the functions $\widetilde{UV}\vphantom{UV}^j_{I}$, $v_I-v_I^j$ and $u^j_i\,\widetilde{UV}\vphantom{UV}^j_{I}$ converge uniformly on $\mathbb{R}$ to $0$ as $j\to\infty$. Using again the boundedness assumptions of the theorem, we conclude that there exist sequences $(c^j_1)_j$, $(c^j_2)_j$ of positive real numbers that converge $0$ as $j\to\infty$ and satisfy
\begin{equation}\label{eq:270520171}
\|D_1^j\pi(s,x)\| \ \leq \ c^j_1\,b(x)
\end{equation}
for every sequence index~$j$ and every $(s,x)\in\mathbb{R}\times{K'}$, and
\begin{equation}\label{eq:270520172}
\|D_2^j\pi(s,x)\| \ \leq \ c^j_2\,b(x)
\end{equation}
for every sequence index~$j$, and every $(s,x)\in\mathbb{R}\times{K'}$ at which the Lie derivatives of the $f_I\alpha$ along $f_0$ exist. Using expansion \cref{eq:24052017} and estimates~\cref{eq:23052017,eq:26052017,eq:270520171,eq:270520172}, we obtain
\begin{align*}
b\big(\Phi^{j}(t,t_0,x_0)\big) \ \leq \ & (1+L_b\,c^j_1)\,b(x_0) + L_b\,c^j_1\,b\big(\Phi^{j}(t,t_0,x_0)\big) \allowdisplaybreaks \\
& + L_b\,(c_{\infty}+c^j_2)\,\int_{t_0}^{t}b\big(\Phi^{j}(s,t_0,x_0)\big)\,\mathrm{d}s
\end{align*}
for every sequence index~$j$, every $t_0\in\mathbb{R}$, every $x_0\in{K}$, and every $t\geq{t_0}$ as long as $\Phi^{j}(\cdot,t_0,x_0)$ stays in $K'$ on $[t_0,t]$. Since $(c_1^j)_j$ converges to $0$, there exists a sequence index $j_{0}$ such that $L_bc_1^j<1$ for every $j\geq{j_{0}}$. Then, Gronwall's inequality implies
\begin{equation}\label{eq:28052017}
b\big(\Phi^{j}(t,t_0,x_0)\big) \ \leq \ \mu^j_1\,b(x_0)\mathrm{e}^{\mu^j_2(t-t_0)}
\end{equation}
for $j\geq{j_{0}}$, where $\mu^j_1:=(1+L_b\,c^j_1)/(1-L_b\,c^j_1)$ and $\mu^j_2:=(c_{\infty}+c^j_2)\,L_b/(1-L_b\,c^j_1)$. On the other hand, expansion~\cref{eq:24052017} and estimates~\cref{eq:22052017,eq:26052017,eq:270520171,eq:270520172} yield
\begin{align*}
d\big(\Phi^{j}(t,t_0,x_0),\Phi^{\infty}(t,t_0,x_0)\big) \, \leq \ & c^j_1\,b(x_0) + c^j_2\,\int_{t_0}^{t}b\big(\Phi^{j}(s,t_0,x_0)\big)\,\mathrm{d}s \allowdisplaybreaks \\
& \ + L_{\infty}\int_{t_0}^{t}d\big(\Phi^{j}(s,t_0,x_0),\Phi^{\infty}(s,t_0,x_0)\big)\,\mathrm{d}s
\end{align*}
for every sequence index~$j$, every $t_0\in\mathbb{R}$, every $x_0\in{K}$, and every $t\geq{t_0}$ as long as $\Phi^{\infty}(\cdot,t_0,x_0)$ stays in~$K$ on $[t_0,t]$ and $\Phi^{j}(\cdot,t_0,x_0)$ stays in $K'$ on $[t_0,t]$. Using~\cref{eq:28052017} and again Gronwall's inequality, it follows that
\[
d\big(\Phi^{j}(t,t_0,x_0),\Phi^{\infty}(t,t_0,x_0)\big) \ \leq \ b(x_0)\,\Big(c^j_1+c^j_2\,\frac{\mu^j_1}{\mu^j_2}\big(\mathrm{e}^{\mu^j_2(t-t_0)}-1\big)\Big)\,\mathrm{e}^{L_{\infty}(t-t_0)}
\]
for $j\geq{j_0}$. By possibly increasing $j_0$, we can achieve that
\[
\Big(c^j_1+c^j_2\,\frac{\mu^j_1}{\mu^j_2}\big(\mathrm{e}^{\mu^j_2T}-1\big)\Big)\,\mathrm{e}^{LT} \ < \ \varepsilon
\]
for $j\geq{j_0}$. This choice of $j_0$ ensures the asserted convergence property. 
\end{proof}

\begin{remark}
As mentioned earlier, the proof of \Cref{thm:convergenceResult} relies manly on the integral expansion~\cref{eq:IntegralExpansion:01} from \Cref{thm:IntegralExpansion} and suitable estimates for its constituents. The same strategy is also applied in~\cite{Kurzweil1988,Liu1997}. In these papers convergence is proved for single trajectories with fixed initial conditions $(t_0,x_0)$, which is motivated by the motion planning problem. However, in the preliminary report~\cite{Liu19932} of~\cite{Liu19972} it is already indicated that uniform convergence in $(t_0,x_0)$ can be obtained by applying integration by parts and the Gronwall inequality. This is precisely what is done in the proof of \Cref{thm:convergenceResult}. A combination of integration by parts and Gronwall's lemma is also used in~\cite{Moreau2003,Duerr2013} to obtain similar convergence results for the case $r=2$.
\end{remark}
If the projection map $\pi$ in~\Cref{thm:convergenceResult} is chosen as the identity map on~$M$, then, in combination with \Cref{thm:LUASToPLUAS,thm:LUESToLUES}, we get the following consequence.
\begin{corollary}\label{thm:Duerr2013}
Suppose that $(\mathbf{u}^j)_j$ GD($r$)-converges uniformly to~$\mathbf{v}$. Let~$d$ be a locally Euclidean distance function on~$M$, and let~$E$ be a compact subset of~$M$.
\begin{enumerate}[label=(\alph*)]
	\item\label{thm:Duerr2013:A} If~$E$ is LUAS for $\Sigma^{\infty}$, then~$E$ is PLUAS for $(\Sigma^j)_j$.
	\item\label{thm:Duerr2013:B} Suppose that for every $i=1,\ldots,m$, every multi-index $0<|I|\leq{r}$, and every $\alpha\in{C^{\infty}(M)}$, the time-varying functions $f_0\alpha$, $f_i\alpha$, $f_if_I\alpha$, $\partial_tf_I\alpha$, and $f_0f_I\alpha$ are locally uniformly bounded by a multiple of the distance $x\mapsto{d(x,E)}$ to~$E$. If~$E$ is LUES for $\Sigma^{\infty}$, then there exists a sequence index $j_0$ such that~$E$ is LUES for $(\Sigma^j)_{j\geq{j_0}}$.
\end{enumerate}
\end{corollary}
Note that part~\ref{thm:Duerr2013:A} of \Cref{thm:Duerr2013} is a generalization of the main result in~\cite{Duerr2013} for Lie brackets of two vector fields.


\subsection{A convergence result for output feedback}\label{sec:convergence:03}
Let $m,r$ be positive integers with ${r\geq2}$, and let $(e_0,e_1,\ldots,e_m)\in\operatorname{Sys}(M;m,r)$. Let~$\psi$ be a non\-negative~$C^r$ function on~$M$ that is tangent to the fibers of~$\pi$. One may interpret this as a control-affine system with output $y$ given by~$\psi$, as depicted in~\Cref{fig:a}. For each $i=1,\ldots,m$, let $h_i\colon\mathbb{R}\to\mathbb{R}$ be a continuous function. Let $f_0:=e_0$, and for each $i=1,\ldots,m$, define a time-varying vector field~$f_i$ on~$M$ by $f_i(t,x):=h_i(\psi(x))\,e_i(t,x)$. Let $(\mathbf{u}^j)_j$ be a sequence of ordinary inputs $(u^j_i)_{i=1,\ldots,m}$. As in \Cref{sec:convergence:02}, we consider for each sequence index~$j$, the time-varying system
\[
\Sigma^j\colon\qquad\dot{x} \ = \ f_0(t,x) + \sum_{i=1}^{m}u^j_i(t)\,f_i(t,x).
\]
Let $(v_I)_{0<|I|\leq{r}}$ be a polynomial input of order~$\leq{r}$. The subsequent \Cref{thm:outputConvergenceResult} shows that under certain assumptions, the trajectories of the $\Sigma^{j}$ converge locally in the limit $j\to\infty$ to the trajectories of the time-varying system
\[
\Sigma^{\infty}\colon\qquad\dot{x} \ = \ f_0(t,x) + \sum_{0<|I|\leq{r}}^{m}\frac{v_I(t)}{|I|}[f_I](t,x).
\]
As in \Cref{thm:convergenceResult} we allow a pseudo distance function to measure distances between trajectories. Let $\pi\colon{M}\to\tilde{M}$ be a smooth proper surjective submersion onto a smooth manifold~$\tilde{M}$. Let~$\tilde{d}$ be a locally Euclidean distance function on~$\tilde{M}$, and define $d\colon{M}\times{M}\to\mathbb{R}$ by $d(x,x'):=\tilde{d}(\pi(x),\pi(x'))$. In this situation, the following result holds.
\begin{figure}
  \centering
  \label{fig:a}\includegraphics{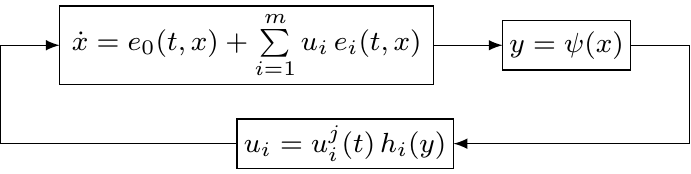}
  \caption{Interpretation of $(e_0,e_1,\ldots,e_m)\in\operatorname{Sys}(M;m,r)$ as a control-affine system. The control law $u_i=u^j_i(t)\,h_i(y)$ for $i=1,\ldots,m$ with output $y=\psi(x)$ leads to the closed-loop system $\Sigma^{j}$. Under the assumptions of \Cref{thm:outputConvergenceResult}, the trajectories of $\Sigma^j$ converge locally in the limit $j\to\infty$ to the trajectories of the system~$\Sigma^{\infty}$.}
  \label{fig:testfig}
\end{figure}
\begin{theorem}\label{thm:outputConvergenceResult}
Suppose that for each $i=1,\ldots,m$, one of the following two conditions is satisfied:
\begin{enumerate}[label=(\arabic*)]
	\item $e_i$ is tangent to the fibers of~$\pi$ and $h_i\equiv1$ on $(0,\infty)$, or 
	\item $h_i$ is of class $C^r$ on $(0,\infty)$ and for each $\nu=0,1,\ldots,r$, the $\nu$th derivative $h_i^{(\nu)}$ of $h_i$ has the property that $y^{\nu-1/2}h_i^{(\nu)}(y)$ remains bounded as $y\downarrow0$.
\end{enumerate}
Then, $(f_0,f_1,\ldots,f_m)\in\operatorname{Sys}(M;m,r)$. Additionally, suppose that $(\mathbf{u}^j)_j$ GD($r$)-con\-verges uniformly to~$\mathbf{v}$ and that the time-varying vector field on the right-hand side of $\Sigma^{\infty}$ is~$\pi$-related to a time-varying vector field on $\tilde{M}$. Then, the following holds.
\begin{enumerate}[label=(\alph*)]
	\item\label{thm:outputConvergenceResult:A} For every compact set $K\subseteq{M}$, the trajectories of $(\Sigma^j)_j$ converge in~$K$ on compact time intervals with respect to $(d,b)$ to the trajectories of $\Sigma^{\infty}$, where $b:\equiv1$ on $M$.
	\item\label{thm:outputConvergenceResult:B} If the drift $e_0$ is tangent to the fibers of~$\pi$ or if $e_0$ locally uniformly bounded by a multiple of $\sqrt{\psi}$, then, for every compact set $K\subseteq{M}$, the trajectories of $(\Sigma^j)_j$ converge in~$K$ on compact time intervals with respect to $(d,\sqrt{\psi})$ to the trajectories of $\Sigma^{\infty}$.
\end{enumerate}
\end{theorem}
The statements in \Cref{thm:outputConvergenceResult} can be identified as special cases of \Cref{thm:convergenceResult}. For this purpose, one has to show that the time-varying vector fields $f_0,f_1,\ldots,f_m$ satisfy the conditions~\ref{item:def:Sys:i}-\ref{item:def:Sys:iii} in \Cref{def:Sys} and that the boundedness assumptions in \Cref{thm:convergenceResult} for the Lie derivatives are satisfied. This turns out to be quite technical since the vector fields $f_1,\ldots,f_m$ are not necessarily differentiable on the zero set of~$\psi$. The proof of \Cref{thm:outputConvergenceResult} is given in \Cref{sec:outputConvergenceResult}.


\section{Applications to the design of output-feedback}\label{sec:06}
Now we apply the results of~\Cref{sec:04,sec:05} to the problem of stabilizing control-affine systems by means output feedback at states where their output functions attains a minimum value.
\subsection{The single-bracket case}\label{sec:06:01}
Let~$M$ be a smooth manifold, and let $p$ be a positive integer. Let $(g_0,g_1,\ldots,g_p)\in\operatorname{Sys}(M;p,2)$ as in \Cref{def:Sys}. We consider the control-affine system
\begin{align}
\dot{x} & \ = \ g_0(t,x) + \sum_{k=1}^{p}u_k\,g_k(t,x), \label{eq:ESCInitialSystem} \allowdisplaybreaks \\
y & \ = \ \psi(x)  \nonumber
\end{align}
whose output $y$ is determined by a non\-negative $C^2$ function~$\psi$ on~$M$. Our goal is to find output feedback of the form $u_k=u_k(t,y)$ for $k=1,\ldots,p$ that stabilizes~\cref{eq:ESCInitialSystem} around local minima of~$\psi$. It is important to note, that we restrict our considerations to a nonnegative output function. This assumption is necessary because of the particular form of the control law that we present in the following. More general, if at least a lower bound of the output function is known, then one can remove this offset from the output value to obtain a situation as above. 

Since we are interested in minima of~$\psi$, an intuitive attempt is to steer~\cref{eq:ESCInitialSystem} constantly into a descent direction of~$\psi$. Note that each of the tangent vectors $-(g_k\psi)(t,x)g_k(t,x)$, $k=1,\ldots,p$, points into such a direction, where $g_k\psi$ denotes the Lie derivative of~$\psi$ along~$g_k$; see~\Cref{sec:03}. Thus, for each $k=1,\ldots,p$, the choice
\begin{equation}\label{eq:naiveControl}
u_k=-\lambda_k(t)\,(g_k\psi)(t,x)
\end{equation}
would be a promising candidate for our purpose, where $\lambda_k\colon\mathbb{R}\to\mathbb{R}$ is a suitably chosen positive, Lebesgue measurable and bounded function. Then, the closed-loop system reads
\begin{equation}\label{eq:ESCLimitSystem}
\Sigma^{\infty}\colon\qquad\dot{x} \ = \ g_0(t,x) - \sum_{k=1}^{p}\lambda_k(t)\,(g_k\psi)(t,x)\,g_k(t,x).
\end{equation}
By choosing the $\lambda_k$ sufficiently large, one can try to compensate the influence of the drift. However, the implementation of~\cref{eq:naiveControl} requires information about the Lie derivatives of~$\psi$ at any given time $t\in\mathbb{R}$ and every state $x\in{M}$, which is not accessible in our setup. To circumvent this problem, we use an approach that is recently studied in the context of extremum seeking control, see, e.g., \cite{Duerr2013,Duerr2014,Scheinker20132,Scheinker2014}. The idea is to find a sequence $(u^j_k(t,y))_j$ of output feedback control laws for~\cref{eq:ESCInitialSystem} so that the trajectories of the corresponding sequence $(\Sigma^j)_j$ of closed-loop systems converge locally, in the sense of \Cref{def:convergence}, to the trajectories of~$\Sigma^{\infty}$ as $j\to\infty$. This way it is possible to carry over stability properties of~$\Sigma^{\infty}$ to the approximating sequence $(\Sigma^j)_j$; cf. \Cref{thm:LUASToPLUAS,thm:LUESToLUES}. Following the strategy in the above papers, we choose a sequence $(u^j_k(t,y))_j$ of the form
\begin{equation}\label{eq:ESCcontrolLaw}
u^j_k(t,y) \ := \ u^j_{2k-1}(t)\,h_{\text{s}}(y) + u^j_{2k}(t)\,h_{\text{c}}(y)
\end{equation}
for every sequence index~$j$ and every $k=1,\ldots,p$ with certain Lebesgue measurable and bounded functions $u^j_1,\ldots,u^j_{2p}\colon\mathbb{R}\to\mathbb{R}$ and certain functions $h_{\text{s}},h_{\text{c}}\colon\mathbb{R}\to\mathbb{R}$, which are specified later. When we insert $u_k=u^j_k(t,\psi(x))$ into~\cref{eq:ESCInitialSystem}, then we obtain the closed-loop system
\[
\Sigma^j\colon\qquad\dot{x} \ = \ f_0(t,x) + \sum_{i=1}^{m}u^j_i(t)\,f_i(t,x),
\]
where $m:=2p$, $f_0:=g_0$, and the time-varying vector fields $f_1,\ldots,f_m$ are given by
\begin{align*}
f_{2k-1}(t,x) & \ := \ h_{\text{s}}(\psi(x))\,g_k(t,x), \allowdisplaybreaks \\
f_{2k}(t,x) & \ := \ h_{\text{c}}(\psi(x))\,g_k(t,x)
\end{align*}
for $k=1,\ldots,p$. If we can choose $h_{\text{s}},h_{\text{c}}$ in such a way that the Lie bracket of $f_{2k-1},f_{2k}$ is given by
\begin{equation}\label{eq:magicLieBracket}
[f_{2k-1},f_{2k}](t,x) \ := \ -(g_{k}\psi)(t,x)\,g_k(t,x)
\end{equation}
for $k=1,\ldots,p$, then~\cref{eq:ESCLimitSystem} can be written as
\[
\Sigma^{\infty}\colon\qquad \dot{x} \ = \ f_0(t,x) + \sum_{0<|I|\leq{2}}\frac{v_I(t)}{|I|}\,[f_I](t,x)
\]
with $[f_I]$ as in \Cref{def:Sys} and $\mathbf{v}=(v_I)_{0<|I|\leq{2}}$ is the polynomial input (cf. \Cref{def:inputs}) whose coefficient functions are given by
\begin{equation}\label{eq:ESCPolynomialInput}
v_{I} \ := \ \left\{ \begin{tabular}{cl} $+\lambda_k$ & for $I=(2k-1,2k)$, \\ $-\lambda_k$ & for $I=(2k,2k-1)$, \\ $0$ & otherwise. \end{tabular} \right.
\end{equation}
Suppose that we can choose the functions $u^j_i$ in~\cref{eq:ESCcontrolLaw} such that the sequence of ordinary inputs $\mathbf{u}^j=(u^j_1,\ldots,u^j_{m})$ GD($2$)-converges uniformly to~$\mathbf{v}$. For instance, if the $\lambda_k$ are bounded $C^1$ functions with bounded derivatives, then one can choose the sequence $(\mathbf{u}^j)_j$ as in the following statement (the proof is given in \Cref{sec:ESCSinusoids}).
\begin{proposition}\label{thm:ESCSinusoids}
Let $\omega_1,\ldots,\omega_p$ be pairwise distinct positive real numbers. Suppose that for each $k=1,\ldots,p$, the function $\lambda_k\colon\mathbb{R}\to\mathbb{R}$ is bounded, of class $C^1$, and also its derivative is bounded. Then
\begin{subequations}\label{eq:ESCSinusoids}
\begin{align}
u^j_{2k-1}(t) & \ := \ \sqrt{2\omega_kj}\,\lambda_k(t)\,\cos(\omega_kjt) \label{eq:ESCSinusoids:A}, \allowdisplaybreaks \\
u^j_{2k}(t) & \ := \ \sqrt{2\omega_kj}\,\sin(\omega_kjt) \label{eq:ESCSinusoids:B},
\end{align}
\end{subequations}
for $k=1,\ldots,p$, defines a sequence of ordinary inputs $\mathbf{u}^j=(u^j_1,\ldots,u^j_{2p})$ that GD($2$)-converges uniformly to the polynomial input $\mathbf{v}=(v_I)_{0<|I|\leq{2}}$ defined by~\cref{eq:ESCPolynomialInput}.
\end{proposition}
Next, we define continuous functions $h_{\text{s}},h_{\text{c}}\colon\mathbb{R}\to\mathbb{R}$ by
\begin{subequations}\label{eq:hshc}
\begin{align}
h_{\text{s}}(y) & \ := \ \sqrt{y}\sin(\log(y)), \label{eq:hshc:A} \allowdisplaybreaks \\
h_{\text{c}}(y) & \ := \ \sqrt{y}\cos(\log(y)) \label{eq:hshc:B}
\end{align}
\end{subequations}
for $y>0$ and by $h_{\text{s}}(y):=h_{\text{c}}(y):=0$ for $y\leq{0}$. It is straight forward to check that for every non\-negative integer $\nu$, the $\nu$th derivatives $h_{\text{s}}^{(\nu)},h_{\text{c}}^{(\nu)}$ of $h_{\text{s}}$, $h_{\text{c}}$ exist on $(0,\infty)$, and that $y^{\nu-1/2}h_{\text{s}}^{(\nu)}(y)$ and $y^{\nu-1/2}h_{\text{c}}^{(\nu)}(y)$ remain bounded as $y\downarrow0$. Thus, we are precisely in the situation of \Cref{thm:outputConvergenceResult} (to see this, define~$\pi$ as the identity map, and let $h_{2k-1}:=h_{\text{s}}$, $h_{2k}:=h_{\text{c}}$, $e_{2k-1}:=e_{2k}:=g_k$ for $k=1,\ldots,p$). In particular, \Cref{thm:outputConvergenceResult} ensures that system \cref{eq:ESCInitialSystem} under control law~\cref{eq:ESCcontrolLaw} has the existence and uniqueness properties of solutions, and that Lie brackets of the $f_i$ exist on the entire manifold. A direct computation shows that~\cref{eq:magicLieBracket} is actually satisfied.
\begin{remark}
This is not the only possible choice of $h_{\text{c}},h_{\text{s}}$, see, e.g.,~\cite{Duerr2013,Scheinker20132,Scheinker2014,Grushkovskaya2017} for different approaches. However, \cref{eq:ESCcontrolLaw} with \cref{eq:hshc} is the only known output-feedback so far that can induce exponential stability (see \Cref{thm:ESCresult} below).
\end{remark}
By \Cref{thm:outputConvergenceResult}, the trajectories of $(\Sigma^{j})_j$ converge to the trajectories of~$\Sigma^{\infty}$ within compact sets and on compact time intervals as $j\to\infty$ with respect to any locally Euclidean distance function~$d$ on~$M$. We also now from \Cref{thm:LUASToPLUAS,thm:LUESToLUES} that this convergence is strong enough to transfer stability properties of~$\Sigma^{\infty}$ to the approximating sequence~$(\Sigma^{j})_j$. This leads to the following result.
\begin{corollary}\label{thm:ESCresult}
Let~$M$ be a smooth manifold, and let $p$ be a positive integer. Let $(g_0,g_1,\ldots,g_p)\in\operatorname{Sys}(M;p,2)$. Let~$\psi$ be a nonnegative~$C^2$ function on~$M$. For $k=1,\ldots,p$, let $\lambda_k\colon\mathbb{R}\to\mathbb{R}$ be a positive, Lebesgue measurable, and bounded function. Suppose that $(\mathbf{u}^j)_j$ is a sequence of ordinary inputs $\mathbf{u}^j=(u^j_1,\ldots{u^j_{2p}})$ that GD($2$)-converges uniformly to the polynomial input~$\mathbf{v}$ of order $\leq{2}$ given by~\cref{eq:ESCPolynomialInput}. For each sequence index~$j$, let~$\Sigma^j$ denote the system~\cref{eq:ESCInitialSystem} under the control law $u_k=u^j_k(t,\psi(x))$ given by~\cref{eq:ESCcontrolLaw,eq:hshc}. Let~$E$ be a nonempty and compact subset of~$M$. Then, the following implications hold.
\begin{enumerate}[label=(\alph*)]
	\item\label{thm:ESCresult:A} If~$E$ is LUAS for~\cref{eq:ESCLimitSystem}, then~$E$ is PLUAS for $(\Sigma^{j})_j$.
	\item\label{thm:ESCresult:B} If $E=\psi^{-1}(0)$ is the zero set of $\psi$, if $g_0$ is locally uniformly bounded by $\sqrt{\psi}$, and if~$E$ is LUES for~\cref{eq:ESCLimitSystem}, then there exists a sequence index $j_0$ such that~$E$ is LUES for $(\Sigma^j)_{j\geq{j_0}}$.
\end{enumerate}
\end{corollary}
\begin{proof}
Statement~\ref{thm:ESCresult:A} follows immediately from \Cref{thm:outputConvergenceResult}~\ref{thm:outputConvergenceResult:A} and \Cref{thm:LUASToPLUAS}. To prove~\ref{thm:ESCresult:B}, we note that $\sqrt{\psi}$ is locally Lipschitz continuous (see \Cref{thm:LipschitzDerivative} in the appendix). This and the compactness of~$E$ imply that $\sqrt{\psi}$ is locally uniformly bounded by a multiple of the distance $x\mapsto{d(x,E)}$ to the set~$E$, where~$d$ is an arbitrarily chosen locally Euclidean distance function on~$M$. 
Therefore statement~\ref{thm:ESCresult:B} follows immediately from \Cref{thm:outputConvergenceResult}~\ref{thm:outputConvergenceResult:B} and \Cref{thm:LUESToLUES}.
\end{proof}
\begin{example}\label{thm:linearSystem}
We consider a linear system
\begin{equation}\label{eq:linearSystem}
\dot{x} \ = \ Ax + Bu,
\end{equation}
with quadratic output $y=\psi(x):=\|x\|^2$, where $A\in\mathbb{R}^{n\times{n}}$ is arbitrary and $B\in\mathbb{R}^{n\times{p}}$ is assumed to have rank~$n$. The system state of~\cref{eq:linearSystem} is treated as an unknown quantity. Our goal is to find output feedback so that the origin becomes exponentially stable for~\cref{eq:linearSystem}. Since~\cref{eq:linearSystem} is of the form~\cref{eq:ESCInitialSystem}, we can apply \Cref{thm:ESCresult}. We choose the functions $\lambda_1,\ldots,\lambda_p$ to be identically equal to a positive constant $\lambda$. Then the limit system~\cref{eq:ESCLimitSystem} reads
\[
\Sigma^{\infty}\colon\qquad \dot{x} \ = \ (A-2\lambda{BB^{\top}})x.
\]
Since $BB^{\top}$ is positive definite, we can choose $\lambda$ sufficiently large so that the origin is exponentially stable for $\Sigma^{\infty}$. Now, we can use for example \Cref{thm:ESCSinusoids} to obtain a sequence $(\mathbf{u}^j)_j$ that GD(2)-converges uniformly to $\mathbf{v}$. By \Cref{thm:ESCresult}~\ref{thm:ESCresult:B}, there exists a sequence index $j_0$ such that for every $j\geq{j_0}$, the origin is LUES for~\cref{eq:linearSystem} under the control law $u_k=u^j_k(t,\psi(x))$ for $k=1,\ldots,p$, where~$u^j_k(t,\psi(x))$ is given by~\cref{eq:ESCcontrolLaw,eq:hshc}. In fact, one can prove that in this particular case of a linear system with quadratic output, control law~\cref{eq:ESCcontrolLaw} induces global uniform exponential stability for~$j$ sufficiently large.

We conclude this example by presenting numerical data. For the sake of simplicity, we let $p=n=A=B=\lambda=1$. We do not choose the sequence of ordinary inputs from \Cref{thm:ESCSinusoids}, but instead use the following alternative. Let $\operatorname{saw}\colon\mathbb{R}\to\mathbb{R}$ be the \emph{sawtooth wave} given by
\[
\operatorname{saw}(t) \ := \ 2(t-\lfloor{t}\rfloor)-1,
\]
where $\lfloor{t}\rfloor$ denotes the greatest integer $\leq{t}$. Using the same strategy as in the proof of \Cref{thm:ESCSinusoids}, it is easy to check that
\begin{subequations}\label{eq:sawtooth}
\begin{align}
u^j_1(t) & \ := \ 10\sqrt{j/\pi}\,\operatorname{saw}(jt+1/4), \allowdisplaybreaks \\
u^j_2(t) & \ := \ 10\sqrt{j/\pi}\,\operatorname{saw}(jt)
\end{align}
\end{subequations}
defines a sequence of ordinary inputs $\mathbf{u}^j=(u^j_1,u^j_2)$ that GD(2)-converges uniformly to the polynomial input~$\mathbf{v}$ given by~\cref{eq:ESCPolynomialInput}. \Cref{fig:sawtooth:a} shows the trajectory of~\cref{eq:linearSystem} under~\cref{eq:ESCcontrolLaw} for $j=1$ with initial condition $x(0)=1$. It turns out that this choice of~$j$ is already sufficient to ensure exponential converge to the origin. When we increase~$j$, then it becomes visible that the trajectories converge locally to the trajectories of the limit system $\Sigma^{\infty}$, which is simply $\dot{x}=-x$.
\begin{figure}
\centering
\subfloat[$j=1$]{\label{fig:sawtooth:a}\includegraphics{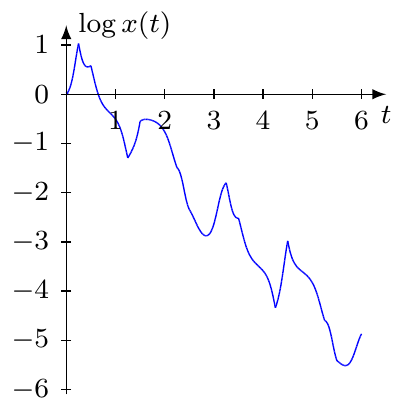}}
\subfloat[$j=10$]{\label{fig:sawtooth:b}\includegraphics{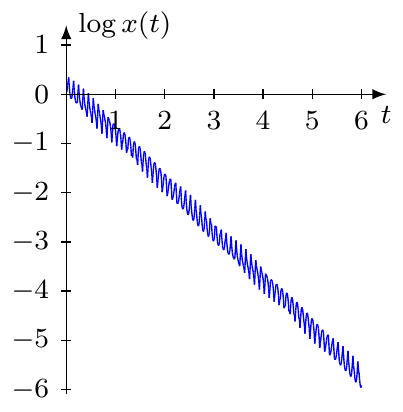}}
\subfloat[$j=100$]{\label{fig:sawtooth:c}\includegraphics{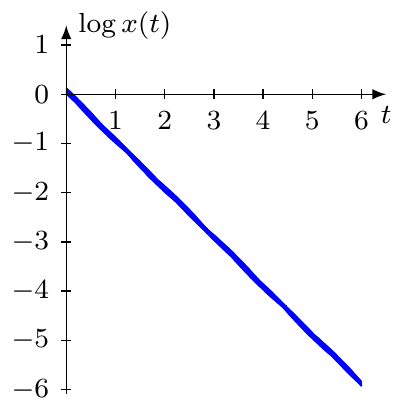}}
\caption{Trajectories of~\cref{eq:linearSystem} for $A=B=1$ under control law~\cref{eq:ESCcontrolLaw} with the phase shifted sawtooth waves in~\cref{eq:sawtooth} for initial condition $x(0)=1$ and different values of the sequence index~$j$.}
\label{fig:sawtooth}
\end{figure}
\end{example}

\subsection{Distance-based formation control for unicycles}\label{sec:06:02}
The approach in the previous subsection allows us to get access to descent directions of the output function~$\psi$ along the control vector fields. Next, we consider a particular example in which we can also get access to descent directions of~$\psi$ along Lie brackets of the control vector fields. An abridged version of this example is also presented in~\cite{Suttner2017}. Now we give a detailed explanation. 

We consider a multi-agent system of $N$ unicycles. In standard local coordinates, their kinematic equations are given by
\begin{align}\label{eq:unicycleKinematics}
\dot{x}_{\nu} & \ = \ u_{\nu,\text{t}}\,\cos\theta_{\nu}, & \dot{y}_{\nu} & \ = \ u_{\nu,\text{t}}\,\sin\theta_{\nu}, & \dot{\theta}_{\nu} & \ = \ u_{\nu,\text{r}}
\end{align}
for $\nu=1,\ldots,N$, where $p_{\nu}=(x_{\nu},y_{\nu})\in\mathbb{R}^2$ is the position of the $\nu$th unicycle, $\theta_{\nu}\in\mathbb{R}$ is an angle to describe the orientation, and $u_{\nu,\text{t}}$ and $u_{\nu,\text{r}}$ are input channels for the translational and rotational velocity, respectively. Suppose that each agent is equipped with a sensor to measure distances to other members of the team. The admissible measurements are described by an undirected graph $(\mathcal{V},\mathcal{E})$ with a set of nodes $\mathcal{V}=\{1,\ldots,N\}$ and an edge set $\mathcal{E}$ whose elements are sets of the form $\{\nu,\nu'\}$ with $\nu,\nu'\in\mathcal{V}$ and $\nu\neq\nu'$. In the following, we simply write $\nu\nu'$ for an edge instead of $\{\nu,\nu'\}$. Whenever $\nu\nu'$ is an edge of $(\mathcal{V},\mathcal{E})$, it means that agents $\nu,\nu'$ can measure their Euclidean distance $\|p_{\nu'}-p_{\nu}\|$ and share the measurement information with the other members of the team. We emphasize that the agents can only measure distances $\|p_{\nu'}-p_{\nu}\|$ but not relative positions $(p_{\nu'}-p_{\nu})$. The collective goal is to reach a configuration in which for every edge $\nu\nu'\in\mathcal{E}$, the Euclidean distance between the agents $\nu,\nu'$ is equal to a prescribed value $d_{\nu\nu'}>0$. We assume that the desired distances $d_{\nu\nu'}$ are \emph{feasible}, i.e., there exist $p_1,\ldots,p_{N}\in\mathbb{R}^{2N}$ such that $\|p_{\nu'}-p_{\nu}\|=d_{\nu\nu'}$ for every $\nu\nu'\in\mathcal{E}$.

The state manifold of the team of unicycles is the $N$-fold product $\operatorname{SE}(2)^N$ of the special Euclidean group $\operatorname{SE}(2)=\mathbb{R}^2\times\mathbb{S}$ in two dimensions. Note that~\cref{eq:unicycleKinematics} is a local representation of the driftless control-affine system
\begin{equation}\label{eq:unicycleKinematics2}
\dot{x} \ = \ \sum_{\nu=1}^{N}\big(u_{\nu,\text{t}}\,g_{\nu,\text{t}}(x) + u_{\nu,\text{r}}\,g_{\nu,\text{r}}(x)\big),
\end{equation}
where $g_{\nu,\text{t}}$ and $g_{\nu,\text{r}}$ are vector fields on $\operatorname{SE}(2)^N$ to the describe the directions of the translational and rotational velocity, respectively. The local representations of $g_{\nu,\text{t}},g_{\nu,\text{r}}$ are given in standard coordinates by
\begin{align*}
g_{\nu,\text{t}} & \ = \ \cos\theta_{\nu}\frac{\partial}{\partial{x_\nu}} + \sin\theta_{\nu}\frac{\partial}{\partial{y_\nu}}, \allowdisplaybreaks \\
g_{\nu,\text{r}} & \ = \ \frac{\partial}{\partial\theta}.
\end{align*}
Let $\pi=(\pi_1,\ldots,\pi_N)\colon\operatorname{SE}(2)^N\to\mathbb{R}^{2N}$ be the projection that maps every system state $x\in\operatorname{SE}(2)^N$ to the corresponding vector $p=(p_1,\ldots,p_N)\in\mathbb{R}^{2N}$ of positions $p_1,\ldots,p_{N}\in\mathbb{R}^{2}$. Then, we can rephrase the object as follows: we are interested in a distanced-based control law that stabilizes~\cref{eq:unicycleKinematics2} around the set
\[
E \ := \ \big\{x\in{M} \ \big| \ \forall\nu\nu'\in\mathcal{E}\colon\ \|\pi_{\nu'}(x)-\pi_{\nu}(x)\|=d_{\nu\nu'}\big\}.
\]
If we project~$E$ onto $\mathbb{R}^{2N}$, we get the set
\[
\tilde{E} \ := \ \pi(E) \ = \ \big\{(p_1,\ldots,p_{N})\in\mathbb{R}^{2N} \ \big| \ \forall\nu\nu'\in\mathcal{E}\colon\ \|p_{\nu'}-p_{\nu}\|=d_{\nu\nu'}\big\}.
\]
This set $\tilde{E}$ also appears in the context of formation control for systems of $N$ point agents in $\mathbb{R}^2$ with kinematic equations $\dot{p}_{\nu}=u_{\nu}$ for $\nu=1,\ldots,N$. A well-established approach (see \cite{Krick2009,Dorfler2009}) to stabilize the system of point agents around the set $\tilde{E}$ is the negative gradient control law $u_{\nu}=-\nabla_{p_{\nu}}\tilde{\psi}(p)$, where the potential function $\tilde{\psi}\colon\mathbb{R}^{2N}\to\mathbb{R}$ is defined by
\[
\tilde{\psi}(p) \ := \ \frac{1}{4}\sum_{\nu\nu'\in\mathcal{E}}\big(\|p_{\nu'}-p_{\nu}\|^2-d_{\nu\nu'}^{2}\big)^2.
\]
This leads to the closed loop system
\begin{equation}\label{eq:Mensa}
\tilde{\Sigma}^{\infty}\colon\qquad\dot{p} \ = \ -\nabla\tilde{\psi}(p)
\end{equation}
on $\mathbb{R}^{2N}$. It is known (see~\cite{Oh2014}) that $\tilde{E}$ is locally asymptotically stable for~$\tilde{\Sigma}^{\infty}$. Moreover (see~\cite{Park2014,Sun2016}), if $(\mathcal{E},\mathcal{V})$ satisfies certain rigidity conditions, then $\tilde{E}$ is locally exponentially stable for~$\tilde{\Sigma}^{\infty}$. However, the implementation of the negative control law
\[
u_{\nu} \ = \ -\nabla_{p_{\nu}}\tilde{\psi}(p) \ = \ \sum_{\nu\nu'\in\mathcal{E}}(p_{\nu'}-p_{\nu})\,\big(\|p_{\nu'}-p_{\nu}\|^2-d_{\nu\nu'}^{2}\big)
\]
requires information about the relative positions $(p_{\nu'}-p_{\nu})$, which are not accessible in our setup. In the following, we present a purely distance-based control law so that the team of unicycles acts like a team of point agents under the negative gradient control law.

Now we lift~$\tilde{\Sigma}^{\infty}$ to the state manifold $\operatorname{SE}(2)^N$ of the unicycle system. For this purpose, we define the non\-negative smooth function $\psi:=\tilde{\psi}\circ\pi$ on $\operatorname{SE}(2)^N$. Moreover, for every $\nu=1,\ldots,N$, we define the Lie bracket
\[
g_{\nu,\text{p}} \ := \ [g_{\nu,\text{r}},g_{\nu,\text{t}}]
\]
on $\operatorname{SE}(2)^N$. In local coordinates $g_{\nu,\text{p}}$ is given by
\[
g_{\nu,\text{p}} \ = \ -\sin\theta_{\nu}\frac{\partial}{\partial{x_\nu}} + \cos\theta_{\nu}\frac{\partial}{\partial{y_\nu}}.
\]
To give an interpretation: $g_{\nu,\text{p}}$ points perpendicular to the alignment of the $\nu$th unicycle. A direct computation shows that
\[
\sum_{\nu=1}^{N}\big((g_{\nu,\text{t}}\psi)(x)\,(g_{\nu,\text{t}}\pi)(x) + (g_{\nu,\text{p}}\psi)(x)\,(g_{\nu,\text{p}}\pi)(x)\big) \ = \ \nabla\tilde{\psi}(\pi(x))
\]
for every $x\in\operatorname{SE}(2)^N$, where $(g_{\nu,\text{t}}\pi)$ and $(g_{\nu,\text{p}}\pi)$ are the componentwise Lie derivatives of~$\pi$ along $g_{\nu,\text{t}}$ and $g_{\nu,\text{p}}$, respectively. Thus, $\tilde{\Sigma}^{\infty}$ can be lifted to the system
\begin{equation}\label{eq:unicycleLimitSystem}
\Sigma^{\infty}\colon\qquad\dot{x} \ = \ -\sum_{\nu=1}^{N}\big((g_{\nu,\text{t}}\psi)(x)\,g_{\nu,\text{t}}(x) + (g_{\nu,\text{p}}\psi)(x)\,g_{\nu,\text{p}}(x)\big)
\end{equation}
on $\operatorname{SE}(2)^N$. It follows that $E=\psi^{-1}(0)$ is locally asymptotically stable for~$\Sigma^{\infty}$. Moreover,~$E$ is locally exponentially stable for~$\Sigma^{\infty}$ if and only if $\tilde{E}$ is locally exponentially stable for $\tilde{\Sigma}^{\infty}$. On $\operatorname{SE}(2)^{N}$, we use the pseudo-distance function~$d$ that is induced by the Euclidean distance on $\mathbb{R}^{2N}$ through~$\pi$, i.e.,
\[
d(x,x') \ := \ \big\|\pi(x')-\pi(x)\big\|.
\]
This pseudo-distance only takes into account the positions of the unicycles, but not their orientations.

We return to the problem of stabilizing the system~\cref{eq:unicycleKinematics2} of unicycles around the set $E:=\psi^{-1}(0)$ by using only distance information. Note that in order to compute~$\psi$ at a given point $x\in{M}$, we do not need information about the entire system state but only the distances $\|\pi_{\nu'}(x)-\pi_{\nu}(x)\|$ for $\nu\nu'\in\mathcal{E}$. Thus, a control law which depends only on the values of~$\psi$ is distanced-based. From now on, we treat~$\psi$ as the output function of the control-affine system~\cref{eq:unicycleKinematics2}. This allows us to use basically the same approach as in the previous subsection: we determine a sequence of output feedback control laws such that the trajectories of the corresponding closed-loop systems converge locally to the trajectories of $\Sigma^{\infty}$. For this reason, we need the following technical result (the proof is given in \Cref{sec:unicycleSinusoids}).
\begin{proposition}\label{thm:unicycleSinusoids}
Let $N$ be a positive integer and let $m:=3N$. Let $\kappa_{\nu}$ denote the $\nu$th prime number, i.e., $\kappa_1=2$, $\kappa_2=3$, and so on. Define
\begin{subequations}\label{eq:unicycleSinusoids:01}
\begin{align}
\omega_{\nu,1} & \ := \ \sqrt{\kappa_{\nu+1}}\,\big(3+2\sqrt{2}\big), \label{thm:unicycleSinusoids:01:A} \allowdisplaybreaks \\
\omega_{\nu,2} & \ := \ \sqrt{\kappa_{\nu+1}}, \label{thm:unicycleSinusoids:01:B} \allowdisplaybreaks \\
\omega_{\nu,3} & \ := \ \big(\omega_{1}+\omega_{2}\big)/2 \ = \ \sqrt{\kappa_{\nu+1}}\,\big(2+\sqrt{2}\big) \label{thm:unicycleSinusoids:01:C}
\end{align}
\end{subequations}
for $\nu=1,\ldots,N$. Moreover, define a sequence of ordinary inputs $\mathbf{u}^j=\sum_{i=1}^{m}u^j_iX_i$ by
\begin{subequations}\label{eq:unicycleSinusoids:02}
\begin{align}
u^j_{3\nu-2}(t) & \ := \ (\omega_{\nu,1}\,j)^{3/4}\,\cos(\omega_{\nu,1}\,j\,t), \label{thm:unicycleSinusoids:02:A} \allowdisplaybreaks \\
u^j_{3\nu-1}(t) & \ := \ (\omega_{\nu,2}\,j)^{3/4}\,\sin(\omega_{\nu,2}\,j\,t), \label{thm:unicycleSinusoids:02:B} \allowdisplaybreaks \\
u^j_{3\nu}(t) & \ := \ {2^{13/8}}\,(\omega_{\nu,3}\,j)^{3/4}\,\cos(\omega_{\nu,3}\,j\,t) \label{thm:unicycleSinusoids:02:C}
\end{align}
\end{subequations}
for $\nu=1,\ldots,N$. Let $\mathbf{v}=\sum_{0<|I|\leq{4}}v_IX_I$ be the constant polynomial input of order $\leq{4}$ whose non\-zero coefficients are given by
\begin{subequations}\label{eq:unicycleSinusoids:03}
\begin{align}
v_{(1,2,3,3)_{\nu}}(t) & \ = \ -\frac{1}{2}+\frac{1}{\sqrt{2}}, &
v_{(1,3,2,3)_{\nu}}(t) & \ = \ +2-\sqrt{2}, \label{thm:unicycleSinusoids:04:A} \allowdisplaybreaks \\
v_{(1,3,3,2)_{\nu}}(t) & \ = \ -2, &
v_{(2,1,3,3)_{\nu}}(t) & \ = \ +\frac{1}{2}+\frac{1}{\sqrt{2}}, \label{thm:unicycleSinusoids:04:B} \allowdisplaybreaks \\
v_{(2,3,1,3)_{\nu}}(t) & \ = \ -2-\sqrt{2}, &
v_{(2,3,3,1)_{\nu}}(t) & \ = \ +2, \label{thm:unicycleSinusoids:04:C} \allowdisplaybreaks \\
v_{(3,1,2,3)_{\nu}}(t) & \ = \ -1, &
v_{(3,1,3,2)_{\nu}}(t) & \ = \ +2+\sqrt{2}, \label{thm:unicycleSinusoids:04:D} \allowdisplaybreaks \\
v_{(3,2,1,3)_{\nu}}(t) & \ = \ +1, &
v_{(3,2,3,1)_{\nu}}(t) & \ = \ -2+\sqrt{2}, \label{thm:unicycleSinusoids:04:E} \allowdisplaybreaks \\
v_{(3,3,1,2)_{\nu}}(t) & \ = \ -\frac{1}{2}-\frac{1}{\sqrt{2}}, &
v_{(3,3,2,1)_{\nu}}(t) & \ = \ +\frac{1}{2}-\frac{1}{\sqrt{2}}, \label{thm:unicycleSinusoids:04:F}
\end{align}
\end{subequations}
with the abbreviation
\begin{equation}\label{eq:unicycleSinusoids:04}
(k_1,k_2,k_3,k_4)_{\nu} \ := \ \big(3(\nu-1)+k_1,\, 3(\nu-1)+k_2,\, 3(\nu-1)+k_3,\, 3(\nu-1)+k_4\big)
\end{equation}
for $\nu=1,\ldots,N$ and $k_1,k_2,k_3,k_4=1,2,3$. Then $(\mathbf{u}^j)_j$ GD($4$)-converges uniformly to~$\mathbf{v}$.
\end{proposition}
Now we are ready to present the control law. Let $h_{\text{s}},h_{\text{c}}\colon\mathbb{R}\to\mathbb{R}$ be defined as in~\cref{eq:hshc}. Let $m:=3N$, and let $(\mathbf{u}^j)_j$ and~$\mathbf{v}$ be defined as in \Cref{thm:unicycleSinusoids}. Define
\begin{subequations}\label{eq:unicycleControlLaw}
\begin{align}
u_{\nu,\text{t}}^j(t,y) & \ := \ u^j_{3\nu-2}(t)\,h_{\text{s}}(y) + u^j_{3\nu-1}(t)\,h_{\text{c}}(y) \label{eq:unicycleControlLaw:A} \allowdisplaybreaks \\
u_{\nu,\text{r}}^j(t,y) & \ := \ u^j_{3\nu}(t) \label{eq:unicycleControlLaw:B}
\end{align}
\end{subequations}
for every agent $\nu=1,\ldots,N$, every sequence index~$j$, every time $t\in\mathbb{R}$ and every output value $y\in\mathbb{R}$. When we insert~$u_{\nu,\text{t}}=u_{\nu,\text{t}}^j(t,\psi(x))$ and $u_{\nu,\text{r}}=u_{\nu,\text{r}}^j(t,\psi(x))$ into~\cref{eq:unicycleKinematics2}, we obtain the closed-loop system
\[
\Sigma^{j}\colon\qquad \dot{x} \ = \ \sum_{i=1}^{m}u^j_i(t)\,f_i(x),
\]
where the vector fields $f_1,\ldots,f_m$ on~$M$ are given by
\begin{align*}
f_{3\nu-2}(x) & \ := \ h_{\text{s}}(\psi(x))\,g_{\nu,\text{t}}(x), \allowdisplaybreaks \\
f_{3\nu-1}(x) & \ := \ h_{\text{c}}(\psi(x))\,g_{\nu,\text{t}}(x), \\
f_{3\nu}(x) & \ := \ g_{\nu,\text{r}}(x).
\end{align*}
It is easy to check that we are now in a situation in which we can apply \Cref{thm:outputConvergenceResult} (to see this, note that~$\pi$ is a smooth proper surjective submersion, note that $g_{\nu,\text{r}}$ is tangent to the fibers of~$\pi$, and let $h_{3\nu-2}:=h_{\text{s}}$, $h_{3\nu-1}:=h_{\text{c}}$, $h_{3\nu}:\equiv1$, $e_{3\nu-2}:=e_{3\nu-1}:=g_{\nu,\text{t}}$, $e_{3\nu}:=g_{\nu,\text{r}}$ for $\nu=1,\ldots,N$). Therefore, $\Sigma^j$ has the existence and uniqueness property of solutions. A direct computation shows that~\cref{eq:unicycleLimitSystem} can be written as
\[
\Sigma^{\infty}\colon\qquad\dot{x} \ = \ \sum_{0<|I|\leq{4}}\frac{v_I}{|I|}[f_I](x)
\]
with the iterated Lie derivatives $[f_I]$ as in \Cref{def:Sys}. Since we deal with a driftless system, \Cref{thm:outputConvergenceResult}~\ref{thm:outputConvergenceResult:B} states that for every compact subset $K\subseteq{M}$, the trajectories of $(\Sigma^j)_j$ converge in~$K$ on compact time intervals with respect to $(d,\sqrt{\psi})$ to the trajectories of $\Sigma^{\infty}$. Combining the convergence result, \Cref{thm:outputConvergenceResult}, and the stability results \Cref{thm:LUASToPLUAS,thm:LUESToLUES}, one can derive the following consequence.
\begin{corollary}\label{thm:unicycles}
Let~$E$, $\tilde{E}$, and $\tilde{\psi}\colon\mathbb{R}^{2N}\to\mathbb{R}$ be defined as before. For each sequence index $j$, let $\Sigma^j$ denote the system \cref{eq:unicycleKinematics} under the control law $u_{\nu,\text{t}}=u_{\nu,\text{t}}^j(t,\tilde{\psi}(p))$ and $u_{\nu,\text{r}}=u_{\nu,\text{r}}^j(t,\tilde{\psi}(p))$ given by \cref{eq:unicycleControlLaw}, \cref{eq:hshc} and \Cref{thm:unicycleSinusoids}. Then, the following holds.
\begin{enumerate}[label=(\alph*)]
	\item\label{item:unicycles:A} The set~$E$ is PLUAS for~$(\Sigma^j)_j$.
	\item\label{item:unicycles:B} If~$\tilde{E}$ is locally exponentially stable for~\cref{eq:Mensa}, then there exists a sequence index $j_0$ such that~$E$ is LUES for~$(\Sigma^j)_{j\geq{j_0}}$.
\end{enumerate}
\end{corollary}
\begin{proof}
From \Cref{thm:outputConvergenceResult} we know that for every compact set $K\subseteq{M}$, the trajectories of $(\Sigma^j)_j$ converge in~$K$ on compact time intervals with respect to $(d,\sqrt{\psi})$ to the trajectories of $\Sigma^{\infty}$. However, to apply \Cref{thm:LUASToPLUAS,thm:LUESToLUES}, we have to show that the trajectories of $(\Sigma^j)_j$ converge in a $\delta$-neighborhood of the non\-compact set~$E$ on compact time intervals with respect to $(d,b)$ to the trajectories of $\Sigma^{\infty}$, where $b(x):=d(x,E)$ is the pseudo distance of $x\in{M}$ to~$E$. This follows if we can prove the following two properties: $\sqrt{\psi}$ is bounded on this $\delta$-neighborhood of~$E$ by a multiple of~$b$, and the trajectories of $(\Sigma^j)_j$ converge in a $\delta$-neighborhood of~$E$ with respect to $(d,\sqrt{\psi})$ to the trajectories of $\Sigma^{\infty}$. For this purpose, we exploit the translational invariance of the problem.

We show that $\sqrt{\psi}$ is bounded on any $\delta$-neighborhood of~$E$ by a multiple of~$b$. For every $p\in\mathbb{R}^{2N}$, let $\tilde{b}(p)$ be the Euclidean distance of $p$ to $\tilde{E}$. Then we have $b=\tilde{b}\circ\pi$. Let $\tilde{K}$ be the compact set of all $(p_1,\ldots,p_N)\in\tilde{E}$ with $p_1=0$. Since $\tilde{\psi}$ is non\-negative and smooth, we know from \cref{thm:LipschitzDerivative} in the appendix that $\sqrt{\text{\footnotesize$\tilde{\psi}$}}$ is locally Lipschitz continuous. It follows that $\sqrt{\text{\footnotesize$\tilde{\psi}$}}$ is Lipschitz continuous on any $\delta$-neighborhood of $\tilde{K}$ (with respect to the Euclidean distance). Note that every element $p=(p_1,\ldots,p_N)$ in a $\delta$-neighborhood of $\tilde{E}$ can be shifted to the $\delta$-neighborhood of $\tilde{K}$ by means of the translation $p\mapsto{p-(p_1,\ldots,p_1)}$. Moreover, $\tilde{\psi}$ is invariant under this shift. It follows that $\sqrt{\text{\footnotesize$\tilde{\psi}$}}$ is Lipschitz continuous on any $\delta$-neighborhood of $\tilde{E}$. This implies that $\sqrt{\text{\footnotesize$\tilde{\psi}$}}$ is bounded by a multiple of $\tilde{b}$ on the $\delta$-neighborhood of $\tilde{E}$. 
Thus, $\sqrt{\psi}$ is bounded by a multiple of~$b$ on any $\delta$-neighborhood of~$E$.

It is left to prove that the trajectories of $(\Sigma^j)_j$ converge in a $\delta$-neighborhood of~$E$ with respect to $(d,\sqrt{\psi})$ to the trajectories of $\Sigma^{\infty}$. Note that both $\Sigma^{j}$ and $\Sigma^{\infty}$ are invariant under translations of the form $p\mapsto{p-(p_1,\ldots,p_1)}$. As above, one can exploit the translational invariance to reduce the problem of convergence within a given $\delta$-neighborhood of~$E$ to a compact subset of~$M$.

So far, we have shown that the trajectories of $(\Sigma^j)_j$ converge in a $\delta$-neighborhood of~$E$ with respect to $(d,b)$ to the trajectories of $\Sigma^{\infty}$, where $b(x):=d(x,E)$ for every $x\in{M}$. Since~$E$ is locally asymptotically stable for $\Sigma^{\infty}$ (see~\cite{Oh2014}), statement~\ref{item:unicycles:A} follows from \Cref{thm:LUASToPLUAS}. Moreover, statement~\ref{item:unicycles:B} follows from \Cref{thm:LUESToLUES}.
\end{proof}
The reader is referred to~\cite{Suttner2017} for simulation results.
\begin{remark}
The control law can be extended to the problem of \emph{distributed} formation control. In this case, the agents cannot share their distance measurements. Instead, each agent tries to minimize its own local potential function. This extension requires however a different notion of convergence of trajectories. We will discuss this problem in a subsequent paper.
\end{remark}


\appendix

\section{Proofs}\label{sec:appendix}

\subsection{Proof of \texorpdfstring{\Cref{thm:integrationByParts}}{Lemma~\ref{thm:integrationByParts}}}\label{sec:integrationByParts}
For the proof of \Cref{thm:integrationByParts}, we use the following observation.
\begin{lemma}\label{thm:tangentCharacterization}
Let $\alpha$ be a locally Lipschitz continuous func\-tion on~$M$. Then, for every $x\in{M}$, and any two curves $\gamma_1,\gamma_2$ passing through~$x$ at~$0$ that are differentiable at $0$, the limit of $(\alpha(\gamma_1(s))- \alpha(x))/s$ exists as $s$ tends to $0$ if and only if the limit of $(\alpha(\gamma_2(s))- \alpha(x))/s$ exists as $s$ tends to $0$, in which case both limits coincides.
\end{lemma}
This statement follows immediately from the definition of local Lipschitz continuity.

Now we prove \Cref{thm:integrationByParts}. The claim that~\cref{eq:integrationByParts:A} has the existence and uniqueness property of solutions follows from Carath{\'e}odory's theorems (see, e.g,~\cite{HaleBook}) if we can show that the time-varying vector field on the right-hand side of~\cref{eq:integrationByParts:A} is a locally uniformly bounded and locally uniformly Lipschitz continuous Carath{\'e}odory vector field. This is clearly satisfied for $f_0$ and the $[f_I]$ because of properties~\ref{item:def:Sys:i} and~\ref{item:def:Sys:ii} in \Cref{def:Sys}. Since the functions $v_I$ are measurable and bounded by \Cref{def:inputs}, each of the maps $(t,x)\mapsto\frac{v_I(t)}{|I|}[f_I](t,x)$ is also a locally uniformly bounded and locally uniformly Lipschitz continuous Carath{\'e}odory vector field. Thus, the same is true for the right-hand side of~\cref{eq:integrationByParts:A}.

To prove the remaining assertions of \Cref{thm:integrationByParts}, fix an ordinary input $(u_i)_{i=1,\ldots,m}$, a solution $\gamma\colon{J}\to{M}$ of~\cref{eq:integrationByParts:B}, a multi-index~$I$ of length $0<|I|\leq{r}$, and a function $\alpha\in{C^{\infty}(M)}$. Note that because of condition~\ref{item:def:Sys:ii} in \Cref{def:Sys}, the function $f_I\alpha$ is locally Lipschitz continuous on the product manifold $\mathbb{R}\times{M}$. 
Moreover, the map $t\mapsto(t,\gamma(t))$ is locally absolutely continuous on~$J$. Thus, the composition $\Gamma\colon{J}\to\mathbb{R}$, $t\mapsto(f_I\alpha)(t,\gamma(t))$ of these two maps is again locally absolutely continuous on~$J$. 
It follows that the derivative $\dot{\Gamma}$ of $\Gamma$ exists almost everywhere on~$J$ and is Lebesgue integrable on compact subintervals of~$J$. Fix any $t\in{J}$ at which both $\gamma$ and $\Gamma$ are differentiable. The proof is complete if we can show that the Lie derivative of $f_I\alpha$ along $f_0$ exists at $(t,\gamma(t))$ with
\[
(f_0f_I\alpha)(t,\gamma(t)) \ = \ \dot{\Gamma}(t) - (\partial_tf_I\alpha)(t,\gamma(t)) - \sum_{i=1}^{m}u_i(t)(f_if_I\alpha)(t,\gamma(t)).
\]
Because of property~\ref{item:def:Sys:ii} in \Cref{def:Sys}, we know that the time derivative of $f_I\alpha$ exists as continuous function on $\mathbb{R}\times{M}$. Using the mean value theorem and the continuity of $\partial_tf_I\alpha$, we obtain
\[
\lim_{s\to0}\frac{1}{s}\big((f_I\alpha)(t+s,\gamma(t+s))-(f_I\alpha)(t,\gamma(t+s))\big) \ = \ \big(\partial_tf_I\alpha\big)\big(t,\gamma(t)\big).
\]
Therefore, it remains to prove that
\[
(f_0f_I\alpha)(t,x) \ = \ \lim_{s\to0}\frac{1}{s}\big((f_I\alpha)(t,\gamma(t+s))-(f_I\alpha)(t,x)\big) - \sum_{i=1}^{m}u_i(t)(f_if_I\alpha)(t,x),
\]
where $x:=\gamma(t)$. Note that all derivatives that appear in the above equation are taken at the fixed time~$t$. We introduce local representations with respect to an arbitrary chart $c=(U,\varphi)$ for~$M$ at~$x$. Let $h\colon\varphi(U)\to\mathbb{R}$ be the local representation of the function $x\mapsto{f_I\alpha(t,x)}$ on~$M$ with respect to $c$. Let $g_0\colon\varphi(U)\to\mathbb{R}^n$ be the local representation of the vector field $x\mapsto{f_0(t,x)}$ on~$M$ with respect to $c$, where $n$ is the dimension of~$M$. For each $i=1,\ldots,m$, let $g_i\colon\varphi(U)\to\mathbb{R}^n$ be the local representation of the vector field $x\mapsto{u_i(t)f_i(t,x)}$ on~$M$ with respect to~$c$. Note that the maps $h,g_0,g_1,\ldots,g_m$ are locally Lipschitz continuous. We know from \Cref{thm:tangentCharacterization} that if the Lie derivative of $f_I\alpha$ along $f_0$ exists at~$(t,x)$, then it is given by
\[
(f_0f_I\alpha)(t,x) \ = \ \lim_{s\to0}\frac{1}{s}\big(h(\xi+s\,g_0(\xi))-h(\xi)\big),
\]
where $\xi:=\varphi(x)\in\mathbb{R}^n$. Moreover by \Cref{thm:tangentCharacterization}, we have
\[
\lim_{s\to0}\frac{1}{s}\big((f_I\alpha)(t,\gamma(t+s))-(f_I\alpha)(t,x)\big) \ = \ \lim_{s\to0}\frac{1}{s}\big(h(\xi+sv)-h(\xi)\big)
\]
where $v:=(\varphi\circ\gamma)\dot{\phantom{.}}(t)\in\mathbb{R}^n$, and, for every $i=1,\ldots,m$,
\[
(g_ih)(\zeta) \ := \ \lim_{s\to0}\frac{1}{s}\big(h(\zeta+s\,g_i(\zeta))-h(\zeta)\big) \ = \ u_i(t)(f_if_I\alpha)(t,\varphi^{-1}(\zeta))
\]
defines a continuous function $g_ih$ on $\varphi(U)$. Thus, the proof is complete if we can show that
\[
\lim_{s\to0}\frac{1}{s}\big(h(\xi+s\,g_0(\xi))-h(\xi)\big) \ = \ \lim_{s\to0}\frac{1}{s}\big(h(\xi+sv)-h(\xi)\big) - \sum_{i=1}^{m}(g_ih)(\xi).
\]
Note that $v=g_0(\xi) + \sum_{i=1}^{m}g_i(\xi)$ by the defining differential equation~\cref{eq:integrationByParts:B} of~$\gamma$. Thus, we can write
\[
\frac{1}{s}\big(h(\xi+sv)-h(\xi)\big) \ = \ \frac{1}{s}\big(h(\xi+s\,g_0(\xi))-h(\xi)\big) + \sum_{i=1}^{m}\frac{1}{s}\big(h\big(\xi_i(s)+sg_i(\xi)\big)-h\big(\xi_i(s)\big)\big)
\]
with $\xi_i(s):=\xi+s\sum_{a=0}^{i-1}g_a(\xi)$ for $s$ sufficiently close to $0$. Since we know that the limit of the difference quotient on the left-hand side of the above equation exists when $s$ tends to $0$, the proof is complete if we can show that
\[
\lim_{s\to0}\frac{1}{s}\big(h\big(\xi_i(s)+sg_i(\xi)\big)-h\big(\xi_i(s)\big)\big) \ = \ (g_ih)(\xi)
\]
for $i=1,\ldots,m$. Since each $g_i$ is locally Lipschitz continuous, for every $\zeta\in\varphi(U)$, there exists a unique maximal integral curve $\tau\mapsto\Phi_i(\zeta,\tau)$ of $g_i$ that passes through $\zeta$ at time $0$. By applying the mean value theorem to the function $\tau\mapsto{h(\Phi_i(\xi_i(s),\tau))}$, we obtain that for every $s\in\mathbb{R}$ sufficiently close to $0$, there exists $\tau_s\in\mathbb{R}$ between $0$ and $s$ such that
\[
h\big(\Phi_i(\xi_i(s),s)\big) \ = \ h\big(\xi_i(s)\big) + s\,(g_ih)\big(\Phi_i(\xi_i(s),\tau_s)\big)
\]
Using the continuity of the flow map $\Phi_i$, this implies
\[
\lim_{s\to0}\frac{1}{s}\big(h\big(\Phi_i(\xi_i(s),s)\big) - h\big(\xi_i(s)\big)\big) \ = \ (g_ih)(\xi).
\]
It is therefore left to prove that
\[
\lim_{s\to0}\frac{1}{s}\big(h(\xi_i(s)+sg_i(\xi))-h(\Phi_i(\xi_i(s),s))\big) \ = \ 0.
\]
Since $h$ is locally Lipschitz continuous, this follows if we can show that
\[
\lim_{s\to0}\frac{1}{s}\big(\Phi_i(\xi_i(s),s)-\xi_i(s)\big) \ = \ g_i(\xi).
\]
By applying the mean value theorem to each component of the map $\tau\mapsto\Phi_i(\xi_i(s),\tau)$ whose derivative is $\tau\mapsto{g_i(\Phi_i(\xi_i(s),\tau))}$, and because of the continuity of $\Phi_i$, one obtains that the above equation is indeed true. This completes the proof.

\subsection{Proof of \texorpdfstring{\Cref{thm:outputConvergenceResult}}{Theorem~\ref{thm:outputConvergenceResult}}}\label{sec:outputConvergenceResult}
Our aim is to show that \Cref{thm:outputConvergenceResult} is just a particular case of \Cref{thm:convergenceResult}. To apply \Cref{thm:convergenceResult}, we have to prove that $(f_0,f_1,\ldots,f_m)$ is an element of $\operatorname{Sys}(M;m,r)$, and that the boundedness conditions of the Lie derivatives in \Cref{thm:convergenceResult} are satisfied. Then, statements~\ref{thm:outputConvergenceResult:A} and~\ref{thm:outputConvergenceResult:B} in \Cref{thm:outputConvergenceResult} follow from \Cref{thm:convergenceResult} for $b\equiv1$ and $b=\sqrt{\psi}$, respectively. The main part of the proof is devoted to show that $(f_0,f_1,\ldots,f_m)$ satisfies conditions~\ref{item:def:Sys:i}-\ref{item:def:Sys:iii} in \Cref{def:Sys}. In particular, we have to show that the iterated Lie derivatives $f_I\alpha$ exist on the entire manifold~$M$ as locally uniformly Lipschitz continuous time-varying functions. For this purpose, we use the following criterion for local uniform Lipschitz continuity.
\begin{lemma}[\cite{Durand2010}]\label{thm:LipschitzCriterion}
A time-varying function~$\beta$ on~$M$ is locally uniformly Lipschitz continuous if and only if for every chart $c=(U,\varphi)$ for~$M$, the \emph{pointwise Lipschitz constant of~$\beta$ with respect to $c$},
\[
\operatorname{Lip}_c(\beta)(t,x) \ := \ \limsup_{x'\to{x}}\frac{|\beta(t,x')-\beta(t,x)|}{\|\varphi(x')-\varphi(x)\|},
\]
defines a locally uniformly bounded time-varying function on~$U$. If $\beta_1,\beta_2,\beta$ are locally uniformly Lipschitz continuous time-varying functions on~$M$, and if $h$ is a $C^1$ function on $\mathbb{R}$, then also $\beta_1+\beta_2,\beta_1\beta_2,h\circ\beta$ are locally uniformly Lipschitz on~$M$, and we have
\begin{align*}
\operatorname{Lip}_c(\beta_1+\beta_2) & \ \leq \ \operatorname{Lip}_c(\beta_1) + \operatorname{Lip}_c(\beta_2), \allowdisplaybreaks \\
\operatorname{Lip}_c(\beta_1\beta_2) & \ \leq \ \operatorname{Lip}_c(\beta_1)\,|\beta_2| + |\beta_1|\,\operatorname{Lip}_c(\beta_2), \allowdisplaybreaks \\
\operatorname{Lip}_c(h\circ\beta) & \ \leq \ |h'\circ\beta|\,\operatorname{Lip}_c(\beta)
\end{align*}
on $\mathbb{R}\times{U}$.
\end{lemma}
\begin{proof}
The criterion for local uniform Lipschitz continuity can be proved in the same way as it is done in~\cite{Durand2010} for time-independent maps. We omit this here. The asserted estimates for the pointwise Lipschitz constant follow directly from the definition of the upper limit.
\end{proof}
Another property that we have to verify for the proof of \Cref{thm:outputConvergenceResult} is that the Lie derivatives of the $f_I\alpha$ along the drift $f_0$ satisfy the boundedness assumption in \Cref{thm:convergenceResult} in the sense of \Cref{thm:boundedAE}. Since the existence of these Lie derivatives is not guaranteed everywhere, we consider instead the following upper bound.
\begin{lemma}\label{thm:setValuedLieDerivative}
Let~$f$ be a time-varying vector field on~$M$ with the existence and uniqueness property of integral curves, and let $\Phi$ denote the flow of~$f$. Let~$\beta$ be a locally uniformly Lipschitz continuous time-varying function on~$M$. Then, for every $(t,x)\in\mathbb{R}\times{M}$, the limit
\[
\operatorname{L}(f,\beta)(t,x) \ := \ \limsup_{s\to0}\frac{1}{|s|}\big|\beta\big(t,\Phi(t+s,t,x)\big)\big)-\beta(t,x)\big|
\]
exists. Moreover, if $\beta_1,\beta_2$ are locally uniformly Lipschitz continuous time-varying functions on~$M$, and if $h$ is a $C^1$ function on $\mathbb{R}$, then we have
\begin{align*}
\operatorname{L}(f,\beta_1+\beta_2) & \ \leq \ \operatorname{L}(f,\beta_1) + \operatorname{L}(f,\beta_2), \allowdisplaybreaks \\
\operatorname{L}(f,\beta_1\beta_2) & \ \leq \ \operatorname{L}(f,\beta_1)\,|\beta_2| + |\beta_1|\,\operatorname{L}(f,\beta_2), \allowdisplaybreaks \\
\operatorname{L}(f,h\circ\beta) & \ \leq \ |h'\circ\beta|\,\operatorname{L}(f,\beta)
\end{align*}
on $\mathbb{R}\times{M}$.
\end{lemma}
\begin{proof}
Existence of the limit $\operatorname{L}(f,\beta)(t,x)$ follows immediately from \Cref{thm:LipschitzCriterion}. The asserted estimates follow directly from the definition of the upper limit.  
\end{proof}
We conclude the preparations for the proof of \Cref{thm:outputConvergenceResult} by summarizing important properties of the non\-negative~$C^r$ function~$\psi$ on~$M$ with $r\geq2$.
\begin{lemma}\label{thm:LipschitzDerivative}
The function $\sqrt{\psi}$ is locally Lipschitz continuous, and, with respect to every chart $c=(U,\varphi)$ for~$M$, the pointwise Lipschitz constant $\operatorname{Lip}_c(\psi)$ is locally bounded a multiple of $\sqrt{\psi}$ on~$U$.
\end{lemma}
\begin{proof}
It is known (see, e.g., Section~4.3.5 in~\cite{LernerBook}) that the square root of every~$C^2$ function~$\alpha$ on an open subset of the Euclidean space is locally Lipschitz continuous, and that the derivative of~$\alpha$ is locally bounded by a multiple of~$\sqrt{\alpha}$. This holds in particular for every local representation of~$\psi$. The claim follows from the observation that the pointwise Lipschitz constant of $\psi$ with respect to a chart $(U,\varphi)$ for~$M$ at $x\in{M}$ is equal to the operator norm of the derivative of $\psi\circ\varphi^{-1}\colon\varphi(U)\to\mathbb{R}$ at $\varphi(x)$.
\end{proof}
Throughout the proof of \Cref{thm:outputConvergenceResult}, we let $E:=\psi^{-1}(0)\subseteq{M}$ denote the possibly empty set of points of~$M$ at which the non-negative function~$\psi$ attains the value~$0$, and we let $M_{>}:=M\setminus{E}$ denote the open submanifold of~$M$ on which~$\psi$ is strictly positive. For the sake of clarity, we divide the proof \Cref{thm:outputConvergenceResult} into a sequence of lemmas.
\begin{lemma}\label{thm:Lemma00}
For each $i=1,\ldots,m$, the function $h_i\circ\psi$ is locally Lipschitz continuous.
\end{lemma}
\begin{proof}
We write $h_i\circ\psi$ as a composition of locally Lipschitz continuous maps. For this purpose define $H_i\colon\mathbb{R}\to\mathbb{R}$ by $H_i(s):=0$ for $s\leq0$ and by $H_i(s):=h_i(s^2)$ for $s>0$. Then we have $h_i\circ\psi=H_i\circ\sqrt{\psi}$. We already know from \Cref{thm:LipschitzDerivative} that $\sqrt{\psi}$ is locally Lipschitz continuous. Since $H_i$ is of class $C^r$ outside the origin, it is left to prove that $H_i$ is Lipschitz continuous in a small neighborhood of the origin. By assumptions on $h_i$, and by definition of $H_i$, such that $H_i(s)/s$ as well as $H_i'(s)$ remain bounded as $s\to0$. This implies that the pointwise Lipschitz constant of $H_i$ (with respect to the standard chart for $\mathbb{R}$) is bounded on a small interval around the origin. It follows from \Cref{thm:LipschitzCriterion} that $H_i$ is locally Lipschitz.
\end{proof}
Next, we derive boundedness properties of the iterated Lie derivatives of $\psi$ along the vector fields $e_i$, $i=1,\ldots,m$. Since we apply the $e_i$ successively from the left to $\psi$, it is more convenient to numerate the indices in a multi-index $I$ from the right to the left. This means that we write $I=(i_{\ell},\ldots,i_1)$ instead of $I=(i_1,\ldots,i_{\ell})$. We use this notation in the rest of the proof.
\begin{lemma}\label{thm:Lemma01}
Let $I=(i_{\ell},\ldots,i_1)$ be a multi-index of length $0<\ell\leq{r}$ with $i_1,\ldots,i_{\ell}\in\{1,\ldots,m\}$. The following statements hold.
\begin{enumerate}[label=(\alph*)]
	\item\label{thm:Lemma01:A} If $\ell<r$, then $e_I\psi$ exists as a locally uniformly bounded and locally uniformly Lipschitz continuous time-varying function on~$M$.
	\item\label{thm:Lemma01:B} If $\ell=r$, then $e_I\psi$ exists as a locally uniformly bounded Carath{\'e}odory function on~$M$.
	\item\label{thm:Lemma01:C} The time derivative $\partial_te_I\psi$ exists as a locally uniformly bounded time-varying function on~$M$ and is continuous as a function on $\mathbb{R}\times{M}$.
	\item\label{thm:Lemma01:D} If $e_{i_2},\ldots,e_{i_{\ell}}$ are tangent to the fibers of~$\pi$, then the time-varying functions in~\ref{thm:Lemma01:A}-\ref{thm:Lemma01:C} are locally uniformly bounded by a multiple of $\sqrt{\psi}$.
	\item\label{thm:Lemma01:E} If $\ell<r$, and if $f_0$ and $e_{i_2},\ldots,e_{i_{\ell}}$ are tangent to the fibers of~$\pi$, then $\operatorname{L}(f_0,e_I\psi)$ is locally uniformly bounded by a multiple of $\sqrt{\psi}$.
\end{enumerate}
\end{lemma}
\begin{proof}
Statements~\ref{thm:Lemma01:A}-\ref{thm:Lemma01:C} follow immediately from the assumptions that~$\psi$ is of class $C^r$ and that $(e_0,e_1,\ldots,e_m)$ is an element of $\operatorname{Sys}(M;m,r)$. 
To prove parts~\ref{thm:Lemma01:D} and~\ref{thm:Lemma01:E}, fix an arbitrary point $x_0\in{M}$, and let $\tilde{x}_0:=\pi(x_0)\in\tilde{M}$. We exploit the assumption that $\pi\colon{M}\to\tilde{M}$ is a smooth proper surjective submersion. By Ehresmann's fibration theorem (see, e.g.,~\cite{AbrahamBook}), there exists an open neighborhood $\tilde{U}$ of $\tilde{x}_0$ in $\tilde{M}$ and a diffeomorphism $\phi\colon\tilde{U}\times\pi^{-1}(\tilde{x}_0)\to\pi^{-1}(\tilde{U})$ such that $\pi(\phi(\tilde{x},x))=\tilde{x}$ for every $\tilde{x}\in\tilde{U}$ and every $x\in\pi^{-1}(\tilde{x}_0)$. Let $(\hat{U},\hat{\varphi})$ be a chart for the smooth manifold $\pi^{-1}(\tilde{x}_0)$ at $x_0$. 
By shrinking $\tilde{U}$ sufficiently, we can ensure that $\tilde{U}$ is the domain of a chart $(\tilde{\varphi},\tilde{U})$ for $\tilde{M}$ and that $\varphi:=(\tilde{\varphi}\times\hat{\varphi})\circ\phi^{-1}|_U$ is a well-defined chart map for~$M$ at $x_0$ with domain $U:=\pi^{-1}(\tilde{U})$. 
Let $n$ and $\tilde{n}$ denote the dimensions of~$M$ and $\tilde{M}$, respectively. Let $\varphi_1,\ldots,\varphi_n\in{C^{\infty}(U)}$ be the coordinate functions of $\varphi$, and let $\partial_1,\ldots,\partial_n$ be the corresponding local basis for vector fields on~$U$. By definition of $\varphi$, every time-varying vector field~$f$ on~$M$ that is is tangent to the fibers of~$\pi$ has the local representation $f=\sum_{\mu=\tilde{n}+1}^{n}(f\varphi_{\mu})\partial_{\mu}$ on~$U$. 
Moreover, since~$\psi$ is assumed to be constant on the fibers of~$\pi$, we have $\partial_{\mu}\psi\equiv0$ for $\mu=\tilde{n}+1,\ldots,n$. 
Now suppose that $e_{i_2},\ldots,e_{i_{\ell}}$ are tangent to the fibers of~$\pi$. Then we have
\[
e_{I}\psi \ = \ \sum_{\mu=1}^{\tilde{n}}(e_I\varphi_{\mu})\,(\partial_{\mu}\psi)
\]
on $\mathbb{R}\times{U}$. 
Note that each of the $(e_I\varphi_{\mu})$ is locally uniformly bounded. Moreover, each of the $(\partial_{\mu}\psi)$ is bounded by the pointwise Lipschitz constant $\operatorname{Lip}_c(\psi)$ with respect to $c:=(U,\varphi)$. By \Cref{thm:LipschitzDerivative}, $\operatorname{Lip}_c(\psi)$ is locally bounded by a multiple of $\sqrt{\psi}$. Thus, $e_I\psi$ is locally uniformly bounded by a multiple of $\sqrt{\psi}$. The same argument shows that also the time derivative of $e_I\psi$ is locally uniformly bounded by a multiple of $\sqrt{\psi}$. Finally, to prove~\ref{thm:Lemma01:E}, suppose additionally that $\ell<r$ and that $f_0$ is tangent to the fibers of~$\pi$. Then we have
\[
\operatorname{L}(f_0,e_{I}\psi) \ \leq \ \sum_{\mu=1}^{\tilde{n}}\operatorname{L}(f_0,e_I\varphi_{\mu})\,|\partial_{\mu}\psi|
\]
on $\mathbb{R}\times{U}$. 
Since $\operatorname{L}(f_0,e_I\varphi_{\mu})$ is locally uniformly bounded, we conclude as in the proof of part~\ref{thm:Lemma01:D} that $\operatorname{L}(f_0,e_{I}\psi)$ is uniformly bounded by a multiple of $\sqrt{\psi}$, which completes the proof of the lemma.
\end{proof}
Let $I=(i_{\ell},\ldots,i_1)$ be a multi-index of length $0<\ell\leq{r}$ with $i_1,\ldots,i_{\ell}\in\{1,\ldots,m\}$. Recall that a \emph{partition} of the set $\{1,\ldots,\ell\}$ is a collection of non\-empty, disjoint subsets of $\{1,\ldots,\ell\}$ such that their union is all of $\{1,\ldots,\ell\}$. We denote by~$P_\ell$ the set of all partitions of~$\{1,\ldots,\ell\}$. For any (possibly empty) subset $S$ of $\{0,1,\ldots,\ell\}$, let $I\langle{S}\rangle$ be the (possibly empty) multi-index $(i_{s_{|S|}},\ldots,i_{s_1})$, where $|S|$ is the number of elements of $S$, and  $s_1,\ldots,s_{|S|}$ are the elements of $S$ in increasing order. We denote by $\#{I}$, the number of subindices $k\in\{1,\ldots,\ell\}$ such that for each $\nu=0,1,\ldots,r$, the $\nu$th derivative $h_{i_k}^{(\nu)}$ of $h_{i_k}$ has the property that $y^{\nu-1/2}h_{i_k}^{(\nu)}(y)$ remains bounded as $y\downarrow0$. Then, by the assumptions of \Cref{thm:outputConvergenceResult}, there are at least $|I|-\#I$ subindices $k\in\{1,\ldots,\ell\}$ such that $e_{i_k}$ is tangent to the fibers of~$\pi$.
\begin{lemma}\label{thm:Lemma02}
Let $I=(i_{\ell},\ldots,i_1)$ be a multi-index of length $0<\ell\leq{r}$ with $i_1,\ldots,i_{\ell}\in\{1,\ldots,m\}$, and let $i\in\{1,\ldots,m\}$. Then we have
\begin{equation}\label{eq:Lemma02}
e_I(h_i\circ\psi) \ = \ \sum_{p\in{P}_{\ell}}(h_i^{(|p|)}\circ\psi)\prod_{S\in{p}}e_{I\langle{S}\rangle}\psi
\end{equation}
on $\mathbb{R}\times{M_{>}}$. Moreover, for every compact subset~$K$ of~$M$, the following holds.
\begin{enumerate}[label=(\alph*)]
	\item\label{thm:Lemma02:A} $e_I(h_i\circ\psi)$ is uniformly bounded by a multiple of $\psi^{(1-\#I)/2}$ on $K\setminus{E}$. 
	\item\label{thm:Lemma02:B} $\partial_te_I(h_i\circ\psi)$ is uniformly bounded by a multiple of $\psi^{(1-\#I)/2}$ on $K\setminus{E}$.
	\item\label{thm:Lemma02:C} If $\ell<{r}$, and if $c=(U,\varphi)$ is a chart for~$M$ with $K\subseteq{U}$, then $\operatorname{Lip}_c(e_I(h_i\circ\psi))$ is uniformly bounded by a multiple of $\psi^{-\#I/2}$ on $K\setminus{E}$.
	\item\label{thm:Lemma02:D} If $\ell<{r}$, and if $f_0$ is tangent to the fibers of~$\pi$, then $\operatorname{L}(f_0,e_I(h_i\circ\psi))$ is uniformly bounded by a multiple of $\psi^{(1-\#I)/2}$ on $K\setminus{E}$.
\end{enumerate}
\end{lemma}
\begin{proof}
Since $h_i$ is of class $C^r$ on $(0,\infty)$, the composition $h_i\circ\psi$ is of class $C^r$ on $M_{>}$, which ensures that $e_I(h_i\circ\psi)$ exists as a time-varying function on $M_{>}$. Note that~\cref{eq:Lemma02} is just a generalized version of the well-known \emph{Fa\'{a} di Bruno formula} for higher-order derivatives of the composition of two maps, see, e.g., Proposition~1 in~\cite{Hardy2006}. For every partition $p\in{P_{\ell}}$, define the time-varying function
\begin{equation}\label{eq:Lemma02:proof:00}
e_I^p\psi \ := \ \prod_{S\in{p}}e_{I\langle{S}\rangle}\psi
\end{equation}
on~$M$. By definition of the time derivative, we obtain from~\cref{eq:Lemma02} that
\begin{equation}\label{eq:Lemma02:proof:01:A}
\partial_te_I(h_i\circ\psi) \ = \ \sum_{p\in{P_{\ell}}}(h_i^{(|p|)}\circ\psi)\,\partial_t(e_I^p\psi)
\end{equation}
on $\mathbb{R}\times{M}_{>}$ with
\begin{equation}\label{eq:Lemma02:proof:01:B}
\partial_t(e_I^p\psi) \ = \ \sum_{S\in{p}}(\partial_te_{I\langle{S}\rangle}\psi)\prod_{S\neq{S'}\in{p}}e_{I\langle{S'}\rangle}\psi.
\end{equation}
If $\ell<{r}$ and if $c=(U,\varphi)$ is a chart for~$M$, then we have
\begin{equation}\label{eq:Lemma02:proof:02:A}
\operatorname{Lip}_c(e_I(h_i\circ\psi))  \ \leq \ \sum_{p\in{P_{\ell}}}\Big(|h_i^{(|p|+1)}\circ\psi|\,\operatorname{Lip}_c(\psi)\,e_I^p\psi + |h_i^{(|p|)}\circ\psi|\,\operatorname{Lip}_c(e_I^p\psi) \Big)
\end{equation}
on $\mathbb{R}\times(M_{>}\cap{U})$ with
\begin{equation}\label{eq:Lemma02:proof:02:B}
\operatorname{Lip}_c(e_I^p\psi) \ \leq \ \sum_{S\in{p}}\operatorname{Lip}_c(e_{I\langle{S}\rangle}\psi)\prod_{S\neq{S'}\in{p}}|e_{I\langle{S'}\rangle}\psi|.
\end{equation}
Finally, if $\ell<{r}$, and if $f_0$ is tangent to the fibers of~$\pi$, then we have
\begin{equation}\label{eq:Lemma02:proof:03:A}
\operatorname{L}(f_0,e_I(h_i\circ\psi))  \ \leq \ \sum_{p\in{P_{\ell}}}|h_i^{(|p|)}\circ\psi|\,\operatorname{L}(f_0,e_I^p\psi)
\end{equation}
on $\mathbb{R}\times{M}_{>}$ with
\begin{equation}\label{eq:Lemma02:proof:03:B}
\operatorname{L}(f_0,e^p_I\psi) \ \leq \ \sum_{S\in{p}}\operatorname{L}(f_0,e_{I\langle{S}\rangle}\psi)\prod_{S\neq{S'}\in{p}}|e_{I\langle{S'}\rangle}\psi|.
\end{equation}
To prove the lemma, we fix an arbitrary partition $p\in{P_{\ell}}$ of $\{1,\ldots,\ell\}$, and then we verify that the terms in~\cref{eq:Lemma02,eq:Lemma02:proof:01:A,eq:Lemma02:proof:02:A,eq:Lemma02:proof:03:A} that correspond to this fixed~$p$ satisfy the asserted boundedness properties. Let~$K$ be a compact subset of~$M$. For every subset $S\in{p}$ of $\{1,\ldots,\ell\}$, we let $k(S,1),\ldots,k(S,|S|)$ denote the elements of $S$ in increasing order. Then, we can write $I\langle{S}\rangle=(i_{k(S,|S|)},\ldots,i_{k(S,1)})$. If there is an $S\in{p}$ such that $e_{i_{k(S,1)}}$ is tangent to the fibers of~$\pi$, then $e_{I\langle{S}\rangle}\psi\equiv0$, and therefore there is nothing left to prove with respect to boundedness. For the rest of the proof, we consider the non\-trivial case in which for every $S\in{p}$, the vector field $e_{i_{k(S,1)}}$ is not tangent to the fibers of~$\pi$. We make a case analysis. First, suppose that $\#I\geq2|p|$. Then, by the assumptions on $h_i$, the function $(h_i^{(|p|)}\circ\psi)$ is bounded by a multiple of $\psi^{(1-\#I)/2}$ on $K\setminus{E}$. The time-varying functions in~\cref{eq:Lemma02:proof:00,eq:Lemma02:proof:01:B,eq:Lemma02:proof:02:B,eq:Lemma02:proof:03:B} are locally uniformly bounded. It follows that the terms in~\cref{eq:Lemma02,eq:Lemma02:proof:01:A,eq:Lemma02:proof:03:A} that correspond to the partition $p$ are uniformly bounded by a multiple of $\psi^{(1-\#I)/2}$ on $K\setminus{E}$. Moreover, by \Cref{thm:LipschitzDerivative}, $\operatorname{Lip}_c(\psi)$ is locally bounded by a multiple of $\sqrt{\psi}$. Thus, the term in~\cref{eq:Lemma02:proof:02:A} that correspond to the partition $p$ is uniformly bounded by a multiple of $\psi^{-\#I/2}$ on $K\setminus{E}$. It remains to prove the same for the case $\#I<2|p|$. Recall that we also assume that for every $S\in{p}$, the vector field $e_{i_{k(S,1)}}$ is not tangent to the fibers of~$\pi$. It follows that there are at least $2|p|-\#I$ sets $S\in{p}$ such that $e_{i_{k(S,2)}},\ldots,e_{i_{k(S,|S|)}}$ are tangent to the fibers of~$\pi$. 
By \Cref{thm:Lemma01}, this implies that~\cref{eq:Lemma02:proof:00},~\cref{eq:Lemma02:proof:01:B}, and~\cref{eq:Lemma02:proof:03:B} are locally uniformly bounded by a multiple of~$\psi^{|p|-\#I/2}$ and that~\cref{eq:Lemma02:proof:02:B} is locally uniformly bounded by a multiple of~$\psi^{|p|-\#I/2-1/2}$. Since $(h_i^{(|p|)}\circ\psi)$ is bounded by a multiple of $\psi^{1/2-|p|}$ on $K\setminus{E}$, we obtain the same boundedness properties as in the case $\#I\geq2|p|$ for the terms in~\cref{eq:Lemma02,eq:Lemma02:proof:01:A,eq:Lemma02:proof:02:A,eq:Lemma02:proof:03:A} that correspond to the partition~$p$.
\end{proof}
Our next aim is to derive an explicit formula for the iterative Lie derivatives $f_I\alpha$, i.e., we want to compute
\begin{equation}\label{eq:proof:theProposition:01}
f_I\alpha \ = \ (h_{i_{\ell}}\circ\psi)e_{i_{\ell}}\big(\cdots(h_{i_2}\circ\psi)e_{i_2}\big((h_{i_1}\circ\psi)e_{i_1}\alpha\big)\cdots\big)
\end{equation}
on $\mathbb{R}\times{M_{>}}$ for arbitrary $I=(i_{\ell},\ldots,i_1)$. For instance, if $\ell=3$, then~\cref{eq:proof:theProposition:01} becomes
\begin{subequations}\label{eq:proof:theProposition:02}
\begin{align}
f_{(i_3,i_2,i_1)}\alpha \ = & \quad (e_{i_3}e_{i_2}e_{i_1}\alpha)\,(h_{i_1}\circ\psi)\,(h_{i_2}\circ\psi)\,(h_{i_3}\circ\psi) \label{eq:proof:theProposition:02:A} \allowdisplaybreaks \\
& + (e_{i_2}e_{i_1}\alpha)\,(h_{i_1}\circ\psi)\,(e_{i_3}(h_{i_2}\circ\psi))\,(h_{i_3}\circ\psi) \label{eq:proof:theProposition:02:B} \allowdisplaybreaks \\
& + (e_{i_2}e_{i_1}\alpha)\,(e_{i_3}(h_{i_1}\circ\psi))\,(h_{i_2}\circ\psi)\,(h_{i_3}\circ\psi) \label{eq:proof:theProposition:02:C} \allowdisplaybreaks \\
& + (e_{i_3}e_{i_1}\alpha)\,(e_{i_2}(h_{i_1}\circ\psi))\,(h_{i_2}\circ\psi)\,(h_{i_3}\circ\psi) \label{eq:proof:theProposition:02:D} \allowdisplaybreaks \\
& + (e_{i_1}\alpha)\,(e_{i_3}e_{i_2}(h_{i_1}\circ\psi))\,(h_{i_2}\circ\psi)\,(h_{i_3}\circ\psi) \label{eq:proof:theProposition:02:E} \allowdisplaybreaks \\
& + (e_{i_1}\alpha)\,(e_{i_2}(h_{i_1}\circ\psi))\,(e_{i_3}(h_{i_2}\circ\psi))\,(h_{i_3}\circ\psi) \label{eq:proof:theProposition:02:F}
\end{align}
\end{subequations}
on $\mathbb{R}\times{M_{>}}$. It is known (cf.~\cite{Grossman19892,Grossman19901,Grossman1992}) that these Lie derivatives can be expressed by means of rooted trees. We refer the reader to \cite{StanleyBook} for a precise definition of a \emph{rooted tree} and related terminology. An \emph{increasing tree on $\{0,1,\ldots,\ell\}$} is a rooted tree~$\tau$ of vertices $0,1,\ldots,\ell$ such that whenever a node $k'$ is a child of a node~$k$ in~$\tau$, we have $k'>k$. (Increasing trees are also known as \emph{recursive trees}.) As an example, all increasing trees on $\{0,1,\ldots,\ell\}$ are shown in \Cref{fig:rootedTrees} for $\ell=3$. We denote the set of increasing trees on $\{0,1,\ldots,\ell\}$ by $\mathcal{T}_{\ell}$. It is easy to check that $\mathcal{T}_{\ell}$ consists of $\ell!$ trees.
\begin{figure}
\centering
\subfloat[]{\label{fig:rootedTrees:a}\includegraphics{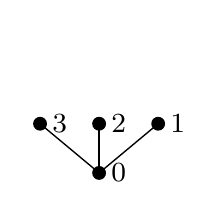}}
\subfloat[]{\label{fig:rootedTrees:b}\includegraphics{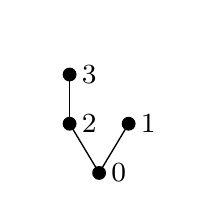}}
\subfloat[]{\label{fig:rootedTrees:c}\includegraphics{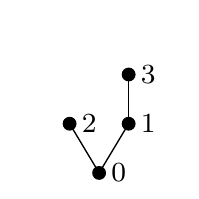}}
\subfloat[]{\label{fig:rootedTrees:d}\includegraphics{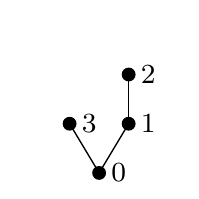}}
\subfloat[]{\label{fig:rootedTrees:e}\includegraphics{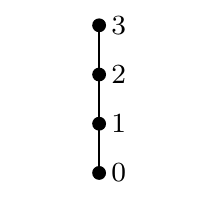}}
\subfloat[]{\label{fig:rootedTrees:f}\includegraphics{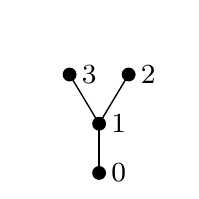}}
\caption{Depiction of all increasing trees on $\{0,1,2,3\}$. The trees are arranged in the same order as the corresponding constituents of $f_I\alpha$ in~\cref{eq:proof:theProposition:02}.}
\label{fig:rootedTrees}
\end{figure}
In order to write down a general version of~\cref{eq:proof:theProposition:02} for arbitrary $\ell$, we introduce the following notation. For every tree $\tau\in\mathcal{T}_{\ell}$ and every node $k\in\{0,1,\ldots,\ell\}$, let $\tau[k]$ be the possibly empty set of children of~$k$ in~$\tau$. If $\tau[k]$ is not empty, then $e_{I\langle\tau[k]\rangle}(h_{i_{k}}\circ\psi)$ is given by~\cref{eq:Lemma02} in \Cref{thm:Lemma02}. Otherwise, i.e., if $\tau[k]$ is empty, then we define $e_{I\langle\tau[k]\rangle}(h_{i_{k}}\circ\psi):=(h_{i_{k}}\circ\psi)$.
\begin{lemma}\label{thm:Lemma03}
Let $I=(i_{\ell},\ldots,i_1)$ be a multi-index of length $0<\ell\leq{r+1}$ with $i_1,\ldots,i_{\ell}\in\{1,\ldots,m\}$, and let $\alpha\in{C^{\infty}(M)}$. Then we have
\begin{equation}\label{eq:Lemma03}
f_I\alpha \ = \ \sum_{\tau\in\mathcal{T}_{\ell}}(e_{I\langle\tau[0]\rangle}\alpha)\prod_{k=1}^{\ell}e_{I\langle\tau[k]\rangle}(h_{i_{k}}\circ\psi)
\end{equation}
on $\mathbb{R}\times{M_{>}}$. Moreover, for every compact subset~$K$ of~$M$, the following holds.
\begin{enumerate}[label=(\alph*)]
	\item\label{thm:Lemma03:A} $f_I\alpha-(e_{I}\alpha)\prod_{k=1}^{\ell}(h_{i_{k}}\circ\psi)$ is uniformly bounded by a multiple of $\sqrt{\psi}$ on $K\setminus{E}$.
	\item\label{thm:Lemma03:B} If $\ell\leq{r}$, then $\partial_t(f_I\alpha)-(\partial_t(e_{I}\alpha))\prod_{k=1}^{\ell}(h_{i_{k}}\circ\psi)$ is uniformly bounded by a multiple of $\sqrt{\psi}$ on $K\setminus{E}$.
	\item\label{thm:Lemma03:C} If $\ell\leq{r}$, and if $c=(U,\varphi)$ is a chart for~$M$ with $K\subseteq{U}$, then $\operatorname{Lip}_c(f_I\alpha)$ is uniformly bounded on $K\setminus{E}$.
	\item\label{thm:Lemma03:D} If $\ell\leq{r}$, and if $f_0$ is tangent to the fibers of~$\pi$, then $\operatorname{L}(f_0,f_I\alpha)-\operatorname{L}(f_0,e_{I}\alpha)\prod_{k=1}^{\ell}(h_{i_{k}}\circ\psi)$ is uniformly bounded by a multiple of $\sqrt{\psi}$ on $K\setminus{E}$.
\end{enumerate}
\end{lemma}
\begin{proof}
As in~\cite{Grossman19901}, we prove~\cref{eq:Lemma03} by induction on the length $\ell$ of~$I$. Clearly,~\cref{eq:Lemma03} holds for $\ell=1$. In the inductive step, we assume that~\cref{eq:Lemma03} holds for $I=(i_{\ell_1},\ldots,i_1)$ and verify the formula for a multi-index $I':=(i_{\ell+1},i_{\ell},\ldots,i_1)$ with an additional index $i_{\ell+1}\in\{1,\ldots,m\}$. When we compute the Lie derivative of $f_{I}\alpha$ along $f_{i_{\ell+1}}$ on $\mathbb{R}\times{M_{>}}$, we can use the linearity and the product rule for derivations. Thus, the inductive step is complete if we can show that
\[
\sum_{\tau\in\mathcal{T}_{\ell}}\sum_{k=0}^{\ell}\,f_{i_{\ell+1}}(e_{I\langle\tau[k]\rangle}\alpha_{i_k})\prod_{k\neq{\kappa}=0}^{\ell}e_{I\langle\tau[\kappa]\rangle}\alpha_{i_\kappa} \ = \ \sum_{\tau'\in\mathcal{T}_{\ell+1}}\prod_{\kappa=0}^{\ell+1}e_{I'\langle\tau'[\kappa]\rangle}\alpha_{i_\kappa}
\]
on $\mathbb{R}\times{M_>}$, where $\alpha_0:=\alpha$ and $\alpha_{i}:=(h_{i}\circ\psi)$ for $i=1,\ldots,m$. Note that for each tree $\tau'\in\mathcal{T}_{\ell+1}$, there is a unique pair $(\tau,k)$ in $\mathcal{T}_{\ell}\times\{0,1,\ldots,\ell\}$ such that $\tau'$ originates by adding the node $\ell+1$ as a child of node~$k$ in~$\tau$. Moreover, both $\mathcal{T}_{\ell+1}$ and $\mathcal{T}_{\ell}\times\{0,1,\ldots,\ell\}$ have $(\ell+1)!$ elements. Thus, it suffices to show that
\[
f_{i_{\ell+1}}(e_{I\langle\tau[k]\rangle}\alpha_{i_k})\prod_{k\neq{\kappa}=0}^{\ell}e_{I\langle\tau[\kappa]\rangle}\alpha_{i_\kappa} \ = \ \prod_{\kappa=0}^{\ell+1}e_{I'\langle\tau'[\kappa]\rangle}\alpha_{i_\kappa}
\]
on $\mathbb{R}\times{M_>}$ for every tree $\tau'\in\mathcal{T}_{\ell+1}$ and its associated pair $(\tau,k)$ as explained above. It is clear that $e_{I\langle\tau[\kappa]\rangle}\alpha_{i_\kappa}=e_{I'\langle\tau'[\kappa]\rangle}\alpha_{i_\kappa}$ for $\kappa=0,1,\ldots,\ell$ with $\kappa\neq{k}$ because these nodes remain unchanged. Since $\tau'$ originates from~$\tau$ by adding the node $\ell+1$ as a child of node~$k$, we also have
\[
f_{i_{\ell+1}}(e_{I\langle\tau[k]\rangle}\alpha_{i_k}) \ = \ (e_{I\langle\tau'[\ell+1]\rangle}\alpha_{i_{\ell+1}})(e_{I'\langle\tau'[k]\rangle}\alpha_{i_k})
\]
on $\mathbb{R}\times{M_>}$, which completes the proof of~\cref{eq:Lemma03}.

Next, we prove the asserted boundedness properties. Let $b_{\ell}$ denote the \emph{bushy tree} with $\ell$ leaves, i.e., the tree in $\mathcal{T}_{\ell}$ whose root has $\ell$ children (this tree is depicted in \cref{fig:rootedTrees:a}for $\ell=3$). For $\tau\in\mathcal{T}_{\ell}$ with $\tau\neq{b_{\ell}}$, define the time-varying function
\begin{equation}\label{eq:Lemma03:proof:01:B}
f_I^{\tau}\psi \ := \ \prod_{k=1}^{\ell}e_{I\langle\tau[k]\rangle}(h_{i_k}\circ\psi)
\end{equation}
on $M_{>}$. Then, we can write
\begin{equation}\label{eq:Lemma03:proof:01:A}
f_I\alpha - (e_{I}\alpha)\prod_{k=1}^{\ell}(h_{i_{k}}\circ\psi) \ = \ \sum_{b_{\ell}\neq\tau\in\mathcal{T}_{\ell}}(e_{I\langle\tau[0]\rangle}\alpha)\,f_I^{\tau}\psi
\end{equation}
on $\mathbb{R}\times{M_{>}}$. If $\ell\leq{r}$, we obtain by definition of the time derivative that
\begin{equation}\label{eq:Lemma03:proof:02:A}
\partial_tf_I\alpha - (\partial_te_{I}\alpha)\prod_{k=1}^{\ell}(h_{i_{k}}\circ\psi) \ = \ \sum_{\!\!b_{\ell}\neq\tau\in\mathcal{T}_{\ell}\!\!}\big((\partial_te_{I\langle\tau[0]\rangle}\alpha)\,(f_I^{\tau}\psi) + (e_{I\langle\tau[0]\rangle}\alpha)\,(\partial_tf_I^{\tau}\psi) \big)
\end{equation}
on $\mathbb{R}\times{M_{>}}$ with
\begin{equation}\label{eq:Lemma03:proof:02:B}
\partial_tf_I^{\tau}\psi \ = \ \sum_{k=1}^{\ell}(\partial_te_{I\langle\tau[k]\rangle}(h_{i_k}\circ\psi))\prod_{k\neq\kappa=1}^{\ell}e_{I\langle\tau[\kappa]\rangle}(h_{i_\kappa}\circ\psi).
\end{equation}
If $\ell\leq{r}$ and if $c=(U,\varphi)$ is a chart for~$M$, then we have
\begin{subequations}\label{eq:Lemma03:proof:03:A}
\begin{align}
\operatorname{Lip}_{c}(f_I\alpha) & \ \leq \ \operatorname{Lip}_{c}(e_{I}\alpha) \prod_{k=1}^{\ell}(h_{i_k}\circ\psi) + (e_{I}\alpha)\sum_{k=1}^{\ell}\operatorname{Lip}_{c}(h_{i_k}\circ\psi)\prod_{k\neq\kappa=1}^{\ell}(h_{i_\kappa}\circ\psi)\label{eq:Lemma03:proof:03:A:A} \allowdisplaybreaks \\
& \qquad + \sum_{b_{\ell}\neq\tau\in\mathcal{T}_{\ell}}\big(\operatorname{Lip}_{c}(e_{I\langle\tau[0]\rangle}\alpha)\,|f_I^{\tau}\psi| + |e_{I\langle\tau[0]\rangle}\alpha|\,\operatorname{Lip}_{c}(f_I^{\tau}\psi) \big)\label{eq:Lemma03:proof:03:A:B}
\end{align}
\end{subequations}
on $\mathbb{R}\times(M_{>}\cap{U})$ with
\begin{equation}\label{eq:Lemma03:proof:03:B}
\operatorname{Lip}_{c}(f_I^{\tau}\psi) \ \leq \ \sum_{k=1}^{\ell}\operatorname{Lip}_{c}(e_{I\langle\tau[k]\rangle}(h_{i_k}\circ\psi))\prod_{k\neq\kappa=1}^{\ell}|e_{I\langle\tau[\kappa]\rangle}(h_{i_\kappa}\circ\psi)|.
\end{equation}
Finally, if $\ell\leq{r}$, and if $f_0$ is tangent to the fibers of~$\pi$, then the Lie derivative of $h_{i_k}\circ\psi$ along $f_0$ vanishes and therefore we have
\begin{subequations}\label{eq:Lemma03:proof:04:A}
\begin{align}
\Big|\operatorname{L}(f_0,f_I\alpha)-\operatorname{L}(f_0,e_{I}\alpha)\prod_{k=1}^{\ell}(h_{i_{k}}\circ\psi)\Big| & \label{eq:Lemma03:proof:04:A:A} \allowdisplaybreaks \\
& \!\!\!\!\!\!\!\!\!\!\!\!\!\!\!\!\!\!\!\!\!\!\!\!\!\!\!\!\!\!\!\!\!\!\!\!\!\!\!\!\!\!\!\!\!\!\!\!\!\!\!\!\!\!\!\!\!\!\!\! \leq \ \sum_{b_{\ell}\neq\tau\in\mathcal{T}_{\ell}}\big(\operatorname{L}(f_0,e_{I\langle\tau[0]\rangle}\alpha)\,|f_I^{\tau}\psi| + |e_{I\langle\tau[0]\rangle}\alpha|\,\operatorname{L}(f_0,f_I^{\tau}\psi) \big)\label{eq:Lemma03:proof:04:A:B}
\end{align}
\end{subequations}
on $\mathbb{R}\times{M_{>}}$ with 
\begin{equation}\label{eq:Lemma03:proof:04:B}
\operatorname{L}(f_0,f_I^{\tau}\psi) \ \leq \ \sum_{k=1}^{\ell}\operatorname{L}(f_0,e_{I\langle\tau[k]\rangle}(h_{i_k}\circ\psi))\prod_{k\neq\kappa=1}^{\ell}|e_{I\langle\tau[\kappa]\rangle}(h_{i_\kappa}\circ\psi)|.
\end{equation}
Fix a tree $\tau\in\mathcal{T}_{\ell}$ which is not the bushy tree $b_{\ell}$. For every $k=0,1,\ldots,\ell$, we denote by $\#I\langle\tau[k]\rangle$ the number of children $\kappa\in\tau[k]$ for which $y^{\nu-1/2}h_{i_{\kappa}}^{(\nu)}(y)$ remains bounded as $y\downarrow0$ for $\nu=0,1,\ldots,r$. If there exists $k\in\{1,\ldots,\ell\}$ such that $\tau[k]$ is not empty and $h_{i_k}\equiv1$ on $(0,\infty)$, then $e_{I\langle\tau[k]\rangle}(h_{i_k}\circ\psi)\equiv0$ and therefore also $f^{\tau}_{I}\alpha\equiv0$ on $\mathbb{R}\times{M_{>}}$. In this case, there is nothing left to prove with respect to boundedness. For the rest of the proof we consider the opposite case in which, for every $k\in\{1,\ldots,\ell\}$, we have that $\tau[k]$ is empty or that $y^{\nu-1/2}h_{i_k}^{(\nu)}(y)$ remains bounded as $y\downarrow0$ for $\nu=0,1,\ldots,r$. Since~$\tau$ is assumed to be not a bushy tree, it follows for the root that $\#I\langle\tau[0]\rangle\geq1$. 
For  each of the remaining nodes $k=1,\ldots,\ell$ of $\tau$, there are two possibilities: either $\tau[k]$ is empty and $h_{i_k}\equiv1$ on $(0,\infty)$, which results in $e_{I\langle\tau[k]\rangle}(h_{i_k}\circ\psi)\equiv1$ on $\mathbb{R}\times{M_>}$, or otherwise, by \Cref{thm:Lemma02}~\ref{thm:Lemma02:A}, $e_{I\langle\tau[k]\rangle}(h_{i_k}\circ\psi)$ is uniformly bounded by a multiple of $\psi^{(1-\#I\langle\tau[k]\rangle)/2}$ on $K\setminus{E}$. Since $\#I=\#I\langle\tau[0]\rangle+\cdots+\#I\langle\tau[\ell]\rangle$, this implies that \cref{eq:Lemma03:proof:01:B} is uniformly bounded by a multiple of $\psi^{\#I\langle\tau[0]\rangle/2}$ on $K\setminus{E}$. 
Because of \Cref{thm:Lemma02}~\ref{thm:Lemma02:B} and~\ref{thm:Lemma02:D}, the same argument implies that also \cref{eq:Lemma03:proof:02:B} and \cref{eq:Lemma03:proof:04:B} are uniformly bounded by a multiple of $\psi^{\#I\langle\tau[0]\rangle/2}$ on $K\setminus{E}$. Since $\#I\langle\tau[0]\rangle\geq1$, this proves the asserted boundedness properties of~\cref{eq:Lemma03:proof:01:A}, \cref{eq:Lemma03:proof:02:A}, and \cref{eq:Lemma03:proof:04:A}, i.e., statements~\ref{thm:Lemma03:A}, \ref{thm:Lemma03:B}, and~\ref{thm:Lemma03:D} of \Cref{thm:Lemma03}. For the pointwise Lipschitz constant, we obtain from \Cref{thm:Lemma02}~\ref{thm:Lemma02:C} and a similar argument as above that \cref{eq:Lemma03:proof:03:B} is uniformly bounded by a multiple of $\psi^{\#I\langle\tau[0]\rangle/2-1}$ on $K\setminus{E}$. Thus, the same is true for the contribution in~\cref{eq:Lemma03:proof:03:A:B}. Since $\#I\langle\tau[0]\rangle\geq1$, this implies that \cref{eq:Lemma03:proof:03:B} is uniformly bounded on $K\setminus{E}$. Note that the same is true for the contribution on the right-hand side of~\cref{eq:Lemma03:proof:03:A:A} because of \Cref{thm:LipschitzDerivative,thm:Lemma00}. This proves the remaining part~\ref{thm:Lemma03:C} of \Cref{thm:Lemma03}.
\end{proof}
So far we have restricted our considerations to the submanifold $M_{>}$ of~$M$ where existence of the iterated Lie derivatives $f_I\alpha$ is ensured by the fact that the functions $h_i\circ\psi$ are of class~$C^r$. In the next step, we show that $f_I\alpha$ and its derivatives also exist on the zero set~$E$ of~$\psi$.
\begin{lemma}\label{thm:Lemma04}
Let $I=(i_{\ell},\ldots,i_1)$ be a multi-index of length $0<\ell\leq{r+1}$ with $i_1,\ldots,i_{\ell}\in\{1,\ldots,m\}$, and let $\alpha\in{C^{\infty}(M)}$. Then:
\begin{enumerate}[label=(\alph*)]
	\item\label{thm:Lemma04:A} $f_{I}\alpha=(e_I\alpha)\prod_{k=1}^{\ell}h_{i_k}(0)$ holds on $\mathbb{R}\times{E}$.
	\item\label{thm:Lemma04:B} If $\ell\leq{r}$, then $f_I\alpha$ is a locally uniformly Lipschitz continuous time-varying function on~$M$.
	\item\label{thm:Lemma04:C} If $\ell\leq{r}$, then $\partial_tf_{I}\alpha=(\partial_te_I\alpha)\prod_{k=1}^{\ell}h_{i_k}(0)$ holds on $\mathbb{R}\times{E}$.
	\item\label{thm:Lemma04:D} If $\ell\leq{r}$, and if $f_0$ is tangent to the fibers of~$\pi$, then $\operatorname{L}(f_0,f_I\alpha)=\operatorname{L}(f_0,e_{I}\alpha)\prod_{k=1}^{\ell}h_{i_{k}}(0)$ holds on $\mathbb{R}\times{E}$.
\end{enumerate}
\end{lemma}
\begin{proof}
We prove the lemma by induction on the length $\ell$ of~$I$. Clearly, statement~\ref{thm:Lemma04:A} is true for $\ell=1$. In the inductive step, we assume that statement~\ref{thm:Lemma04:A} is true for some $\ell\in\{1,\ldots,r\}$ and then we prove statements~\ref{thm:Lemma04:B}, \ref{thm:Lemma04:C}, \ref{thm:Lemma04:D} for $|I|=\ell$ and statement~\ref{thm:Lemma04:A} for $|I|=\ell+1$. To prove that $f_I\alpha$ is locally uniformly Lipschitz continuous, we use \Cref{thm:LipschitzCriterion}. For this reason, fix an arbitrary chart $c=(U,\varphi)$ for~$M$ and a compact subset~$K$ of~$U$. We already know from \Cref{thm:Lemma03}~\ref{thm:Lemma03:C} that $\operatorname{Lip}_c(f_I\alpha)$ is uniformly bounded on $K\setminus{E}$. Therefore, it remains to prove that $\operatorname{Lip}_c(f_I\alpha)$ is also uniformly bounded on $K\cap{E}$. By the induction hypothesis, we can define the time-varying function
\begin{equation}\label{eq:Lemma04:proof:01}
\Delta{f_I\alpha} \ := \ f_I\alpha - (e_{I}\alpha)\prod_{k=1}^{\ell}(h_{i_{k}}\circ\psi)
\end{equation}
on~$M$. By \Cref{thm:Lemma00}, the product $(e_{I}\alpha)\prod_{k=1}^{\ell}(h_{i_{k}}\circ\psi)$ on the right-hand side of~\cref{eq:Lemma04:proof:01} is a locally uniformly Lipschitz continuous time-varying function on~$M$. Thus, statement~\ref{thm:Lemma04:B} follows if we can prove that $\operatorname{Lip}_c(\Delta{f_I\alpha})$ is uniformly bounded on $K\cap{E}$. We conclude from the inductive hypothesis and \Cref{thm:Lemma03}~\ref{thm:Lemma03:A} that $\Delta{f_I\alpha}$ is locally uniformly bounded by a multiple of $\sqrt{\psi}$. Thus, there exist a compact neighborhood $K'\subseteq{U}$ of~$K$ and a constant $c>0$ such that
\[
\big|\Delta{f_I\alpha}(t,x')-\Delta{f_I\alpha}(t,x)\big| \ \leq \ c\,\big|\sqrt{\psi(x')}-\sqrt{\psi(x)}\big|
\]
for every $t\in\mathbb{R}$, every $x\in{K\cap{E}}$ and every $x'\in{K'}$. Since $\sqrt{\psi}$ is locally Lipschitz continuous (cf. \Cref{thm:LipschitzDerivative}), we conclude from \Cref{thm:LipschitzCriterion} that $\operatorname{Lip}_c(\Delta{f_I\alpha})$ is uniformly bounded on $K\cap{E}$, which completes the proof of statement~\ref{thm:Lemma04:B} for $|I|=\ell$. Statement~\ref{thm:Lemma04:C} follows directly from the induction hypothesis and the definition of the time derivative. To prove statement~\ref{thm:Lemma04:D}, assume that $f_0=e_0$ is tangent to the fibers of~$\pi$, and fix arbitrary $t\in\mathbb{R}$ and $x\in{M}$. Let $\gamma$ be the maximal integral curve of $f_0$ passing through~$x$ at $t$. Since $f_0$ is tangent to the fibers of $\pi$ and since $\psi$ is constant on the fibers of $\pi$ with $\psi(x)=0$, we have $\psi\circ\gamma\equiv0$, i.e., the curve $\gamma$ runs in~$E$. Therefore, we have
\[
\frac{1}{s}\big((f_I\alpha)(t,\gamma(t+s))-(f_I\alpha)(t,x)\big) \ = \ \frac{1}{s}\big((e_I\alpha)(t,\gamma(t+s))-(e_I\alpha)(t,x)\big)\prod_{k=1}^{\ell}h_{i_k}(0).
\]
for $s\neq0$ sufficiently close to $0$. In the limit $s\to0$, we obtain statement~\ref{thm:Lemma04:D} for $|I|=\ell$. To finish the inductive step, we have to show that for every $i_{\ell+1}=1,\ldots,m$, the Lie derivative of $f_I\alpha$ along $f_{i_{\ell+1}}$ is given as in~\ref{thm:Lemma04:A} for the multi-index $I'=(i_{\ell+1},\ldots,i_1)$. If $e_{i_{\ell+1}}$ is tangent to the fibers of~$\pi$, then a similar argument as in the proof of statement~\ref{thm:Lemma04:D} yields the asserted formula for $f_{I'}\alpha$ on $\mathbb{R}\times{E}$. Otherwise, we know that $f_{i_{\ell+1}}$ is locally uniformly bounded by a multiple of $\sqrt{\psi}$, which implies that $f_{i_{\ell+1}}$ vanishes on $\mathbb{R}\times{E}$. Then~\ref{thm:Lemma04:A} is also true, since both $f_{I'}\alpha$ and $(e_{I'}\alpha)\prod_{k=1}^{\ell+1}h_{i_{k}}(0)$ are identically equal to zero on $\mathbb{R}\times{E}$.
\end{proof}
The next lemma is an important step in the identification of \Cref{thm:outputConvergenceResult} as a particular case of \Cref{thm:convergenceResult}.
\begin{lemma}\label{thm:Lemma05}
$(f_0,f_1,\ldots,f_m)\in\operatorname{Sys}(M;m,r)$.
\end{lemma}
\begin{proof}
We have to check that $(f_0,f_1,\ldots,f_m)$ satisfies properties~\ref{item:def:Sys:i},\ref{item:def:Sys:ii},\ref{item:def:Sys:iii} in \Cref{def:Sys}. Property~\ref{item:def:Sys:i} is trivially satisfied, since $f_0=e_0$. Let $I=(i_{\ell},\ldots,i_1)$ be a multi-index of length $0<\ell\leq{r+1}$ with $i_1,\ldots,i_{\ell}\in\{1,\ldots,m\}$, and let $\alpha\in{C^{\infty}(M)}$. We know from \Cref{thm:Lemma03,thm:Lemma04} that the iterated Lie derivative $f_I\alpha$ exists as a time-varying function on~$M$ and that~$\Delta{f_I\alpha}$ in~\cref{eq:Lemma04:proof:01} is locally uniformly bounded by a multiple of $\sqrt{\psi}$. Note that the product on the right-hand side of~\cref{eq:Lemma04:proof:01} is locally uniformly bounded and therefore the same is true for $f_I\alpha$. Now suppose that $\ell\leq{r}$. Then, we know from \Cref{thm:Lemma04}~\ref{thm:Lemma04:B} that $f_I\alpha$ is locally uniformly Lipschitz continuous. Moreover, from \Cref{thm:Lemma03}~\ref{thm:Lemma03:B} and \Cref{thm:Lemma04}~\ref{thm:Lemma04:C}, we derive that the time derivative $\partial_tf_I\alpha$ exists as a time-varying function on~$M$ and that
\begin{equation}\label{eq:Lemma05:01}
\Delta\partial_t(f_I\alpha) \ := \ \partial_t(f_I\alpha) - (\partial_t(e_{I}\alpha))\prod_{k=1}^{\ell}(h_{i_{k}}\circ\psi)
\end{equation}
is locally uniformly bounded by a multiple of $\sqrt{\psi}$. Since the product on the right-hand side in the above equation is locally uniformly bounded, the same is true for $\partial_t(f_I\alpha)$. To complete the proof of property~\ref{item:def:Sys:ii} in \Cref{def:Sys}, it remains to show that $\partial_t(f_I\alpha)$ is continuous as a function on $\mathbb{R}\times{M}$. This is true on $\mathbb{R}\times{M_{>}}$ because equations~\cref{eq:Lemma03:proof:02:A,eq:Lemma03:proof:02:B,eq:Lemma02:proof:01:A,eq:Lemma02:proof:01:B} show that $\partial_t(f_I\alpha)$ is a sum of products of continuous functions functions on $\mathbb{R}\times{M}$. It remains to prove continuity of $\partial_t(f_I\alpha)$ at points in $\mathbb{R}\times{E}$. The product on the right-hand side of~\cref{eq:Lemma05:01} is clearly continuous. Also $\Delta\partial_t(f_I\alpha)$ is continuous at points of $\mathbb{R}\times{E}$ because it is locally uniformly bounded by a multiple of $\sqrt{\psi}$. Thus, $\partial_t(f_I\alpha)$ is continuous as a function on $\mathbb{R}\times{M}$. A similar argument, using the local uniform boundedness of $\Delta{f_I\alpha}$ by a multiple of $\sqrt{\psi}$, implies that $f_I\alpha$ is a Carath{\'e}odory function for $|I|=r+1$, which completes the proof that $(f_0,f_1,\ldots,f_m)$ satisfies all properties in \Cref{def:Sys}.
\end{proof}
To complete the proof of \Cref{thm:outputConvergenceResult}, we have to show that the boundedness properties in \Cref{thm:convergenceResult} are satisfied. This is done in the next lemma.
\begin{lemma}\label{thm:Lemma06}
Let~$I$ be a multi-index of length $0<\ell\leq{r}$, and let $i\in\{1,\ldots,m\}$. Let $\alpha\in{C^{\infty}(M)}$ be constant on the fibers of $\pi$. Then the following holds.
\begin{enumerate}[label=(\alph*)]
	\item\label{thm:Lemma06:A} $f_i\alpha$, $f_if_I\alpha$, and $\partial_tf_I\alpha$ are locally uniformly bounded by a multiple of $\sqrt{\psi}$.
	\item\label{thm:Lemma06:B} $f_0\alpha$ and $f_0(f_I\alpha)$ are locally uniformly bounded (in the sense of \Cref{thm:boundedAE}).
	\item\label{thm:Lemma06:C} If the drift $e_0$ is tangent to the fibers of~$\pi$ or if $e_0$ locally uniformly bounded by a multiple of $\sqrt{\psi}$, then $f_0\alpha$ and $f_0(f_I\alpha)$ are locally uniformly bounded by a multiple of $\sqrt{\psi}$ (in the sense of \Cref{thm:boundedAE}).
\end{enumerate}
\end{lemma}
\begin{proof}
Note that the function $\alpha$ is constant on the fibers of $\pi$. The asserted bounded properties of $f_0\alpha$ and $f_i\alpha$ are therefore clear. Next, we turn our attention to $\partial_tf_I\alpha$. We already know that $\Delta\partial_t(f_I\alpha)$ in~\cref{eq:Lemma05:01} is locally uniformly bounded by a multiple of $\sqrt{\psi}$. Since $\alpha$ is constant on the fibers of $\pi$, it follows that also the product on the right-hand side of~\cref{eq:Lemma05:01} is locally uniformly bounded by a multiple of $\sqrt{\psi}$. Thus, the same is true for $\partial_tf_I\alpha$. Using, \Cref{thm:Lemma03}~\ref{thm:Lemma03:A} and \Cref{thm:Lemma04}~\ref{thm:Lemma04:A}, a similar argument implies that also $f_if_I\alpha$ is locally uniformly bounded by a multiple of $\sqrt{\psi}$. It is left to prove the boundedness properties of $f_0(f_I\alpha)$ in parts~\ref{thm:Lemma06:B} and~\ref{thm:Lemma06:C}. Note that for every $(t,x)\in\mathbb{R}\times{M}$ at which $f_0(f_I\alpha)$ exists, we have
\[
|(f_0(f_I\alpha))(t,x)| \ = \ \operatorname{L}(f_0,f_I\alpha)(t,x)
\]
Fix an arbitrary chart $c=(U,\varphi)$ for~$M$. Let $f_0\varphi$ denote the componentwise Lie derivative of $\varphi$ along $f_0$, which is locally uniformly bounded by assumption on $f_0=e_0$. It follows from the definitions that
\[
\operatorname{L}(f_0,f_I\alpha) \ \leq \ \|f_0\varphi\|\,\operatorname{Lip}_c(f_I\alpha)
\]
on $\mathbb{R}\times{U}$. Since $f_I\alpha$ is locally uniformly Lipschitz continuous, we know that $\operatorname{Lip}_c(f_I\alpha)$ is locally uniformly bounded. Thus, $\operatorname{L}(f_0,f_I\alpha)$ is locally uniformly bounded as well. The above estimate also implies that $\operatorname{L}(f_0,f_I\alpha)$ is locally uniformly bounded by a multiple of $\sqrt{\psi}$ if $f_0$ is locally uniformly bounded by a multiple of $\sqrt{\psi}$. Finally, suppose that $f_0$ is tangent to the fibers of~$\pi$. To verify that $\operatorname{L}(f_0,f_I\alpha)$ is locally uniformly bounded by a multiple of $\sqrt{\psi}$, we use the estimate
\[
\operatorname{L}(f_0,f_I\alpha) \ \leq \ \Big|\operatorname{L}(f_0,f_I\alpha)-\operatorname{L}(f_0,e_{I}\alpha)\prod_{k=1}^{\ell}(h_{i_{k}}\circ\psi)\Big| + \operatorname{L}(f_0,e_{I}\alpha)\prod_{k=1}^{\ell}|h_{i_{k}}\circ\psi|.
\]
The first term on the right-hand side of the above inequality is locally uniformly bounded by a multiple of $\sqrt{\psi}$ by \Cref{thm:Lemma03}~\ref{thm:Lemma03:D} and \Cref{thm:Lemma04}~\ref{thm:Lemma04:D}. Since~$\alpha$ is tangent to the fibers of $\pi$, the second contribution on the right-hand side of the above inequality is also locally uniformly bounded by a multiple of $\sqrt{\psi}$. Thus, the same is true for $\operatorname{L}(f_0,f_I\alpha)$, which completes the proof of the lemma.
\end{proof}

\subsection{Proof of \texorpdfstring{\Cref{thm:ESCSinusoids}}{Proposition~\ref{thm:ESCSinusoids}}}\label{sec:ESCSinusoids}
We follow the proof of Theorem~5.1 in~\cite{Liu19972}. For each $k=1,\ldots,p$, we define two sets $\Omega(2k-1):=\Omega(2k):=\{\pm\omega_k\}$ and four $C^1$ maps $\eta_{\pm\omega_k,2k-1},\eta_{\pm\omega_k,2k}\colon\mathbb{R}\to\mathbb{C}$ by $\eta_{\pm\omega_k,2k-1}(t):=\sqrt{2\omega_k}\lambda_k(t)/2$ and $\eta_{\pm\omega_k,2k}(t):=\pm\sqrt{2\omega_k}/(2\mathrm{i})$, respectively. Then, we can write 
\[
u^j_{\ell}(t)=j^{\frac{1}{2}}\sum_{\omega\in\Omega(\ell)}\eta_{\omega,\ell}(t)\,\mathrm{e}^{\mathrm{i}j\omega{t}}
\]
for $\ell=1,\ldots,m$. Using integration by parts, we get
\[
\int_{t_0}^{t}u^j_{\ell}(s)\,\mathrm{d}s \ = \ j^{-\frac{1}{2}}\sum_{\omega\in\Omega(\ell)}\left(\vphantom{\prod^{.}}\right.\frac{\eta_{\omega,\ell}(t)\,\mathrm{e}^{\mathrm{i}j\omega{t}}}{\mathrm{i}\omega} - \frac{\eta_{\omega,\ell}(t_0)\,\mathrm{e}^{\mathrm{i}j\omega{t_0}}}{\mathrm{i}\omega} - \int_{t_0}^{t}\frac{\dot{\eta}_{\omega,\ell}(s)\,\mathrm{e}^{\mathrm{i}j\omega{s}}}{\mathrm{i}\omega}\,\mathrm{d}s\left.\vphantom{\prod^{.}}\right).
\]
Thus, by defining
\begin{align*}
v^j_{\ell}(t) & \ := \ -j^{-\frac{1}{2}}\sum_{\omega\in\Omega(\ell)}\frac{\dot{\eta}_{\omega,\ell}(t)\,\mathrm{e}^{\mathrm{i}j\omega{t}}}{\mathrm{i}\omega}, \allowdisplaybreaks \\
\widetilde{UV}\vphantom{UV}^j_{\ell}(t) & \ := \ -j^{-\frac{1}{2}}\sum_{\omega\in\Omega(\ell)}\frac{\eta_{\omega,\ell}(t)\,\mathrm{e}^{\mathrm{i}j\omega{t}}}{\mathrm{i}\omega},
\end{align*}
we ensure that $\dot{\widetilde{UV}}\vphantom{UV}^j_{\ell}=v^j_{\ell}-u^j_{\ell}$. If we multiply $u^j_{\ell}$ by $\widetilde{UV}\vphantom{UV}^j_{\ell'}$ and integrate we get
\[
\int_{t_0}^{t}u^j_{\ell}(s)\widetilde{UV}\vphantom{UV}^j_{\ell'}(s)\,\mathrm{d}s \ = \ -\sum_{(\omega,\omega')\in\Omega(\ell)\times\Omega(\ell')}\int_{t_0}^{t}\frac{\eta_{\omega,\ell}(s)\,\eta_{\omega',\ell'}(s)}{\mathrm{i}\omega'}\,\mathrm{e}^{\mathrm{i}j(\omega+\omega'){s}}\,\mathrm{d}s.
\]
It is straight forward to check that the terms with $\omega+\omega'=0$ in the above sum lead to $v_{\ell,\ell'}(t)$, i.e.,
\[
v_{\ell,\ell'}(t) \ = \ -\sum_{\substack{(\omega,\omega')\in\Omega(\ell)\times\Omega(\ell')\\\omega+\omega'=0}}\frac{\eta_{\omega,\ell}(t)\,\eta_{\omega',\ell'}(t)}{\mathrm{i}\omega'}.
\]
For the terms with $\omega+\omega'\neq0$, we apply again integration by parts and obtain
\begin{align*}
& \!\! \int_{t_0}^{t}\frac{\eta_{\omega,\ell}(s)\,\eta_{\omega',\ell'}(s)}{\mathrm{i}\omega'}\,\mathrm{e}^{\mathrm{i}j(\omega+\omega'){s}}\,\mathrm{d}s \ = \ -j^{-1}\int_{t_0}^{t}\frac{(\eta_{\omega,\ell}\eta_{\omega',\ell'})\dot{\vphantom{.}}(s)}{\mathrm{i}^2\omega'(\omega+\omega')}\,\mathrm{e}^{\mathrm{i}j(\omega+\omega'){s}}\,\mathrm{d}s \allowdisplaybreaks \\
& \qquad\quad + j^{-1}\left(\vphantom{\prod^{.}}\right.\frac{(\eta_{\omega,\ell}\eta_{\omega',\ell'})(t)}{\mathrm{i}^2\omega'(\omega+\omega')}\,\mathrm{e}^{\mathrm{i}j(\omega+\omega'){t}} - \frac{(\eta_{\omega,\ell}\eta_{\omega',\ell'})(t_0)}{\mathrm{i}^2\omega'(\omega+\omega')}\,\mathrm{e}^{\mathrm{i}j(\omega+\omega'){t_0}}\left.\vphantom{\prod^{.}}\right).
\end{align*}
To ensure $\dot{\widetilde{UV}}\vphantom{UV}^j_{\ell,\ell'}=v^j_{\ell,\ell'}-u^j_{\ell}\widetilde{UV}\vphantom{UV}^j_{\ell'}$, we define
\begin{align*}
v^j_{\ell,\ell'}(t) & \ := \ v_{\ell,\ell'}(t) + j^{-1}\sum_{\substack{(\omega,\omega')\in\Omega(\ell)\times\Omega(\ell')\\\omega+\omega'\neq0}}\frac{(\eta_{\omega,\ell}\eta_{\omega',\ell'})\dot{\vphantom{.}}(t)}{\mathrm{i}^2\omega'(\omega+\omega')}\,\mathrm{e}^{\mathrm{i}j(\omega+\omega'){t}} \allowdisplaybreaks \\
\widetilde{UV}\vphantom{UV}^j_{\ell,\ell'}(t) & \ := \ j^{-1}\sum_{\substack{(\omega,\omega')\in\Omega(\ell)\times\Omega(\ell')\\\omega+\omega'\neq0}}\frac{(\eta_{\omega,\ell}\eta_{\omega',\ell'})(t)}{\mathrm{i}^2\omega'(\omega+\omega')}\,\mathrm{e}^{\mathrm{i}j(\omega+\omega'){t}}.
\end{align*}
It is easy to verify that the above defined sequences of $v^j_I$ and $\widetilde{UV}\vphantom{UV}^j_I$ satisfy conditions~\ref{def:GDConvergence:1}-\ref{def:GDConvergence:3} in \Cref{def:GDConvergence} with respect to the $u^j_i$ and $v^j_I$ for $r=2$.

\subsection{Proof of \texorpdfstring{\Cref{thm:unicycleSinusoids}}{Proposition~\ref{thm:unicycleSinusoids}}}\label{sec:unicycleSinusoids}
For each $\nu=1,\ldots,N$, and $k=1,2,3$ we define the set $\Omega(3(\nu-1)+k):=\{\pm\omega_{\nu,k}\}$ and complex-valued constants
\begin{align*}
\eta_{\omega_{\nu,1}} & \ := \ \frac{\omega_{\nu,1}^{3/4}}{2}, & \eta_{\omega_{\nu,2}} & \ := \ \frac{\omega_{\nu,2}^{3/4}}{2\,\mathrm{i}} , & \eta_{\omega_{\nu,3}} & \ := \ {2^{13/8}}\,\frac{\omega_{\nu,3}^{3/4}}{2}, \allowdisplaybreaks \\
\eta_{-\omega_{\nu,1}} & \ := \ \frac{\omega_{\nu,1}^{3/4}}{2}, & \eta_{-\omega_{\nu,2}} & \ := \ -\frac{\omega_{\nu,2}^{3/4}}{2\,\mathrm{i}}, & \eta_{-\omega_{\nu,3}} & \ := \ {2^{13/8}}\,\frac{\omega_{\nu,3}^{3/4}}{2}
\end{align*}
Then, we can write
\[
u^j_{\ell}(t) \ = \ j^{\frac{3}{4}}\sum_{\omega\in\Omega(\ell)}\eta_{\omega,\ell}\,\mathrm{e}^{\mathrm{i}j\omega{t}}
\]
for $\ell=1,\ldots,m$. As a preparation for the subsequent calculations, we prove that for all $\ell_i\in\{1,\ldots,m\}$ and $\omega_i\in\Omega(\ell_i)$ with $i=1,2,3,4$, the following properties hold.
\begin{enumerate}[label=(\roman*)]
	\item\label{item:proofUnicycle:i} We always have $\omega_1\neq0$.
	\item\label{item:proofUnicycle:ii} If $\omega_1+\omega_2=0$, then $\ell_1=\ell_2$.
	\item\label{item:proofUnicycle:iii} We always have $\omega_1+\omega_2+\omega_3\neq0$.
	\item\label{item:proofUnicycle:iv} We have $\omega_1+\omega_2+\omega_3+\omega_4=0$ if and only if
	\begin{enumerate}[label=(\Alph*)]
		\item \label{item:proofUnicycle:iv:A} each $\omega_{i}$, $i=1,2,3,4$, is canceled out by its negative $-\omega_{i}$ (this case is referred to as \emph{pure cancelation by pairs} in~\cite{Liu19972}); or
		\item \label{item:proofUnicycle:iv:B} there is $\nu\in\{1,\ldots,N\}$ such that $(\omega_1,\omega_2,\omega_3,\omega_4)$ is either a permutation of $(\omega_{\nu,1},\omega_{\nu,2},-\omega_{\nu,3},-\omega_{\nu,3})$ or a permutation of $(-\omega_{\nu,1},-\omega_{\nu,2},\omega_{\nu,3},\omega_{\nu,3})$.
	\end{enumerate}
\end{enumerate}
Property~\ref{item:proofUnicycle:i} is clear. Property~\ref{item:proofUnicycle:ii} follows from the fact that the $\omega_{\nu,a}$, $\nu=1,\ldots,N$, $a=1,2,3$, are pairwise distinct. We prove the remaining properties~\ref{item:proofUnicycle:iii} and~\ref{item:proofUnicycle:iv} by means of a result from number theory. Recall that a positive integer is said to be \emph{square free} if in its prime factorization no prime factor occurs with an exponent larger than one. A finite sequence $x_1,\ldots,x_n$ of real numbers is said to be \emph{linearly independent} over the set of integers $\mathbb{Z}$ if $\lambda_1x_1+\cdots+\lambda_nx_n=0$ with $\lambda_1,\ldots,\lambda_n\in\mathbb{Z}$ is only satisfied for $\lambda_1=\cdots{\lambda_n}=0$. In the following, we will use the known fact (see~\cite{Besicovitch1940}) that square roots of pairwise distinct square free integers $>1$ are linearly independent over $\mathbb{Z}$. Note that the numbers $\kappa_2,2\kappa_2,\kappa_3,2\kappa_3,\ldots$ are pairwise distinct square free integers $>1$. Moreover each $\omega\in\Omega(\ell)$ is of the form $\omega=a\sqrt{\kappa_{\nu+1}}+b\sqrt{2\kappa_{\nu+1}}$ for some $\nu\in\{1,\ldots,N\}$ and for some $(a,b)$ from the set $\Lambda:=\{\pm(3,2),\pm(1,0),\pm(2,1)\}$. To prove~\ref{item:proofUnicycle:iii}, we write each $\omega_i$ uniquely as $\omega_i=a_i\sqrt{\kappa_{\nu_i+1}}+b_i\sqrt{2\kappa_{\nu_i+1}}$ with $(a_i,b_i)\in\Lambda$. If the numbers $\nu_1,\nu_2,\nu_3$ are not all equal, then $\omega_1+\omega_2+\omega_3$ is clearly a non\-trivial linear combination of the $\sqrt{\kappa_{\nu_i+1}},\sqrt{2\kappa_{\nu_i+1}}$ and therefore non\-zero. Otherwise, i.e., if the numbers $\nu_1,\nu_2,\nu_3$ are all equal, then $\omega_1+\omega_2+\omega_3=0$ would imply $a_1+a_2+a_3=0$ and $b_1+b_2+b_3=0$. However, it is easy to check that these equations have no solutions for $(a_i,b_i)\in\Lambda$. It remains to prove property~\ref{item:proofUnicycle:iv}. Suppose that $\omega_1+\omega_2+\omega_3+\omega_4=0$ and that condition~\ref{item:proofUnicycle:iv:A} of pure cancelation by pairs is not satisfied. Thus, we have to show that we are in the situation of case~\ref{item:proofUnicycle:iv:B}. Again we write $\omega_i=a_i\sqrt{\kappa_{\nu_i+1}}+b_i\sqrt{2\kappa_{\nu_i+1}}$ with $(a_i,b_i)\in\Lambda$. Since we have ruled out pure cancelation by pairs, the linear independence of the $\sqrt{\kappa_{\nu_i+1}},\sqrt{2\kappa_{\nu_i+1}}$ implies that the numbers $\nu_1,\ldots,\nu_4$ have to be all equal. This in return implies that $a_1+\cdots+a_4=0$ and $b_1+\cdots+b_4=0$. It is straight forward to check that, under the assumption of no canceling by pairs, there are exactly~12 solutions of $a_1+\cdots+a_4=0$ and $b_1+\cdots+b_4=0$ with $(a_i,b_i)\in\Lambda$, namely when $((a_1,b_1),\ldots,(a_4,b_4))$ is either a permutation of $(+(3,2),+(1,0),-(2,1),-(2,1))$ or a permutation of $(-(3,2),-(1,0),+(2,1),+(2,1))$. These 12 solutions coincide with the 12 possible assignments to $(\omega_1,\omega_2,\omega_3,\omega_4)$ in~\ref{item:proofUnicycle:iv:B}.

Now we can apply the same procedure as in the proof of Theorem~5.1 in~\cite{Liu19972}. For a single index $\ell\in\{1,\ldots,m\}$, we compute
\[
\int_{t_0}^{t}u^j_{\ell}(s)\,\mathrm{d}s \ = \ \int_{t_0}^{t}v^j_{\ell}(s)\,\mathrm{d}s - \big(\widetilde{UV}\vphantom{UV}^j_{\ell}(t)-\widetilde{UV}\vphantom{UV}^j_{\ell}(t_0)\big)
\]
with $v^j_{\ell}:=v_{\ell}\equiv0$ and
\[
\widetilde{UV}\vphantom{UV}^j_{\ell}(t) \ := \ -j^{-\frac{1}{4}}\sum_{\omega\in\Omega(\ell)}\frac{\eta_{\omega,\ell}\,\mathrm{e}^{\mathrm{i}j\omega{t}}}{\mathrm{i}\omega},
\]
because of property~\ref{item:proofUnicycle:i}.

Fix an arbitrary multi-index $I=(\ell_1,\ell_2)$ with $\ell_1,\ell_2\in\{1,\ldots,m\}$. If we multiply $u^j_{\ell_1}$ by $\widetilde{UV}\vphantom{UV}^j_{\ell_2}$ and integrate we get
\[
\int_{t_0}^{t}u^j_{\ell_1}(s)\widetilde{UV}\vphantom{UV}^j_{\ell_2}(s)\,\mathrm{d}s \ = \ -j^{\frac{2}{4}}\sum_{(\omega_1,\omega_2)\in\Omega(\ell_1)\times\Omega(\ell_2)}\int_{t_0}^{t}\frac{\eta_{\omega_1,\ell_1}\,\eta_{\omega_2,\ell_2}}{\mathrm{i}\omega_2}\,\mathrm{e}^{\mathrm{i}j(\omega_1+\omega_2){s}}\,\mathrm{d}s.
\]
Let $(\omega_1,\omega_2)\in\Omega(\ell_1)\times\Omega(\ell_2)$. Then, also $(-\omega_1,-\omega_2)\in\Omega(\ell_1)\times\Omega(\ell_2)$. If $\omega_1+\omega_2=0$, then property~\ref{item:proofUnicycle:ii} implies $\ell_1=\ell_2$, and therefore $\frac{\eta_{\omega_1,\ell_1}\,\eta_{\omega_2,\ell_2}}{\mathrm{i}\omega_2}=\frac{|\eta_{\omega_2,\ell_2}|^2}{\mathrm{i}\omega_2}$ and $\frac{\eta_{-\omega_1,\ell_1}\,\eta_{-\omega_2,\ell_2}}{-\mathrm{i}\omega_2}=-\frac{|\eta_{\omega_2,\ell_2}|^2}{\mathrm{i}\omega_2}$. Thus, all terms with $\omega_1+\omega_2=0$ in the above sum add up to $0$, and therefore the summation reduces to the set $\Omega(I)$ of all $(\omega_1,\omega_2)\in\Omega(\ell_1)\times\Omega(\ell_2)$ with $\omega_1+\omega_2\neq0$. We compute the remaining integrals and obtain
\[
\int_{t_0}^{t}u^j_{\ell_1}(s)\widetilde{UV}\vphantom{UV}^j_{\ell_2}(s)\,\mathrm{d}s \ = \ \int_{t_0}^{t}v^j_{I}(s)\,\mathrm{d}s - \big(\widetilde{UV}\vphantom{UV}^j_{I}(t)-\widetilde{UV}\vphantom{UV}^j_{I}(t_0)\big)
\]
with $v^j_{I}:=v_{I}\equiv0$ and
\[
\widetilde{UV}\vphantom{UV}^j_{I}(t) \ := \ (-1)^2j^{-\frac{2}{4}}\sum_{\hat{\omega}\in\Omega(I)}\frac{\eta_{\hat{\omega},I}}{\mathrm{i}^2h(\hat{\omega})}\,\mathrm{e}^{\mathrm{i}j(\sum\hat{\omega}){t}},
\]
where we have used the abbreviations $\eta_{\hat{\omega},I}:=\eta_{\omega_1,\ell_1}\eta_{\omega_2,\ell_2}$, $h(\hat{\omega}):=\omega_2(\omega_2+\omega_1)$, and $\sum\hat{\omega}:=\omega_1+\omega_2$ for $\hat{\omega}=(\omega_1,\omega_1)\in\Omega(I)$.

Fix an arbitrary multi-index $I=(\ell_1,\ell_2,\ell_3)$ with $\ell_1,\ell_2,\ell_3\in\{1,\ldots,m\}$ and write $\bar{I}:=(\ell_2,\ell_3)$. Multiplying $u^j_{\ell_1}$ by $\widetilde{UV}\vphantom{UV}^j_{\bar{I}}$ and integrating we get
\[
\int_{t_0}^{t}u^j_{\ell_1}(s)\widetilde{UV}\vphantom{UV}^j_{\bar{I}}(s)\,\mathrm{d}s \ = \ (-1)^2j^{\frac{1}{4}}\sum_{(\omega_1,\hat{\omega})\in\Omega(\ell_1)\times\Omega(\bar{I})}\int_{t_0}^{t}\frac{\eta_{\omega_1,\ell_1}\,\eta_{\hat{\omega},\bar{I}}}{\mathrm{i}^2h(\hat{\omega})}\,\mathrm{e}^{\mathrm{i}j(\omega_1+\sum\hat{\omega}){s}}\,\mathrm{d}s.
\]
By definition of~$\Omega(\bar{I})$ and because of property~\ref{item:proofUnicycle:iii}, the set $\Omega(\ell_1)\times\Omega(\bar{I})$ in the above sum can be identified with the set $\Omega(I)$ of all $\hat{\omega}=(\omega_1,\omega_2,\omega_3)$ in $\Omega(\ell_1)\times\Omega(\ell_2)\times(\Omega_3)$ with $h(\hat{\omega}):=\omega_3(\omega_3+\omega_2)(\omega_3+\omega_2+\omega_1)\neq0$. We compute the integrals and obtain
\[
\int_{t_0}^{t}u^j_{\ell_1}(s)\widetilde{UV}\vphantom{UV}^j_{\bar{I}}(s)\,\mathrm{d}s \ = \ \int_{t_0}^{t}v^j_{I}(s)\,\mathrm{d}s - \big(\widetilde{UV}\vphantom{UV}^j_{I}(t)-\widetilde{UV}\vphantom{UV}^j_{I}(t_0)\big)
\]
with $v^j_{I}:=v_{I}\equiv0$ and
\[
\widetilde{UV}\vphantom{UV}^j_{I}(t) \ := \ (-1)^3j^{-\frac{3}{4}}\sum_{\hat{\omega}\in\Omega(I)}\frac{\eta_{\hat{\omega},I}}{\mathrm{i}^3h(\hat{\omega})}\,\mathrm{e}^{\mathrm{i}j(\sum\hat{\omega}){t}},
\]
where we have used the abbreviations $\eta_{\hat{\omega},I}:=\eta_{\omega_1,\ell_1}\eta_{\omega_2,\ell_2}\eta_{\omega_3,\ell_3}$ and $\sum\hat{\omega}:=\omega_1+\omega_2+\omega_3$ for $\hat{\omega}=(\omega_1,\omega_2,\omega_3)\in\Omega(I)$.

Fix an arbitrary multi-index $I=(\ell_1,\ldots,\ell_4)$ with $\ell_1,\ldots,\ell_4\in\{1,\ldots,m\}$ and write $\bar{I}:=(\ell_2,\ell_3,\ell_4)$. Multiplying $u^j_{\ell_1}$ by $\widetilde{UV}\vphantom{UV}^j_{\bar{I}}$ and integrating we get
\[
\int_{t_0}^{t}u^j_{\ell_1}(s)\widetilde{UV}\vphantom{UV}^j_{\bar{I}}(s)\,\mathrm{d}s \ = \ (-1)^3j^{0}\sum_{(\omega_1,\hat{\omega})\in\Omega(\ell_1)\times\Omega(\bar{I})}\int_{t_0}^{t}\frac{\eta_{\omega_1,\ell_1}\,\eta_{\hat{\omega},\bar{I}}}{\mathrm{i}^3h(\hat{\omega})}\,\mathrm{e}^{\mathrm{i}j(\omega_1+\sum\hat{\omega}){s}}\,\mathrm{d}s.
\]
We let $\Omega(I)$ be the set of all $\hat{\omega}=(\omega_1,\ldots,\omega_4)$ in $\Omega(\ell_1)\times\cdots\times(\Omega_4)$ with $h(\hat{\omega}):=\omega_4(\omega_4+\omega_3)\cdots(\omega_4+\omega_3+\omega_2+\omega_1)\neq0$. Again, we introduce the abbreviations $\eta_{\hat{\omega},I}:=\eta_{\omega_1,\ell_1}\cdots\eta_{\omega_4,\ell_4}$ and $\sum\hat{\omega}:=\omega_1+\cdots+\omega_4$ for $\hat{\omega}=(\omega_1,\ldots,\omega_4)\in\Omega(I)$. Then, we can write the above integral as
\[
\int_{t_0}^{t}u^j_{\ell_1}(s)\widetilde{UV}\vphantom{UV}^j_{\bar{I}}(s)\,\mathrm{d}s \ = \ \int_{t_0}^{t}v^j_{I}(s)\,\mathrm{d}s - \big(\widetilde{UV}\vphantom{UV}^j_{I}(t)-\widetilde{UV}\vphantom{UV}^j_{I}(t_0)\big)
\]
with
\begin{align*}
v^j_I(t) & \ := \ (-1)^3j^{0}\sum_{\substack{(\omega_1,(\omega_2,\omega_3,\omega_4))\in\Omega(\ell_1)\times\Omega(\bar{I})\\\omega_1+\omega_2+\omega_3+\omega_4=0}}\frac{\eta_{\omega_1,\ell_1}\,\eta_{\omega_2,\ell_2}\,\eta_{\omega_3,\ell_3}\,\eta_{\omega_4,\ell_4}}{\mathrm{i}^3\omega_4(\omega_4+\omega_3)(\omega_4+\omega_3+\omega_2)}, \allowdisplaybreaks \\
\widetilde{UV}\vphantom{UV}^j_{I}(t) & \ := \ (-1)^4j^{-1}\sum_{\hat{\omega}\in\Omega(I)}\frac{\eta_{\hat{\omega},I}}{\mathrm{i}^4h(\hat{\omega})}\,\mathrm{e}^{\mathrm{i}j(\sum\hat{\omega}){t}}.
\end{align*}
We show that $v^j_I(t)=v_I(t)$ with $v_I(t)$ defined as in~\Cref{thm:unicycleSinusoids}. We start with the case in which the set of $(\omega_1,(\omega_2,\omega_3,\omega_4))\in\Omega(\ell_1)\times\Omega(\bar{I})$ with $\omega_1+\omega_2+\omega_3+\omega_4=0$ is not empty. We know from property~\ref{item:proofUnicycle:iv} that $\omega_1+\omega_2+\omega_3+\omega_4=0$ holds if and only if either condition~\ref{item:proofUnicycle:iv:A} or condition~\ref{item:proofUnicycle:iv:B} is satisfied. First, consider the case in which some $(\omega_1,\omega_2,\omega_3,\omega_4)$ satisfies condition~\ref{item:proofUnicycle:iv:A}. Then, there are $i,i'\in\{1,2,3,4\}$ and $\ell,\ell'\in\{1,\ldots,m\}$ such that
\[
\eta_{\omega_1,\ell_1}\,\eta_{\omega_2,\ell_2}\,\eta_{\omega_3,\ell_3}\,\eta_{\omega_4,\ell_4} \ = \ |\eta_{\omega_{i},\ell}|^2\,|\eta_{\omega_{i'},\ell'}|^2.
\]
Note that also $(-\omega_1,(-\omega_2,-\omega_3,-\omega_4))\in\Omega(\ell_1)\times\Omega(\bar{I})$ and
\[
\eta_{-\omega_1,\ell_1}\,\eta_{-\omega_2,\ell_2}\,\eta_{-\omega_3,\ell_3}\,\eta_{-\omega_4,\ell_4} \ = \ |\eta_{\omega_{i},\ell}|^2\,|\eta_{\omega_{i'},\ell'}|^2.
\]
It follows that terms in the sum on the right-hand side of~$v^j_I(t)$ for which condition~\ref{item:proofUnicycle:iv:A} holds cancel each other. Thus, the only possibly non\-vanishing contribution comes from terms with $(\omega_1,\omega_2,\omega_3,\omega_4)$ which satisfy condition~\ref{item:proofUnicycle:iv:B}. In this case, there exists $\nu\in\{1,\ldots,N\}$ such that $(\omega_1,\omega_2,\omega_3,\omega_4)$ is either a permutation of $(\omega_{\nu,1},\omega_{\nu,2},-\omega_{\nu,3},-\omega_{\nu,3})$ or a permutation of $(-\omega_{\nu,1},-\omega_{\nu,2},\omega_{\nu,3},\omega_{\nu,3})$. Thus, $I=(k_1,k_2,k_3,k_4)_{\nu}$, where $(k_1,k_2,k_3,k_4)_{\nu}$ is given by~\cref{eq:unicycleSinusoids:04} and $(k_1,k_2,k_3,k_4)$ is a permutation of $(1,2,3,3)$. A direct computation for the 12~permutations $(k_1,k_2,k_3,k_4)$ of $(1,2,3,3)$ shows that
\[
v^j_{(k_1,k_2,k_3,k_4)_{\nu}}(t) \ = \ v_{(k_1,k_2,k_3,k_4)_{\nu}}(t)
\]
with $v_{(k_1,k_2,k_3,k_4)_{\nu}}(t)$ as in~\cref{eq:unicycleSinusoids:03}. If there is no $(\omega_1,(\omega_2,\omega_3,\omega_4))\in\Omega(\ell_1)\times\Omega(\bar{I})$ for which $\omega_1+\omega_2+\omega_3+\omega_4=0$ holds, then $I$ is not of the form $(k_1,k_2,k_3,k_4)_{\nu}$ with $(k_1,k_2,k_3,k_4)$ being a permutation of $(1,2,3,3)$, and therefore $v^j_I(t)=0=v_I(t)$.

It is easy to verify that the above defined sequences of $v^j_I$ and $\widetilde{UV}\vphantom{UV}^j_I$ satisfy conditions~\ref{def:GDConvergence:1}-\ref{def:GDConvergence:3} in \Cref{def:GDConvergence} with respect to the $u^j_i$ and $v^j_I$ for $r=4$.

\section*{Acknowledgments}
The author would like to thank Professor U.~Helmke for proposing to him the topic.

\bibliographystyle{abbrv}
\bibliography{../../bibFile}

\begin{thebibliography}{10}

\bibitem{AbrahamBook}
R.~Abraham, J.~E. Marsden, and T.~Ratiu.
\newblock {\em {Manifolds, Tensor Analysis, and Applications}}, volume~75 of
  {\em Applied Mathematical Sciences}.
\newblock Springer, New York, second edition, 1988.

\bibitem{Besicovitch1940}
A.~S. Besicovitch.
\newblock {On the Linear Independence of Fractional Powers of Integers}.
\newblock {\em Journal of the London Mathematical Society}, 15(1):3--6, 1940.

\bibitem{Dorfler2009}
F.~D\"orfler and B.~A. Francis.
\newblock {Formation control of autonomous robots based on cooperative
  behavior}.
\newblock In {\em {Proceedings of the 2009 European Control Conference}}, pages
  2432--2437, 2009.

\bibitem{Durand2010}
E.~Durand-Cartagena and J.~A. Jaramillo.
\newblock {Pointwise Lipschitz functions on metric spaces}.
\newblock {\em Journal of Mathematical Analysis and Applications},
  363(2):525--548, 2010.

\bibitem{Duerr2010}
H.-B. D\"urr.
\newblock {Distributed Positioning of Autonomous Mobile Sensors with
  Application to the Coverage Problem}.
\newblock Master's thesis, Royal Institute of Technology, Stockholm, Sweden,
  2010.

\bibitem{Duerr2015}
H.-B. D\"urr, M.~Krsti{\'c}, A.~Scheinker, and C.~Ebenbauer.
\newblock {Singularly Perturbed Lie Bracket Approximation}.
\newblock {\em IEEE Transactions on Automatic Control}, 60(12):3287--3292,
  2015.

\bibitem{Duerr2013}
H.-B. D\"urr, M.~Stankovic, C.~Ebenbauer, and K.~J. Johansson.
\newblock {Lie Bracket Approximation of Extremum Seeking Systems}.
\newblock {\em Automatica}, 49(6):1538--1552, 2013.

\bibitem{Duerr2014}
H.-B. D\"urr, M.~Stankovic, K.~H. Johansson, and C.~Ebenbauer.
\newblock {Extremum Seeking on Submanifolds in the Euclidean Space}.
\newblock {\em Automatica}, 50(10):2591--2596, 2014.

\bibitem{Grossman19892}
R.~L. Grossman and R.~G. Larson.
\newblock {Labeled trees and the efficient computation of derivations}.
\newblock In {\em {Proceedings of 1989 International Symposium on Symbolic and
  Algebraic Computation, ACM}}, pages 74--80, 1989.

\bibitem{Grossman19901}
R.~L. Grossman and R.~G. Larson.
\newblock {Solving nonlinear equations from higher order derivations in linear
  stages}.
\newblock {\em Advances in Applied Mathematics}, 82(2):180--202, 1990.

\bibitem{Grossman1992}
R.~L. Grossman and R.~G. Larson.
\newblock {Symbolic computation of derivations using labelled trees}.
\newblock {\em Journal of Symbolic Computation}, 13(5):511--523, 1992.

\bibitem{Grushkovskaya2017}
V.~Grushkovskaya, A.~Zuyev, and C.~Ebenbauer.
\newblock {On a class of generating vector fields for the extremum seeking
  problem: Lie bracket approximation and stability properties}.
\newblock arXiv:1703.02348, 2017.

\bibitem{HaleBook}
J.~K. Hale.
\newblock {\em {Ordinary Differential Equations}}.
\newblock John Wiley \& Sons, Inc., New York, second edition, 1980.

\bibitem{Hardy2006}
M.~Hardy.
\newblock {Combinatorics of Partial Derivatives}.
\newblock {\em Electronic Journal of Combinatorics}, 13(1), 2006.

\bibitem{Krick2009}
L.~Krick, M.~E. Broucke, and B.~A. Francis.
\newblock {Stabilization of infinitesimally rigid formations of multi-robot
  networks}.
\newblock {\em International Journal of Control}, 82(3):423--439, 2009.

\bibitem{Kurzweil1987}
J.~Kurzweil and J.~Jarn\'{\i}k.
\newblock {Limit Processes in Ordinary Differential Equations}.
\newblock {\em Zeitschrift f\"ur angewandte Mathematik und Physik},
  38(2):241--256, 1987.

\bibitem{Kurzweil19882}
J.~Kurzweil and J.~Jarn\'{\i}k.
\newblock {A convergence effect in ordinary differential equations}.
\newblock In V.~S. Korolyuk, editor, {\em {Asymptotic methods of mathematical
  physics}}, pages 134--144. Naukova Dumka, Kiev, 1988.

\bibitem{Kurzweil1988}
J.~Kurzweil and J.~Jarn\'{\i}k.
\newblock {Iterated Lie Brackets in Limit Processes in Ordinary Differential
  Equations}.
\newblock {\em Results in Mathematics}, 14(1-2):125--137, 1988.

\bibitem{LeeBook}
J.~M. Lee.
\newblock {\em {Introduction to Smooth Manifolds}}, volume 218 of {\em Graduate
  Texts in Mathematics}.
\newblock Springer, New York, second edition, 2012.

\bibitem{LernerBook}
N.~Lerner.
\newblock {\em {Metrics on the Phase Space and Non-Selfadjoint
  Pseudo-Differential Operators}}, volume~3 of {\em Pseudo-Differential
  Operators}.
\newblock Birkh\"auser, Basel, 2010.

\bibitem{LiuBook}
S.-J. Liu and M.~Krsti\'c.
\newblock {\em {Stochastic Averaging and Stochastic Extremum Seeking}}.
\newblock Communications and Control Engineering. Springer, London, 2012.

\bibitem{Liu19932}
W.~Liu.
\newblock {An Approximation Algorithm for Nonholonomic Systems}.
\newblock Technical report, Rutgers Center for Systems and Control, New
  Brunswick, 1993.
\newblock Report SYCON 93-11.

\bibitem{Liu19972}
W.~Liu.
\newblock {An Approximation Algorithm for Nonholonomic Systems}.
\newblock {\em SIAM Journal on Control and Optimization}, 35(4):1328--1365,
  1997.

\bibitem{Liu1997}
W.~Liu.
\newblock {Averaging Theorems for Highly Oscillatory Differential Equations and
  Iterated Lie Brackets}.
\newblock {\em SIAM Journal on Control and Optimization}, 35(6):1989--2020,
  1997.

\bibitem{Liu1999}
W.~Liu and H.~J. Sussmann.
\newblock {Continuous Dependence with Respect to the Input of Trajectories of
  Control-Affine Systems}.
\newblock {\em SIAM Journal on Control and Optimization}, 37(3):777--803, 1999.

\bibitem{MichelBook}
A.~N. Michel, L.~Hou, and D.~Liu.
\newblock {\em {Stability of Dynamical Systems}}.
\newblock Systems \& Control: Foundations \& Applications. Birkh\"auser, Basel,
  second edition, 2015.

\bibitem{Moreau19994}
L.~Moreau and D.~Aeyels.
\newblock {Asymptotic methods in stability analysis and control}.
\newblock In D.~Aeyels, F.~Lamnabhi-Lagarrigue, and A.~Schaft, editors, {\em
  {Stability and Stabilization of Nonlinear Systems}}, volume 246 of {\em
  Lecture Notes in Control and Information Sciences}, pages 201--213. Springer,
  London, 1999.

\bibitem{Moreau20002}
L.~Moreau and D.~Aeyels.
\newblock {Asymptotic methods in the stability analysis of parametrized
  homogeneous flows}.
\newblock {\em Automatica}, 36(8):1213--1218, 2000.

\bibitem{Moreau2000}
L.~Moreau and D.~Aeyels.
\newblock {Practical stability and stabilization}.
\newblock {\em IEEE Transactions on Automatic Control}, 45(8):1554--1558, 2000.

\bibitem{Moreau2003}
L.~Moreau and D.~Aeyels.
\newblock {Trajectory-Based Local Approximations of Ordinary Differential
  Equations}.
\newblock {\em SIAM Journal on Control and Optimization}, 41(6):1922--1945,
  2003.

\bibitem{Morin1999}
P.~Morin, J.-B. Pomet, and C.~Samson.
\newblock {Design of Homogeneous Time-Varying Stabilizing Control Laws for
  Driftless Controllable Systems Via Oscillatory Approximation of Lie Brackets
  in Closed Loop}.
\newblock {\em SIAM Journal on Control and Optimization}, 38(1):22--49, 1999.

\bibitem{Oh2014}
K.~K. Oh and H.~S. Ahn.
\newblock {Distance-based undirected formations of single-integrator and
  double-integrator modeled agents in $n$-dimensional space}.
\newblock {\em International Journal of Robust and Nonlinear Control},
  24(12):1809--1820, 2014.

\bibitem{Park2014}
M.~C. Park, K.~Jeong, and H.~S. Ahn.
\newblock {Exponential stabilization of infinitesimally rigid formations}.
\newblock In {\em {Proceedings of the 2014 International Conference on Control,
  Automation and Information Sciences}}, pages 36--40, 2014.

\bibitem{ReutenauerBook}
C.~Reutenauer.
\newblock {\em {Free Lie Algebras}}, volume~7 of {\em London Mathematical
  Society Monographs}.
\newblock Clarendon Press, New York, 1993.

\bibitem{Scheinker2013}
A.~Scheinker and M.~Krsti{\'c}.
\newblock {Minimum-Seeking for CLFs: Universal Semiglobally Stabilizing
  Feedback Under Unknown Control Directions}.
\newblock {\em IEEE Transactions on Automatic Control}, 58(5):1107--1122, 2013.

\bibitem{Scheinker2014}
A.~Scheinker and M.~Krsti{\'c}.
\newblock {Extremum seeking with bounded update rates}.
\newblock {\em Systems \& Control Letters}, 63:25--31, 2014.

\bibitem{Scheinker20132}
A.~Scheinker and M.~Krsti{\'c}.
\newblock {Non-$C^2$ Lie Bracket Averaging for Nonsmooth Extremum Seekers}.
\newblock {\em Journal of Dynamic Systems, Measurements, and Control},
  136:011010, 2014.

\bibitem{ScheinkerBook}
A.~Scheinker and M.~Krsti{\'c}.
\newblock {\em {Model-Free Stabilization by Extremum Seeking}}.
\newblock Springer Briefs in Control, Automation and Robotics. Springer, Chur,
  2017.

\bibitem{StanleyBook}
R.~P. Stanley.
\newblock {\em {Enumerative Combinatorics}}.
\newblock Cambridge Studies in Advanced Mathematics. Cambridge University
  Press, New York, second edition, 2012.

\bibitem{Sun2016}
Z.~Sun, S.~Mou, B.~D.~O. Anderson, and M.~Cao.
\newblock {Exponential stability for formation control systems with generalized
  controllers: A unified approach}.
\newblock {\em Systems \& Control Letters}, 93:50 -- 57, 2016.

\bibitem{Sussmann1998}
H.~J. Sussmann.
\newblock {An Introduction to the Coordinate-Free Maximum Principle}.
\newblock In B.~Jakubczyk and W.~Respondek, editors, {\em {Geometry of Feedback
  and Optimal Control}}, number 207 in Lecture notes in pure and applied
  mathematics, chapter~12, pages 463--557. Marcel Dekker, New York, 1998.

\bibitem{Sussmann1991}
H.~J. Sussmann and W.~Liu.
\newblock {Limits of Highly Oscillatory Controls and the Approximation of
  General Paths by Admissible Trajectories}.
\newblock In {\em {Proceedings of the 30th IEEE Conference on Decision and
  Control}}, pages 437--442, 1991.

\bibitem{Suttner2017}
R.~Suttner and S.~Dashkovskiy.
\newblock {Exponential Stability for Extremum Seeking Control Systems}.
\newblock In {\em {Proceedings of the 20th IFAC World Congress}}, pages
  15464--15470, 2017.

\bibitem{Tilbury1992}
D.~Tilbury, J.-P. Laumond, R.~M. Murray, S.~Sastry, and G.~Walsh.
\newblock {Steering car-like systems with trailers using sinusoids}.
\newblock In {\em {Proceedings 1992 IEEE International Conference on Robotics
  and Automation}}, pages 1993--1998, 1992.

\bibitem{ZhangBook}
C.~Zhang and R.~Ord{\'{o}}{\~{n}}ez.
\newblock {\em {Extremum-Seeking Control and Applications}}.
\newblock Advances in Industrial Control. Springer, London, 2012.

\end{thebibliography}
\end{document}